\documentclass[a4paper,11pt]{amsart}
\usepackage{amsmath,amssymb,graphics,ytableau,color}
\usepackage[all]{xy}
\usepackage[top=30truemm,bottom=30truemm,left=25truemm,right=25truemm]{geometry}
\allowdisplaybreaks[1]

\newtheorem{Thm}{Theorem}[section]
\newtheorem{Cor}[Thm]{Corollary}
\newtheorem{Prop}[Thm]{Proposition}
\newtheorem{Lem}[Thm]{Lemma}
\newtheorem*{Thm*}{Theorem}

\theoremstyle{definition}
\newtheorem{Def}[Thm]{Definition}

\newtheorem{Ex}[Thm]{Example}

\theoremstyle{remark}

\makeatletter

\@addtoreset{equation}{section}
\makeatother

\DeclareMathOperator{\maks}{\mathsf{max}}
\DeclareMathOperator{\mini}{\mathsf{min}}
\DeclareMathOperator{\module}{\mathsf{mod}}

\DeclareMathOperator{\jirr}{\mathsf{j-irr}}
\DeclareMathOperator{\inv}{\mathsf{inv}}
\DeclareMathOperator{\cov}{\mathsf{cov}}
\DeclareMathOperator{\des}{\mathsf{des}}

\DeclareMathOperator{\sttilt}{\mathsf{s\tau-tilt}}
\DeclareMathOperator{\stitilt}{\mathsf{s\tau^{--1}-tilt}}

\DeclareMathOperator{\torf}{\mathsf{torf}}
\DeclareMathOperator{\add}{\mathsf{add}}
\DeclareMathOperator{\Fac}{\mathsf{Fac}}
\DeclareMathOperator{\Filt}{\mathsf{Filt}}
\DeclareMathOperator{\Sub}{\mathsf{Sub}}
\DeclareMathOperator{\itirigid}{\mathsf{i\tau^{--1}-rigid}}
\DeclareMathOperator{\brick}{\mathsf{brick}}

\DeclareMathOperator{\sbrick}{\mathsf{sbrick}}

\DeclareMathOperator{\Hom}{\mathsf{Hom}}
\DeclareMathOperator{\End}{\mathsf{End}}

\DeclareMathOperator{\soc}{\mathsf{soc}}

\DeclareMathOperator{\Ker}{\mathsf{Ker}}

\newcommand{\bbA}{\mathbb{A}}
\newcommand{\bbD}{\mathbb{D}}
\newcommand{\Z}{\mathbb{Z}}

\newcommand{\calC}{\mathcal{C}}
\newcommand{\calF}{\mathcal{F}}

\newcommand{\rmb}{\mathrm{b}}

\newcommand{\sfD}{\mathsf{D}}
\newcommand{\sfF}{\mathsf{F}}

\newcommand{\ang}[1]{\langle #1 \rangle}

\newcommand{\pmone}{\begin{smallmatrix}  1 \\ -1 \end{smallmatrix}}
\newcommand{\mpone}{\begin{smallmatrix} -1 \\  1 \end{smallmatrix}}

\newcommand{\xpone}{\begin{smallmatrix} \phantom{-1} \\  1 \end{smallmatrix}}
\newcommand{\xmone}{\begin{smallmatrix}  \phantom{1} \\ -1 \end{smallmatrix}}

\renewcommand{\max}{\maks}
\renewcommand{\min}{\mini}
\renewcommand{\mod}{\module}

\renewcommand{\Gamma}{\varGamma}
\renewcommand{\Delta}{\varDelta}
\renewcommand{\epsilon}{\varepsilon}
\renewcommand{\Lambda}{\varLambda}
\renewcommand{\Pi}{\varPi}
\renewcommand{\Omega}{\varOmega}

\newcommand{\rtwo}{\textcolor{red}{2}}
\newcommand{\rthr}{\textcolor{red}{3}}
\newcommand{\rfou}{\textcolor{red}{4}}
\newcommand{\rfiv}{\textcolor{red}{5}}
\newcommand{\rsix}{\textcolor{red}{6}}

\newcommand{\rmtwo}{\textcolor{red}{-2}}
\newcommand{\rmthr}{\textcolor{red}{-3}}
\newcommand{\rmfou}{\textcolor{red}{-4}}
\newcommand{\rmfiv}{\textcolor{red}{-5}}
\newcommand{\rmsix}{\textcolor{red}{-6}}

\newcommand{\btwo}{\textcolor{blue}{2}}
\newcommand{\bthr}{\textcolor{blue}{3}}
\newcommand{\bfou}{\textcolor{blue}{4}}
\newcommand{\bfiv}{\textcolor{blue}{5}}

\newcommand{\bmtwo}{\textcolor{blue}{-2}}
\newcommand{\bmthr}{\textcolor{blue}{-3}}
\newcommand{\bmfou}{\textcolor{blue}{-4}}
\newcommand{\bmfiv}{\textcolor{blue}{-5}}

\newcommand{\ptwo}{\textcolor{blue}{(2)}}
\newcommand{\pthr}{\textcolor{blue}{(3)}}
\newcommand{\pfou}{\textcolor{blue}{(4)}}
\newcommand{\pfiv}{\textcolor{blue}{(5)}}

\title[Bricks over preprojective algebras]{Bricks over preprojective algebras and \\
join-irreducible elements in Coxeter groups}
\date{\today}
\author{Sota Asai}
\address{Graduate School of Mathematics, Nagoya University, Furo-cho,
Chikusa-ku, Nagoya-shi, Aichi-ken, 464-8602, Japan}
\email{m14001v@math.nagoya-u.ac.jp}

\begin{document}

\ytableausetup{centertableaux}

\begin{abstract}
A (semi)brick over an algebra $A$ is a module $S$ such that 
the endomorphism ring $\End_A(S)$ is a (product of) division algebra.
For each Dynkin diagram $\Delta$, 
there is a bijection from the Coxeter group $W$ of type $\Delta$ 
to the set of semibricks over the preprojective algebra $\Pi$ of type $\Delta$, 
which is restricted to a bijection 
from the set of join-irreducible elements of $W$ to the set of bricks over $\Pi$.
This paper is devoted to giving an explicit description of these bijections 
in the case $\Delta=\bbA_n$ or $\bbD_n$. 
First, for each join-irreducible element $w \in W$, 
we describe the corresponding brick $S(w)$ in terms of 
``Young diagram-like'' notation.  
Next, we determine the canonical join representation $w=\bigvee_{i=1}^m w_i$
of an arbitrary element $w \in W$ based on Reading's work, 
and prove that $\bigoplus_{i=1}^n S(w_i)$ is the semibrick corresponding to $w$.
\end{abstract}

\maketitle

\tableofcontents

\setcounter{section}{-1}

\section{Introduction}

The representation theory of \textit{preprojective algebras} $\Pi$ of Dynkin type $\Delta$ 
has been developed by investigating their relationship 
with the \textit{Coxeter groups} $W=W(\Delta)$ associated to $\Delta$.
In particular, 
the ideal $I(w)$ of $\Pi$ associated to each element $w \in W$
introduced by \cite{IR,BIRS} plays an important role.
It is also useful to study cluster algebras.
For example, see \cite{AM,AIRT,BKT,GLS,Kimura,Mizuno2,ORT,Thomas}.

The Coxeter group $W$ is finite,
and there is \textit{the right weak order} $\le$ on $W$.
Then, the partially ordered set $(W,\le)$ is a \textit{lattice} \cite{BB}, 
that is, $W$ admits the two binary operations, 
the \textit{join} $x \vee y$ and the \textit{meet} $x \wedge y$ for any $x,y \in W$.

In our study, we efficiently use 
\textit{join-irreducible} elements in a lattice $L$.
We write $\jirr L$ for the set of join-irreducible elements in $L$.
Reading \cite{Reading} 
introduced the important notion of \textit{canonical join representations},
that is, ``minimum'' join-irreducible elements $u_1,u_2,\ldots,u_m \in \jirr L$ satisfying
$x=\bigvee_{i=1}^m u_i$ for a given element $x \in L$.

Any element in a Coxeter group of Dynkin type has a unique canonical join representation,
since the Coxeter group is a \textit{semidistributive} lattice,
see \cite{IRRT} for the detail. 
One of the aims of this paper is to show that 
the canonical join representations of the elements in the Coxeter group $W$
are strongly related to the representation theory of $\Pi$.
We will explain the detail later in this section.

Some of our results also hold in a more general setting.
Let $A$ be a finite-dimension algebra over a field $K$.
We write $\torf A$ for the set of \textit{torsion-free classes} 
in the category $\mod A$ of finite-dimensional $A$-modules. 
There is a natural partial order $\subset$ defined by inclusion relations,
and then, the partially ordered set $(\torf A, \subset)$ is also a lattice.

In the rest, we assume that $A$ is $\tau$-tilting finite,
that is, $\torf A$ is a finite set.
There are many bijections between $\torf A$ and many important objects 
in $\mod A$ or in its bounded derived category $\sfD^\rmb(\mod A)$
\cite{AIR,Asai,BY,KY,MS}.
In particular, 
we have a bijection $\sfF$ from the set $\sbrick A$ of \textit{semibricks} in $\mod A$
to $\torf A$, where $\sfF(S)$ is defined as the minimum torsion-free class
containing a semibrick $S$.
Here, a semibrick $S$ is defined as a module in $\mod A$
which admits a decomposition $S=\bigoplus_{i=1}^s S_i$
with $\End_A(S_i)$ a division $K$-algebra (that is, $S_i$ is a \textit{brick})
and with $\Hom_A(S_i,S_j)=0$ for $i \ne j$.
The sets $\torf A$ and $\sbrick A$ have bijections from
the set $\stitilt A$ of \textit{support $\tau^{-1}$-tilting $A$-modules}
satisfying the following commutative diagram \cite{AIR,Asai}:
\begin{align}\label{Eq_comm_intro}
\begin{xy}
(-40,  0) *+{\stitilt A} = "0",
(  0,  0) *+{\torf A} = "1",
( 40,  0) *+{\sbrick A} = "3",
(-40, -8) = "4",
( 40, -8) = "5",
\ar^{\Sub} "0"; "1"
\ar_{\sfF} "3"; "1"
\ar@{-} "0"; "4"
\ar@{-}_{M \mapsto \soc_{\End_A(M)} M} "4"; "5" 
\ar "5"; "3"
\end{xy}.
\end{align}
Moreover, the bijection $\sfF$ is restricted to a bijection from 
the set $\brick A$ of bricks in $\mod A$ to $\jirr (\torf A)$,
and we have the following commutative diagram of bijections:
\begin{align}\label{Eq_indec_comm_intro}
\begin{xy}
(-40,  0) *+{\itirigid A} = "0",
(  0,  0) *+{\jirr (\torf A)} = "1",
( 40,  0) *+{\brick A} = "3",
(-40, -8) = "4",
( 40, -8) = "5",
\ar^{\Sub} "0"; "1"
\ar_{\sfF} "3"; "1"
\ar@{-} "0"; "4"
\ar@{-}_{M \mapsto \soc_{\End_A(M)} M} "4"; "5" 
\ar "5"; "3"
\end{xy}.
\end{align}
Here, $\itirigid A$ denotes 
the set of indecomposable $\tau^{-1}$-rigid modules in $\mod A$.

As the first step, 
we will show that the canonical join representation of a torsion-free class
is given by the decomposition of the corresponding semibrick as a direct sum of bricks. 
This fact is independently obtained also in \cite{BCZ}.

\begin{Thm}[Theorem \ref{Thm_decompose}]\label{Thm_decompose_intro}
Let $\calF \in \torf A$,
take the unique semibrick $S \in \sbrick A$ satisfying $\calF=\sfF(S)$,
and decompose $S$ as $\bigoplus_{i=1}^m S_i$ with $S_i \in \brick A$.
Then the representation $\calF=\bigvee_{i=1}^m \sfF(S_i)$ is 
the canonical join representation.
\end{Thm}

For the preprojective algebra $\Pi$, 
Mizuno \cite{Mizuno} proved that the two lattices 
$(W,\le)$ and $(\torf \Pi,\subset)$ are isomorphic by the correspondence
$w \mapsto \Sub (\Pi/I(w))$
and that $\Pi/I(w)$ is a support $\tau^{-1}$-tilting $\Pi$-module.
Therefore, we obtain a bijection $S(\bullet) \colon W \to \sbrick \Pi$
given by $S(w):=\soc_{\End_\Pi(\Pi/I(w))}(\Pi/I(w))$.
The main aim of this paper is to describe the semibrick $S(w)$ for each element $w \in W$
as a quiver representation in the case $\Delta = \bbA_n$ or $\bbD_n$:
\begin{align*}
\bbA_n \colon & 
\begin{xy}
( 0,  0) *+{1} = "1",
(12,  0) *+{2} = "2",
(24,  0) *+{3} = "3",
(36,  0) *+{\cdots} = "4",
(48,  0) *+{n} = "5",
\ar@{-} "1";"2"
\ar@{-} "2";"3"
\ar@{-} "3";"4"
\ar@{-} "4";"5"
\end{xy}, & 
\bbD_n \colon &
\begin{xy}
( 0,  5) *+{1}  = "1",
( 0, -5) *+{-1} = "0",
(12,  0) *+{2}  = "2",
(24,  0) *+{3}  = "3",
(36,  0) *+{\cdots} = "4",
(52,  0) *+{n\!-\!1}= "5",
\ar@{-} "1";"2"
\ar@{-} "0";"2"
\ar@{-} "2";"3"
\ar@{-} "3";"4"
\ar@{-} "4";"5"
\end{xy}.
\end{align*}
If $\Delta=\bbA_n$, then $W$ is the symmetric group $\mathfrak{S}_{n+1}$,
and if $\Delta=\bbD_n$, then $W$ is the subgroup of the automorphism group 
on the set $\{\pm 1,\pm 2,\ldots,\pm n\}$ consisting of all elements $w$
such that $w(-i)=-w(i)$ holds for each $i$ and that $\#\{ i>0 \mid w(i)<0 \}$ is even.
Thus, we can express every $w \in W$ in the form $(w(1),w(2),\ldots,w(m))$,
and our description of the semibrick $S(w)$ is constructed by this expression.

Mizuno's isomorphism $W \to \torf \Pi$ of lattices is restricted to a bijection
$\jirr W \to \jirr(\torf \Pi)$ between the join-irreducible elements, so
we also obtain a bijection $S(\bullet) \colon \jirr W \to \brick \Pi$.
By \cite{IRRT} (types $\bbA_n$ and $\bbD_n$) and 
\cite{Demonet} (type $\mathbb{E}_n$, calculated by a computer program), 
the cardinality of each set is
\begin{align*}
\begin{cases}
2^{n+1}-n-2 & (\Delta=\bbA_n) \\
3^n-n \cdot 2^{n-1}-n-1 & (\Delta=\bbD_n) \\
1272 & (\Delta=\mathbb{E}_6) \\
17635 & (\Delta=\mathbb{E}_7) \\
881752 & (\Delta=\mathbb{E}_8)
\end{cases}.
\end{align*}
Moreover, we immediately obtain the following property from 
Theorem \ref{Thm_decompose_intro}.

\begin{Cor}[Corollary \ref{Cor_decompose_Coxeter}]
Let $w \in W$ and take $w_1,w_2,\ldots,w_m \in \jirr W$ 
such that $S(w)=\bigoplus_{i=1}^m S(w_i)$.
Then, $w=\bigvee_{i=1}^m w_i$ holds, 
and it is the canonical join representation of $w$ in $W$.
\end{Cor}

In this paper, we will give a description of the semibrick $S(w)$ 
by the following two steps:
\begin{itemize}
\item[(a)]
we find the canonical join representation $\bigvee_{i=1}^m w_i$ of $w$; and
\item[(b)]
we explicitly describe the brick $S(w_i)$ for each $w_i \in \jirr W$.
\end{itemize}
There is a combinatorial ``Young diagram-like''
description by Iyama--Reading--Reiten--Thomas \cite{IRRT} of 
$J(w):=(\Pi/I(w))e_l$ for $w \in \jirr W$ in the case $\Delta$ is $\bbA_n$ or $\bbD_n$,
where $l$ is the unique \textit{descent} of $w \in W$. 
In this setting, $S(w)=\soc_{\End_\Pi (J(w))}J(w)$ follows.

For example, let $w:=(2,5,8,1,3,4,6,7,9) \in W(\bbA_8)$ and
$w':=(6,9,-7,-4,1,2,3,5,8) \in W(\bbD_9)$.
Then, 
\begin{align*}
J(w)&=\begin{ytableau} 
3  &  2  &  1  \\
4  &  3  \\
5  &  4  \\
6  \\
7  \\
\end{ytableau}, &
S(w)&=\begin{ytableau} 
\none  &  2  &  1  \\
\none  &  3  \\
5      &  4  \\
6      \\
7      \\
\end{ytableau}, \\
J(w')&=\begin{ytableau} 
2      & \mpone & -2     & -3     & \bmfou & \bmfiv & -6     \\
3      & 2      & \pmone & \bmtwo & \bmthr \\
4      & 3      & \btwo  & \xpone \\
5      & \bfou  & \bthr  & 2      \\
6      & \bfiv  & 4      & 3\\
7      & 6      & 5      \\
8
\end{ytableau}, &
S(w')&=\begin{ytableau} 
\none  & \none  & \none  & \none  & \pfou  & \pfiv  & -6     \\
\none  & \none  & \xmone & \ptwo  & \pthr  \\
\none  & \none  & \ptwo  & \xpone \\
\none  & \pfou  & \pthr  & 2      \\
\none  & \pfiv  & 4      & 3\\
7      & 6      & 5      \\
8
\end{ytableau}
\end{align*}

Here, for each module $M$ above, 
each square $\begin{ytableau}i\end{ytableau}$ in the figure for $M$
denotes a one-dimensional subspace of $e_i M$ if $i \ge -1$;
and of $e_{|i|} M$ if $i \le -2$.
As $K$-vector spaces, 
$M$ is the direct sum of these one-dimensional subspaces.
In the figure for $S(w')$, for each $i=2,3,4,5$, 
the two squares $\begin{ytableau}\textcolor{blue}{(i)}\end{ytableau}$ 
together denote an element 
in the two-dimensional vector space corresponding to the two squares 
$\begin{ytableau}\textcolor{blue}{i}\end{ytableau}$ 
and $\begin{ytableau}\textcolor{blue}{-i}\end{ytableau}$ in the figure for $J(w')$.

Now, we will give a combinatorial description of the brick $S(w)$ for each $w \in W$.

For $\Delta=\bbA_n$, the following assertion holds.

\begin{Thm}[Theorem \ref{Thm_brick_A}, Corollary \ref{Cor_brick_abbr_A}]
Let $w \in W(\bbA_n)$ with its unique descent $l$.
Then, the brick $S(w)$ is given as follows.
\begin{itemize}
\item Set $R:=w([l+1,n+1])$, $a:=w(l)$, $b:=w(l+1)$, and $V:=[b,a-1]$.
\item The brick $S(w)$ has a $K$-basis $(\ang{i})_{i \in V}$,
where $\ang{i}$ belongs to $e_iS(w)$.
\item For each $i \in V$, place a symbol $i$ denoting 
the $K$-vector subspace $K\ang{i}$.
\item For each $i \in V \setminus \{\max V\}$,
we write exactly one arrow between $i$ and $i+1$,
where the orientation is $i \to i+1$ if $i+1 \in R$ and 
$i \leftarrow i+1$ if $i+1 \notin R$.
\end{itemize}
\end{Thm}

In this procedure, the brick $S(w)$ appearing in the example above is expressed as 
\begin{align*}
1 \leftarrow 2 \rightarrow 3 \rightarrow 4 \leftarrow 5 \rightarrow 6 \rightarrow 7.
\end{align*}

For $\Delta=\bbD_n$, the brick $S(w)$ is obtained from the following procedure.

\begin{Thm}[Theorem \ref{Thm_brick_D}, Corollary \ref{Cor_brick_abbr_D}]
Let $w \in W(\bbD_n)$ with its unique descent $l$.
Then, the brick $S(w)$ is given as follows.
\begin{itemize}
\item Set $R:=w([|l|+1,n])$, $a:=w(l)$, $b:=w(|l|+1)$, and 
\begin{align*}
& r:=\max \{ k \ge 0 \mid [1,k] \subset \pm R \}, \quad
c:=\begin{cases}
w^{-1}(|w(1)|)& (r \ge 1) \\
1 & (r=0)
\end{cases}, \\
& (V_-,V_+) :=
\begin{cases}
(\emptyset, [b,a-1]) & (b \ge 2) \\
(\emptyset, \{c\} \cup [2,a-1]) & (b = \pm 1) \\
([b+1,-2] \cup \{-c\}, \{c\} \cup [2,a-1]) & (b \le -2)
\end{cases}, \quad
V := V_+ \amalg V_-.
\end{align*}
\item The brick $S(w)$ has a $K$-basis $(\ang{i})_{i \in V}$,
where $\ang{i}$ belongs to $e_iS(w)$ if $i \ge -1$,
and $e_{|i|}S(w)$ if $i \le -2$.
\item For each $i \in V$, place a symbol $i$ denoting 
the $K$-vector subspace $K\ang{i}$.
\item We write the following arrows.
\begin{itemize}
\item[(i)]
For each $i \in V_+ \setminus \{\max V_+\}$,
draw an arrow $i \to |i|+1$ if $|i|+1 \in R$; and $i \leftarrow |i|+1$ otherwise.
\item[(ii)]
For each $i \in V_- \setminus \{\min V_-\}$,
draw an arrow $i \leftarrow -(|i|+1)$ if $-(|i|+1) \in R$; 
and $i \to -(|i|+1)$ otherwise.
\item[(iii)]
If $r \ge 1$,
for each $i \in V_-$ with $|i| \le r$,
draw an arrow $-i \leftarrow -(|i|+1)$ if $|i|+1 \in R$; 
and $i \to |i|+1$ otherwise.
\item[(iv)]
If $r=0$,
draw an arrow $-c \leftarrow 2$ if $c \leftarrow 2$ exists in \textup{(i)},  
and draw an arrow $c \to -2$ if $-c \to -2$ exists in \textup{(ii)}.
\end{itemize}
\end{itemize}
\end{Thm}

This theorem gives the following expression of the brick $S(w')$ in the example above 
obtained in the theorem:

\begin{align*}
\begin{xy}
(  0, 6) *+{-1} = "1-",
( 12, 6) *+{\textcolor{blue}{-2}} = "2-",
( 24, 6) *+{\textcolor{blue}{-3}} = "3-",
( 36, 6) *+{\textcolor{blue}{-4}} = "4-",
( 48, 6) *+{\textcolor{blue}{-5}} = "5-",
( 60, 6) *+{-6} = "6-",
(  0,-6) *+{ 1} = "1+",
( 12,-6) *+{ 2} = "2+",
( 24,-6) *+{ 3} = "3+",
( 36,-6) *+{ 4} = "4+",
( 48,-6) *+{ 5} = "5+",
( 60,-6) *+{ 6} = "6+",
( 72,-6) *+{ 7} = "7+",
( 84,-6) *+{ 8} = "8+",
\ar "1+";"2+" \ar "1-";"2-" \ar "2-";"1+"
\ar "2+";"3+" \ar "2-";"3-" \ar "3-";"2+"
\ar "4+";"3+" \ar "4-";"3-" \ar "3-";"4+"
\ar "4+";"5+" \ar "4-";"5-" \ar "5-";"4+"
\ar "6+";"5+" \ar "5-";"6-" \ar "5-";"6+"
\ar "7+";"6+"
\ar "7+";"8+"
\end{xy}.
\end{align*}
 
Then, the step (b) is done. 
Consequently, 
we obtain that the bricks over the preprojective algebra of type $\bbA_n$ is 
a module over some path algebra of type $\bbA_n$.
On the other hand, 
the preprojective algebra of type $\bbD_n$ does not have the corresponding property.

Finally, we consider an arbitrary element $w \in W$.
In Propositions \ref{Prop_decompose_A} and \ref{Prop_decompose_D},
we will explicitly determine 
the canonical join representation $\bigvee_{i=1}^m w_i$ of $w \in W$
by using the characterization of canonical join representations in Coxeter groups
given by Reading \cite{Reading}.
Then, in Theorems \ref{Thm_sbrick_A} and \ref{Thm_sbrick_D},
we explicitly write down the semibrick $S(w)=\bigoplus_{i=1}^m S(w_i)$
by using the description of bricks.
This is what we desire in this paper.

For example, let $\Delta:=\bbA_8$ and $w:=(4,9,3,6,2,8,5,1,7)$.
Then, its canonical join representation is $w_2 \vee w_4 \vee w_6 \vee w_7$, where
\begin{align*}
w_2 := (1,2,4,9,3,5,6,7,8), \quad
w_4 := (1,3,4,6,2,5,7,8,9), \\
w_6 := (1,2,3,4,6,8,5,7,9), \quad
w_7 := (2,3,4,5,1,6,7,8,9).
\end{align*}
Thus, the semibrick $S(w)$ is the direct sum of the following bricks:
\begin{align*}
S(w_2) &= \phantom{1 \rightarrow 2 \rightarrow {}} 
3 \leftarrow 4 \rightarrow 5 \rightarrow 6 \rightarrow 7 \rightarrow 8, \\
S(w_4) &= \phantom{1 \rightarrow {}}
2 \leftarrow 3 \leftarrow 4 \rightarrow 5
\phantom{{} \leftarrow 6 \leftarrow 7 \leftarrow 8}, \\
S(w_6) &= \phantom{1 \rightarrow 2 \rightarrow 3 \rightarrow 4 \rightarrow{}}
5 \leftarrow 6 \rightarrow 7
\phantom{{} \leftarrow 8}, \\
S(w_7) &= 1 \leftarrow 2 \leftarrow 3 \leftarrow 4
\phantom{{} \leftarrow 5 \leftarrow 6 \leftarrow 7 \leftarrow 8}. 
\end{align*}

\subsection{Notation}

The composition of two maps $f \colon X \to Y$ and $g \colon Y \to Z$ is denoted by $gf$.

We define the multiplication on the automorphism group on a finite set $X$ by
$(\sigma \tau)(i) := \sigma(\tau(i))$ for $i \in X$.
For $a,b \in X$, 
the notation $(a \quad b)$ means the transposition which exchanges $a$ and $b$.

For integers $a,b \in \Z$,
we define $[a,b]:=\{ i \in \Z \mid a \le i \le b \}$.
For a set $X \subset \Z$, we set 
$-X:=\{-i \mid i \in \Z \}$ and $\pm X:=X \cup (-X)$.

Throughout this paper, $K$ is a field and $A$ is a finite-dimensional $K$-algebra.
Unless otherwise stated, $A$-modules are finite-dimensional left $A$-modules, and
we write $\mod A$ for the category of finite-dimensional left $A$-modules.
Let $M \in \mod A$, 
and decompose $M$ as $M \cong \bigoplus_{i=1}^m M_i^{\oplus l_i}$
with $M_i \not \cong M_j$ for $i \ne j$ and with $l_i \ge 1$ for each $i$.
Then, we define the number $|M|:=m$,
and we say that $M$ is \textit{basic} if $l_i=1$ for any $i$.
We set the multiplication on
the endomorphism algebra $\End_A(M)$ as $g \cdot f := gf$.
Thus, $M$ is also a left $\End_A(M)$-module by $fx := f(x)$
for $f \in \End_A(M)$ and $x \in M$.

For a quiver $Q$, the composition of the two arrows
$\alpha \colon i \to j$ and $\beta \colon j \to k$ in $Q$ is denoted by $\alpha\beta$,
which is a path from $i$ to $k$.

\section{General observations of $\tau$-tilting finite algebras}\label{Sec_general}

In this section, we observe some general properties holding for
$\tau$-tilting finite algebras $A$ over a field $K$.

\subsection{Lattices}\label{Subsec_lattice}

First, we recall the notion of lattices.

\begin{Def}
Let $(L,\le)$ be a partially ordered set.
\begin{itemize}
\item[(1)]
For $x,y,z \in L$,
the element $z$ is called the \textit{meet} of $x$ and $y$
if $z$ is the maximum element satisfying $z \le x$ and $z \le y$.
In this case, $z$ is denoted by $x \wedge y$.
\item[(2)]
For $x,y,z \in L$,
the element $z$ is called the \textit{join} of $x$ and $y$
if $z$ is the minimum element satisfying $z \ge x$ and $z \ge y$.
In this case, $z$ is denoted by $x \vee y$.
\item[(3)]
The set $L$ is called a \textit{lattice}
if $L$ admits the meet $x \wedge y$ and the join $x \vee y$ for any $x,y \in L$.
\item[(4)]
The set $L$ is called a \textit{finite lattice}
if $L$ is a finite set and a lattice.
\end{itemize}
\end{Def}

The operations join and meet clearly satisfy the associative relations,
so we may use the expressions $x \wedge y \wedge z$ and $x \vee y \vee z$.
If $L \ne \emptyset$ is a finite lattice,
there exist the maximum element $\max L$ and the minimum element $\min L$.
In this case,
we define $\bigwedge_{x \in \emptyset} x := \max L$ and 
$\bigvee_{x \in \emptyset} x := \min L$.

Later in this paper, we will consider the decomposition of 
an element in a lattice with respect to the operation join,
so we recall the notion of join-irreducible elements.

\begin{Def}
Let $L$ be a lattice.
An element $x \in L$ is called a \textit{join-irreducible} element
if the following conditions hold:
\begin{itemize}
\item $x$ is not the minimum element of $L$; and 
\item for any $y,z \in L$, if $x=y \vee z$, then $y=x$ or $z=x$.
\end{itemize}
We write $\jirr L$ for the set of join-irreducible elements in $W$. 
\end{Def}

We remark that $x \in \jirr L$ is equivalent to that 
there exists a unique maximal element of the set $\{y \in W \mid y<x\}$
if $L$ is a finite lattice.
This fails if we drop the assumption that $L$ is finite \cite[Remark 3.1.2]{BCZ}.

\subsection{Torsion-free classes}\label{Subsec_tors_free}

Let $A$ be a finite-dimensional algebra.

A full subcategory $\calF$ of $\mod A$ is called a \textit{torsion-free class} in $\mod A$
if $\calF$ is closed under submodules and extensions,
and we write $\torf A$ for the set of torsion-free classes in $\mod A$.
For a full subcategory $\calC \subset \mod A$, we define
\begin{align*}
\add \calC &:= \{ M \in \mod A \mid \text{$M$ is a direct summand of
$\textstyle \bigoplus_{i=1}^s C_i$ for some $C_1,C_2,\ldots,C_s \in \calC$}\}, \\
\Filt \calC &:= \{ M \in \mod A \mid \text{there exists 
$0 = M_0 \subset M_1 \subset \cdots \subset M_l = M$ 
with $M_i/M_{i-1} \in \add \calC$}\}, \\
\Sub \calC &:= \{ M \in \mod A \mid 
\text{$M$ is a submodule of some object in $\add \calC$} \}, \\
\sfF(\calC) &:= \Filt(\Sub \calC).
\end{align*}
Then $\sfF(\calC)$ is the smallest torsion-free class containing $\calC$.

The set $\torf A$ has a natural partial order defined by inclusions,
and then, the partially ordered set $(\torf A,\subset)$ is a finite lattice
with $\calF_1 \wedge \calF_2 = \calF_1 \cap \calF_2$ and 
$\calF_1 \vee \calF_2 = \sfF(\calF_1 \cup \calF_2)$.
The notion of \textit{torsion classes} is dually defined.

A torsion-free class in $\mod A$ is 
not necessarily functorially finite in $\mod A$.
Demonet--Iyama--Jasso \cite{DIJ} introduced the notion of 
\textit{$\tau$-tilting finiteness}, which is equivalent to that $\torf A$ is a finite set.
In their paper, they proved that $A$ is $\tau$-tilting finite 
if and only if every torsion-free class is functorially finite.
In the rest, $A$ is assumed to be $\tau$-tilting finite.

Functorially finite torsion-free classes are strongly connected with
support $\tau^{-1}$-tilting $A$-modules, 
which were introduced by Adachi--Iyama--Reiten \cite{AIR}.
They proved that the set $\torf A$ has a bijection from the set $\stitilt A$ of
support $\tau^{-1}$-tilting $A$-modules.

Let $M \in \mod A$ and $I$ be an injective $A$-module in $\mod A$.
Then, $M$ is called a \textit{$\tau^{-1}$-rigid module} if $\Hom_A(\tau^{-1}M,M)=0$,
and the pair $(M,I)$ is called a \textit{$\tau^{-1}$-rigid pair}
if $M$ is $\tau^{-1}$-rigid and $\Hom_A(M,I)=0$.
If a $\tau^{-1}$-rigid pair $(M,I)$ satisfies $|M|+|I|=|A|$,
the pair $(M,I)$ is called a \textit{support $\tau^{-1}$-tilting pair}, and
an $A$-module $M$ is called a \textit{support $\tau^{-1}$-tilting module}
if there exists some injective module $I$ 
such that $(M,I)$ is a support $\tau$-tilting pair.
We write $\stitilt A$ for the set of basic support $\tau^{-1}$-tilting modules in $\mod A$.
The notion of \textit{support $\tau$-tilting modules} is dually defined.

If $M$ is $\tau^{-1}$-rigid, then the full subcategory
$\Sub M$ is a torsion-free class \cite{AS}.
Adachi--Iyama--Reiten proved that 
this correspondence $\stitilt A \ni M \mapsto \Sub M \in \torf A$ is a bijection.

\begin{Prop}\label{Prop_stitilt_torf}\cite[Theorem 2.7]{AIR}
The correspondence $\stitilt A \ni M \mapsto \Sub M \in \torf A$ is a bijection.
\end{Prop}

Thus, we induce a partial order $\le$ on the set $\stitilt A$ 
from inclusion relations on $\torf A$;
namely, $M \le N$ holds if and only if $\Sub M \subset \Sub N$.
Then, $(\stitilt A, \le)$ is clearly a lattice. 

\subsection{Semibricks}\label{Subsec_gen_semibrick}

We assume that $A$ is $\tau$-tilting finite as in the previous subsection.

The aim of this paper is to investigate bricks and semibricks over preprojective algebras.
These notions are defined as follows.

\begin{Def}
Let $S$ be an $A$-module.
\begin{itemize}
\item[(1)] 
The module $S$ is called a \textit{brick} if
the endomorphism ring $\End_A(S)$ is a division ring.
We write $\brick A$ for the set of bricks.
\item[(2)] 
The module $S$ is called a \textit{semibrick} if
$S$ is decomposed as the direct sum $\bigoplus_{i=1}^m S_i$ of bricks
$S_1,S_2,\ldots,S_m \in \brick A$ satisfying $\Hom_A(S_i,S_j)=0$ if $i \ne j$.
We write $\sbrick A$ for the set of semibricks in $\mod A$.
\end{itemize}
\end{Def}

The notion of semibricks is 
originally defined as sets of Hom-orthogonal bricks in \cite{Asai},
but it does not matter here, since $A$ is assumed to be $\tau$-tilting finite
\cite[Corollary 1.10]{Asai}.
Then, \cite[Proposition 1.9]{Asai} tells us that 
there is a bijection $\sfF \colon \sbrick A \to \torf A$
taking the minimum torsion-free class $\sfF(S)$ containing each semibrick $S$.
Moreover, it satisfies the property below. 

\begin{Prop}\cite[Proposition 1.9]{Asai}\label{Prop_comm}
We have the following commutative diagram of bijections:
\begin{align*}
\begin{xy}
(-40,  0) *+{\stitilt A} = "0",
(  0,  0) *+{\torf A} = "1",
( 40,  0) *+{\sbrick A} = "3",
(-40, -8) = "4",
( 40, -8) = "5",
\ar^{\Sub} "0"; "1"
\ar_{\sfF} "3"; "1"
\ar@{-} "0"; "4"
\ar@{-}_{M \mapsto \soc_{\End_A(M)} M} "4"; "5" 
\ar "5"; "3"
\end{xy}.
\end{align*}
\end{Prop}

Now, we set $\itirigid A$ as the set of indecomposable $\tau^{-1}$-rigid $A$-modules
in $\mod A$.
Then, we also have another commutative diagram.

\begin{Prop}\label{Prop_indec_comm}
We have the following commutative diagram of bijections:
\begin{align*}
\begin{xy}
(-40,  0) *+{\itirigid A} = "0",
(  0,  0) *+{\jirr (\torf A)} = "1",
( 40,  0) *+{\brick A} = "3",
(-40, -8) = "4",
( 40, -8) = "5",
\ar^{\Sub} "0"; "1"
\ar_{\sfF} "3"; "1"
\ar@{-} "0"; "4"
\ar@{-}_{M \mapsto \soc_{\End_A(M)} M} "4"; "5" 
\ar "5"; "3"
\end{xy}.
\end{align*}
\end{Prop}
	
\begin{proof}
For any $\calF \in \torf A$, the join-irreducibility of $\calF$ is equivalent to that 
there exists a unique maximal element in the set 
$\{\calF' \in \torf A \mid \calF' \subsetneq \calF\}$, since $\torf A$ is a finite lattice.
From \cite[Example 3.5]{DIJ}, the latter condition holds if and only if 
the corresponding $M \in \stitilt A$ has exactly one mutation $M'$ satisfying
$\Sub M' \subsetneq \Sub M$.
We here write $\stitilt_1 A$ for the set of such $M \in \stitilt A$.
Then, we have a bijection $\Sub \colon \stitilt_1 A \to \jirr(\torf A)$.

For any $M \in \stitilt A$,
the condition $M \in \stitilt_1 A$ is equivalent to that
there uniquely exists an indecomposable direct summand $M_1$ of $M$
such that $\Sub M_1=\Sub M$ \cite[Definition-Proposition 2.28]{AIR}.
The correspondence $\stitilt_1 A \ni M \mapsto M_1 \in \itirigid A$ is a bijection
by Proposition \ref{Prop_stitilt_torf}.
Thus, we have a bijection
$\itirigid A \leftarrow \stitilt_1 A \xrightarrow{\Sub} \jirr(\torf A)$,
and it coincides with $\Sub \colon \itirigid A \to \jirr(\torf A)$.

By using \cite[Proposition 1.13]{Asai} again, for any $M \in \stitilt A$,
the condition $M \in \stitilt_1 A$ is also equivalent to that
the corresponding semibrick $S$ is actually a brick.
Thus, the bijection $\stitilt A \to \sbrick A$ is restricted to 
a bijection $\stitilt_1 A \to \brick A$.
By composing it to the bijection $\stitilt_1 A \ni M \mapsto M_1 \in \itirigid A$,
we have a bijection $\itirigid A \to \brick A$,
and it is also given by the formula $M \mapsto \soc_{\End_A(M)} M$
by \cite[Theorem 1.3]{Asai}.

From these observations and Proposition \ref{Prop_comm}, 
we can obtain the desired commutative diagram of bijections.
\end{proof}

\subsection{Canonical join representations}\label{Subsec_gen_can_join}

Now that the bijection $\sfF \colon \sbrick A \to \torf A$ 
is restricted to a bijection $\brick A \to \jirr (\torf A)$,
the following natural question occurs:
\begin{quote}
Let $\calF \in \torf A$,
take the unique semibrick $S \in \sbrick A$ satisfying $\calF=\sfF(S)$,
and decompose $S$ as $\bigoplus_{i=1}^m S_i$ with $S_i \in \brick A$.
Then what is the relationship between $\sfF(S) \in \torf A$ and 
$\sfF(S_1),\sfF(S_2),\ldots,\sfF(S_m) \in \jirr(\torf A)$?
\end{quote}

Clearly, $\sfF(S)=\bigvee_{i=1}^m \sfF(S_i)$ holds,
since $\sfF(S)$ is the minimum torsion-free class containing all $\sfF(S_i)$.
Actually, this will turn out to be a canonical join representation.
Here, the notion of canonical join representations 
was introduced by Reading \cite{Reading},
and defined as below.

\begin{Def}\label{Def_can_join}
Let $L$ be a finite lattice, $x \in L$, and $U \subset L$.
Then, we say that $U$ is a \textit{canonical join representation} if
\begin{itemize}
\item[(a)] $x=\bigvee_{u \in U} u$ holds; and
\item[(b)] for any proper subset $U' \subsetneq U$,
the join $\bigvee_{u \in U'} u$ never coincides with $x$; and
\item[(c)] if $V \subset L$ satisfies the properties (a) and (b), then, 
for every $u \in U$, there exists $v \in V$ such that $u \le v$.
\end{itemize}
In this case, we also say $x=\bigvee_{u \in U} u$ is a canonical join representation.
\end{Def}

If $x \in L$ has a canonical join representation $U$, 
then we can easily check that 
it is the unique canonical join representation for each $x \in L$,
and that $U$ is a subset of $\jirr L$.
The existence of canonical join representations is not guaranteed for
general finite lattices.
In the case that $L=\torf A$, 
every $\calF \in \torf A$ has a canonical join representation
given by the indecomposable decomposition of semibricks.

\begin{Thm}\label{Thm_decompose}
Let $\calF \in \torf A$,
take the unique semibrick $S \in \sbrick A$ satisfying $\calF=\sfF(S)$,
and decompose $S$ as $\bigoplus_{i=1}^m S_i$ with $S_i \in \brick A$.
Then the representation $\calF=\bigvee_{i=1}^m \sfF(S_i)$ is 
the canonical join representation.
\end{Thm}

\begin{proof}
We have seen the property (a): $\calF=\sfF(S)=\bigvee_{i=1}^m \sfF(S_i)$.

We show the property (b). 
Let $I$ be a proper subset of $[1,m]$.
Take $j \in [1,m] \setminus I$.
Then, the brick $S_j$ cannot belong to 
$\sfF(\{S_i\}_{i \in I})=\sfF(\bigcup_{i \in I} \sfF(S_i))=\bigvee_{i \in I}\sfF(S_i)$, 
since $\Hom_A(S_j,S_i)=0$ holds for each $i \in I$.
This implies that $\sfF(S) \ne \bigvee_{i \in I}\sfF(S_i)$.

Next, we show the property (c).
Let $\calF_1,\ldots,\calF_{m'} \in \torf A$ satisfy 
$\calF=\bigvee_{j=1}^{m'} \calF_j$ and the property (a).
For each $i \in [1,m]$, the brick $S_i$ belongs to $\sfF(S)=\calF$,
which coincides with $\bigvee_{j=1}^{m'} \calF_j=\sfF(\bigcup_{j=1}^{m'} \calF_j)$.
Thus, there must exist some $j \in [1,m']$ such that 
$\Hom_A(S_i,\calF_j) \ne 0$.
We take a semibrick $S'$ such that $\calF_j=\sfF(S')$,
then there exists a nonzero homomorphism $f \colon S_i \to S'$.
By \cite[Lemma 1.7]{Asai}, $f$ is injective,
since $S_i,S' \in \calF=\sfF(S)$ and $S_i$ is a direct summand of $S$.
This implies that $\sfF(S_i) \subset \calF_j$.
\end{proof}

In particular, 
the partially ordered set $\torf A$ admits a canonical join representation
for any $\calF \in \torf A$.

The notion of canonical join representations is defined in a fully combinatorial way,
but decomposing semibricks into direct sums of bricks is 
a purely representation-theoritic problem.
These two are related by Theorem \ref{Thm_decompose}.

The relationship between semibricks and torsion classes are independently discussed by 
Barnard--Carroll--Zhu \cite{BCZ} and Demonet--Iyama--Reiten--Reading--Thomas \cite{DIRRT}
in the setting that the algebra $A$ is not necessarily $\tau$-tilting finite.
In particular, our Theorem \ref{Thm_decompose} is
generalized in \cite[Proposition 3.2.5]{BCZ}.

\section{Preliminaries for preprojective algebras}\label{Sec_Pre}

In this section, we recall some properties on Coxeter groups and 
preprojective algebras of Dynkin type.

\subsection{Coxeter groups}\label{Subsec_Coxeter}

Coxeter groups of Dynkin type are strongly 
related to the corresponding preprojective algebras.
In this subsection, we state the definition of Coxeter groups of Dynkin type,
and prepare some basic terms on combinatorics of Coxeter groups.
For more information, see \cite{BB}.

Let $\Delta$ be a Dynkin diagram whose vertices set is $\Delta_0$.
Then, the \textit{Coxeter group} $W$ for $\Delta$ is the group
defined by the generators $\{s_i \mid i \in \Delta_0\}$ and the relations
\begin{itemize}
\item $s_i^2=1$ for each $i$;
\item $s_is_j=s_js_i$ if there is no edge between $i$ and $j$ in $\Delta$; and
\item $s_is_js_i=s_js_is_j$ if there is exactly one edge between $i$ and $j$ in $\Delta$.
\end{itemize}
It is well-known that the Coxeter group $W$ associated to a Dynkin diagram 
$\Delta$ is a finite group.

Each element $w \in W$ has the minimum number $l$
such that $w$ can be written as a product $s_{i_1} s_{i_2} \cdots s_{i_l}$
of $l$ generators.
Such number is called the \textit{length} of $w$, and is denoted by $l(w)$.
If $l=l(w)$ and $w=s_{i_1} s_{i_2} \cdots s_{i_l}$,
then $s_{i_1} s_{i_2} \cdots s_{i_l}$
is called a \textit{reduced expression} of $w$, which is not necessarily unique. 

If an element $w \in W$ has the maximum length among the elements of $W$,
then $w$ is called a \textit{longest element} of $W$.
Actually, such an element uniquely exists, and it is often denoted by $w_0$.

We can consider several partial orders on the Coxeter group $W$,
but in this paper, we only use the \textit{right weak order}:
for $w,w' \in W$, the inequality $w \le w'$ holds if and only if $l(w')=l(w)+l(w^{-1}w')$.
Then, the poset $(W,\le)$ is a lattice.

We write $\jirr W$ for the set of join-irreducible elements 
of the partially ordered set $(W,\le)$.
For $w \in W$, the minimal elements of the set $\{w' \in W \mid w'<w \}$ 
are $ws_i$ for all $i \in \Delta_0$ satisfying $l(w)>l(ws_i)$.
Therefore, $w \in W$ is join-irreducible if and only if 
there uniquely exists $i \in \Delta_0$ such that $l(w)>l(ws_i)$.
In this case, we say that $w$ is a join-irreducible element of \textit{type} $i$. 

When we consider the right weak order of the Coxeter group, 
the notion of \textit{inversions} is useful.
We call an element $t \in W$ a \textit{reflection} of $W$
if there exist some $w \in W$ and $i \in \Delta_0$ satisfying $t=ws_iw^{-1}$.
Fix $w \in W$, then a reflection $t$ of $W$ is called an \textit{inversion}
if $l(tw) < l(w)$, and the set of inversions of $w$ is denoted by $\inv(w)$.
It is well-known that, for two elements $w,w' \in W$, the inequality $w \le w'$ holds 
if and only if $\inv(w) \subset \inv(w')$.

\subsection{Bijections}\label{Subsec_bij}

Now that the preparation on Coxeter groups of Dynkin type is done,
let us see how they are related to the corresponding preprojective algebras.

We quickly recall the definition of preprojective algebras of Dynkin type.
Let $\Delta$ be a Dynkin diagram.
We define the double quiver $Q$ for $\Delta$, 
that is, the set $Q_0$ of vertices of $Q$ is $\Delta_0$,
and the set $Q_1$ of arrows of $Q$ consists of  
$i \to j$ and $j \to i$ for each edge between $i$ and $j$ of $\Delta$.
For each arrow $\alpha \colon i \to j$ in $Q_1$,
we write $\alpha^*$ for the reversed arrow $j \to i$.
There is a subset $Q'_1 \subset Q_1$ such that, for each $\alpha \in Q_1$,
the condition $\alpha \in Q'_1$ holds if and only if $\alpha^* \notin Q'_1$.
Then, the preprojective algebra $\Pi$ corresponding to $\Delta$ is given by
$KQ/\langle \sum_{\alpha \in Q'_1} (\alpha \alpha^* - \alpha^* \alpha) \rangle$.
Here, the choice of the subset $Q'_1$ is not unique in general, but 
$\Pi$ is uniquely defined up to isomorphisms, since $\Delta$ is Dynkin.
For each vertex $i \in Q_0$, we write $e_i$ for the idempotent of $\Pi$ 
corresponding to the vertex $i$.

Let $\Pi$ be the preprojective algebra of Dynkin type $\Delta$,
and set $I_i:=\Pi(1-e_i)\Pi$, which is a maximal ideal of $\Pi$.
We write $\langle I_i \mid i \in \Delta_0 \rangle$ for the set of ideals of the form 
$I_{i_1} I_{i_2} \cdots I_{i_k}$.

There is an important ideal $I(w)$ of $\Pi$
associated to each element $w$ of the Coxeter group $W$ for $\Delta$.
The ideal $I(w)$ is defined as follows:
take a reduced expression of $w=s_{i_1} s_{i_2} \cdots s_{i_k}$
and set $I(w):=I_{i_1} I_{i_2} \cdots I_{i_k}$.
Clearly, $I(w)$ belongs to the set $\langle I_i \mid i \in \Delta_0 \rangle$.

By \cite[Theorem 2.14]{Mizuno},
$I(w)$ does not depend on the choice of a reduced expression of $w$,
and the well-defined correspondence $w \mapsto I(w)$ gives
a bijection $W \to \langle I_i \mid i \in \Delta_0 \rangle$.
We remark that a similar bijection exists for a preprojective algebra of 
non-Dynkin type, see \cite[Theorem III.1.9]{BIRS}.

Moreover, Mizuno proved 
the set $\langle I_i \mid i \in \Delta_0 \rangle$ coincides with 
the set $\sttilt \Pi$ of support $\tau$-tilting $\Pi$-modules.
He also proved that the bijection $W \ni w \mapsto I(w) \in \sttilt \Pi$ is
an isomorphism $(W,\le) \to (\sttilt \Pi, \ge)$ of lattices
\cite[Theorem 2.30]{Mizuno}.

In our convention, we need the dual version of this isomorphism. 
The torsion-free class corresponding to the torsion class $\Fac I(w)$ is
$\Sub (\Pi/I(w))$,
and it follows from Mizuno's isomorphism and \cite[Proposition 6.4]{ORT} that 
the module $\Pi/I(w)$ is a support $\tau^{-1}$-tilting module.
Thus, we obtain the following isomorphism of lattices.

\begin{Prop}\label{Prop_W_stitilt}
There exists an isomorphism 
$(W,\le) \to (\stitilt \Pi,\le)$ of lattices given by $w \mapsto \Pi/I(w)$.
\end{Prop}

In this map, the longest element $w_0 \in W$ corresponds to 
the injective cogenerator $\Pi$, 
and the identity element $\mathrm{id}_W$ corresponds to 0.

Since the Coxeter group $W$ for the Dynkin diagram $\Delta$ is a finite group, 
$\Pi$ is $\tau$-tilting finite.
Therefore, we obtain the following bijections 
from Propositions \ref{Prop_comm}, \ref{Prop_indec_comm}, and \ref{Prop_W_stitilt}.

\begin{Prop}\label{Prop_W_sbrick}
There exists a bijection $S(\bullet) \colon W \to \sbrick \Pi$ defined by the formula 
$S(w):=\soc_{\End(\Pi/I(w))} (\Pi/I(w))$.
As a restriction, we have another bijection 
$S(\bullet) \colon \jirr W \to \brick \Pi$.
\end{Prop}

The aim of this paper is to describe the semibrick $S(w)$ for each $w \in W$ explicitly.

Since the partially ordered sets $(W,\le)$ and $(\torf A, \subset)$ are isomorphic, 
we obtain the following property immediately from Theorem \ref{Thm_decompose}.

\begin{Cor}\label{Cor_decompose_Coxeter}
Let $w \in W$ and take $w_1,w_2,\ldots,w_m \in \jirr W$ 
such that $S(w)=\bigoplus_{i=1}^m S(w_i)$.
Then, $w=\bigvee_{i=1}^m w_i$ holds, 
and it is the canonical join representation of $w$ in $W$.
\end{Cor}

We will explicitly determine the canonical join representation for each $w \in W$
in Section \ref{Sec_semibrick}.
It is a purely combinatorial problem. 

Then, the remained task is 
to describe the brick $S(w)$ for each join-irreducible element $w \in \jirr W$.
For this purpose, 
we use the following bijection by Iyama--Reading--Reiten--Thomas \cite{IRRT}.

\begin{Prop}\label{Prop_jirrW_itirigid}\cite[Theorem 4.1]{IRRT}
For each $w \in \jirr W$ of type $l$,
we set a module $J(w):=(\Pi/I(w))e_l$, which is a direct summand of $\Pi/I(w)$.
Then $\Sub J(w)=\Sub (\Pi/I(w))$ holds, 
and this induces a bijection $J(\bullet) \colon \jirr W \to \itirigid \Pi$.
\end{Prop}

Thus, by Proposition \ref{Prop_indec_comm}, we obtain the following formula.

\begin{Prop}\label{Prop_jirrW_brick}
Let $w \in \jirr W$ be of type $l$, and set $J(w):=(\Pi/I(w))e_l$.
Then, the brick $S(w)$ is equal to $\soc_{\End_\Pi(J(w))}J(w)$.
\end{Prop}

Moreover, they have already given 
a combinatorial description of $J(w)$ for $\Delta=\bbA_n,\bbD_n$.
This will be cited in the following subsections.
By using this and Proposition \ref{Prop_indec_comm}, 
we will write down the explicit structure of the brick $S(w)$
for each $w \in \jirr W$ in Section \ref{Sec_brick}.

Now, we have recalled some properties 
holding for any preprojective algebra of Dynkin type.
In the next two subsections,
we will observe the preprojective algebras of type $\bbA_n$ and $\bbD_n$ in detail.

\subsection{Type $\bbA_n$}

Let $\Delta:=\bbA_n$ in this subsection.
The preprojective algebra $\Pi$ of type $\bbA_n$ is given by the following quiver
and relations:
\begin{align*}
& \begin{xy}
( 0,  0) *+{1} = "1",
(16,  0) *+{2} = "2",
(32,  0) *+{3} = "3",
(48,  0) *+{\cdots} = "4",
(64,  0) *+{n} = "5",
\ar^{\alpha_1} @<1mm> "1";"2"
\ar^{\beta_2} @<1mm> "2";"1"
\ar^{\alpha_2} @<1mm> "2";"3"
\ar^{\beta_3} @<1mm> "3";"2"
\ar^{\alpha_3} @<1mm> "3";"4"
\ar^{\beta_4} @<1mm> "4";"3"
\ar^{\alpha_{n-1}} @<1mm> "4";"5"
\ar^{\beta_n} @<1mm> "5";"4"
\end{xy}; \\
& \alpha_1\beta_2=0, \quad \alpha_i\beta_{i+1}=\beta_i\alpha_{i-1} \ (2 \le i \le n-1), \quad \beta_{n}\alpha_{n-1}=0.
\end{align*}

The Coxeter group $W$ of type $\bbA_n$ is isomorphic to 
the symmetric group $\mathfrak{S}_{n+1}$
by sending each $s_i$ to the transposition $(i \quad i+1)$.
We identify the Coxeter group with $\mathfrak{S}_{n+1}$ by this isomorphism,
and we express $w \in W$ as $(w(1),w(2),\ldots,w(n+1))$.

The reflections of $W$ are precisely 
the transpositions $(a \quad b)$ with $a,b \in [1,n+1]$ and $a>b$,
and the set $\inv(w)$ of inversions of $w \in W$ is 
\begin{align*}
\{ (a \quad b) \mid a,b \in [1,n+1],\ a>b, \ w^{-1}(a) < w^{-1}(b) \}.
\end{align*}

An element $w \in W$ is a join-irreducible element of type $l$
if and only if $l$ is the unique element in $[1,n]$ satisfying $w(l)>w(l+1)$.
In this case, we have $w(l) \ge 2$.

We set a basis of each indecomposable projective module $\Pi e_l$ as follows.
Let $i,j,l \in Q_0=[1,n]$ with $i \le j \ge l$.
We define a path $p(i,j,l)$ in $Q$ as 
\begin{align*}
p(i,j,l):=(\alpha_i \alpha_{i+1} \cdots \alpha_{j-1}) \cdot
(\beta_j \beta_{j-1} \cdots \beta_{l+1}).
\end{align*}
This is the shortest path starting from $i$, going through $j$, and ending at $l$.
As an element in $\Pi$, 
the path $p(i,j,l)$ is not zero in $\Pi$ if and only if $i \ge j-l+1$,
so set 
\begin{align*}
\Gamma[l]:=\{ (i,j) \in Q_0 \times Q_0 \mid j-l+1 \le i \le j \ge l \}.
\end{align*}

We obtain the following assertion from straightforward calculation.

\begin{Lem}\label{Lem_basis_A}
The set $\{p(i,j,l) \mid (i,j) \in \Gamma[l] \}$ forms a $K$-basis of $\Pi e_l$.
\end{Lem}

This basis allows us to express $\Pi e_l$ as
\begin{align}\label{eq_pattern_A}
\begin{xy}
( 0, 15) *+{l}      = "11",
(15, 15) *+{l-1}    = "12",
(30, 15) *+{\cdots} = "13",
(45, 15) *+{1}      = "14",
( 0,  5) *+{l+1}    = "21",
(15,  5) *+{l}      = "22",
(30,  5) *+{\cdots} = "23",
(45,  5) *+{2}      = "24",
( 0, -5) *+{\vdots} = "31",
(15, -5) *+{\vdots} = "32",
(30, -5) *+{}       = "33",
(45, -5) *+{\vdots} = "34",
( 0,-15) *+{n}      = "41",
(15,-15) *+{n-1}    = "42",
(30,-15) *+{\cdots} = "43",
(45,-15) *+{n-l+1}  = "44",
\ar "11";"12" \ar "12";"13" \ar "13";"14"
\ar "11";"21" \ar "12";"22" \ar "14";"24"
\ar "21";"22" \ar "22";"23" \ar "23";"24"
\ar "21";"31" \ar "22";"32" \ar "24";"34"
\ar "31";"41" \ar "32";"42" \ar "34";"44"
\ar "41";"42" \ar "42";"43" \ar "43";"44"
\end{xy}.
\end{align}
Here, each number $i$ in the row starting at $j$ 
denotes a one-dimensional vector space $K p(i,j,l)$ with a basis $p(i,j,l)$,
and each arrow stands for the identity map $K \to K$
with respect to these bases.

In examples later, we sometimes write $\Pi e_l$ like a Young diagram
by enclosing each entry with a square and omitting arrows:
for example, if $n=8$ and $l=3$, then $\Pi e_l$ is denoted by
\begin{align}\label{eq_Young_A}
\begin{ytableau} 
3      & 2      & 1      \\
4      & 3      & 2      \\
5      & 4      & 3      \\
6      & 5      & 4      \\
7      & 6      & 5      \\
8      & 7      & 6     
\end{ytableau}.
\end{align}
We use similar notation for subfactor modules of $\Pi e_l$. 

Under this preparation, we recall the result of \cite{IRRT} for type $\bbA_n$.

\begin{Prop}\label{Prop_rigid_A}\cite[Theorem 6.1]{IRRT}
Let $w \in \jirr W$ be a join-irreducible element of type $l$.
Then the module $J(w) \in \itirigid \Pi$ is expressed as follows.
\begin{itemize}
\item Consider the diagram \textup{(\ref{eq_pattern_A})}.
\item For each $j \in [l,n]$, in the row starting at $j$,
keep the entries $i$ satisfying $i \ge w(j+1)$ and delete the others.
\end{itemize}
\end{Prop}

\subsection{Type $\bbD_n$}

Let $\Delta:=\bbD_n$ in this subsection.
The preprojective algebra $\Pi$ of type $\bbD_n$ is given by the following quiver
and relations:
\begin{align*}
& \begin{xy}
( 0, 12) *+{1}  = "1",
( 0,-12) *+{-1} = "0",
(16,  0) *+{2}  = "2",
(32,  0) *+{3}  = "3",
(48,  0) *+{\cdots} = "4",
(64,  0) *+{n-1}= "5",
\ar^{\alpha_1^+} @<1mm> "1";"2"
\ar^{\beta_2^+} @<1mm> "2";"1"
\ar^{\alpha_1^-} @<1mm> "0";"2"
\ar^{\beta_2^-} @<1mm> "2";"0"
\ar^{\alpha_2} @<1mm> "2";"3"
\ar^{\beta_3} @<1mm> "3";"2"
\ar^{\alpha_3} @<1mm> "3";"4"
\ar^{\beta_4} @<1mm> "4";"3"
\ar^{\alpha_{n-2}} @<1mm> "4";"5"
\ar^{\beta_{n-1}} @<1mm> "5";"4"
\end{xy}; \\
& \alpha_1^+\beta_2^+=0, \quad \alpha_1^-\beta_2^-=0, \quad 
\alpha_2\beta_3=\beta_2^+\alpha_1^+ + \beta_2^-\alpha_1^-, \\
& \alpha_i\beta_{i+1}=\beta_i\alpha_{i-1} \ (3 \le i \le n-2), \quad \beta_{n-1}\alpha_{n-2}=0.
\end{align*}
To avoid complicated notation, we set
$\alpha_1:=\alpha_1^+ + \alpha_1^-$ and $\beta_2:=\beta_2^+ + \beta_2^-$.

The Coxeter group $W$ of type $\bbD_n$
is isomorphic to the group consisting of all automorphisms $w$
on the set $\pm [1,n]$ satisfying the following conditions:
\begin{itemize}
\item $w(-i)=-w(i)$ holds for each $i \in [1,n]$; and
\item the number of elements in $\{ i \in [1,n] \mid w(i)<0 \}$ is even.
\end{itemize}
Here, $s_i \in W$ is sent to
$(-1 \quad 2)(-2 \quad 1)$ if $i=-1$; and 
$(-i \quad {-(i+1)})(i \quad i+1)$ if $i \ne -1$.
We identify $W$ with the group above by this isomorphism.
Since $w(-i)=-w(i)$ holds, we express $w \in W$ as $(w(1),w(2),\ldots,w(n))$.

The reflections of $W$ are precisely 
the elements of the form
$({-a} \quad {-b})(a \quad b)$ with $a,b \in \pm [1,n]$ and $a>|b|$,
and the set $\inv(w)$ of inversions of $w \in W$ is 
\begin{align*}
\{ ({-a} \quad {-b})(a \quad b) \mid 
a,b \in \pm [1,n], \ a>|b|, \ w^{-1}(a)<w^{-1}(b) \}.
\end{align*}

An element $w \in W$ is a join-irreducible element of type $l$
if and only if $l$ is the unique element in $\{-1\} \cup [1,n-1]=Q_0$ 
such that $w(l)>w(|l|+1)$ holds.

We set two bases of each indecomposable projective module $\Pi e_l$ as follows.
We divide the argument by whether $l = \pm 1$ or not. 

We consider the case $l = \pm 1$ first.
Let $i,j \in Q_0=\{-1\} \cup [1,n-1]$ with $i \le j \ne -l$.
We define a path $p(i,j,l)$ by
\begin{align*}
p(i,j,\pm 1):=\begin{cases}
(\alpha_i \alpha_{i+1} \cdots \alpha_{j-1}) \cdot
(\beta_j \beta_{j-1} \cdots \beta_3) \beta_2^\pm & (i \ge 2) \\
\alpha_1^+ p(2,j,\pm1) & (i=1) \\
\alpha_1^- p(2,j,\pm1) & (i=-1)
\end{cases}.
\end{align*}
This is a shortest path starting from $i$, going through $j$, and ending at $l$.
As an element in $\Pi$, 
the path $p(i,j,l)$ is not zero in $\Pi$ if and only if $i \ne (-1)^j l$,
so set 
\begin{align*}
\Gamma[l]:=\{ (i,j) \in Q_0 \times Q_0 \mid (-1)^j l \ne i \le j \ne -l \}.
\end{align*}

We obtain the following assertion from straightforward calculation.

\begin{Lem}\label{Lem_basis_D_1}
The set $\{p(i,j,l) \mid (i,j) \in \Gamma[l] \}$ 
forms a $K$-basis of $\Pi e_l$.
\end{Lem}

This basis allows us to express $\Pi e_l$ as
\begin{align}\label{eq_pattern_D_1}
\begin{xy}
( 0, 20) *+{l}           = "11",
( 0, 10) *+{2}           = "21",
(15, 10) *+{-l}          = "22",
( 0,  0) *+{\vdots}      = "31",
(15,  0) *+{\vdots}      = "32",
( 0,-10) *+{n-2}         = "41",
(15,-10) *+{n-3}         = "42",
(30,-10) *+{\cdots}      = "43",
(45,-10) *+{(-1)^{n-3}l} = "44",
( 0,-20) *+{n-1}         = "51",
(15,-20) *+{n-2}         = "52",
(30,-20) *+{\cdots}      = "53",
(45,-20) *+{2}           = "54",
(60,-20) *+{(-1)^{n-2}l} = "55"
\ar "11";"21"
\ar "21";"22" 
\ar "21";"31" \ar "22";"32"
\ar "31";"41" \ar "32";"42" 
\ar "41";"42" \ar "42";"43" \ar "43";"44"
\ar "41";"51" \ar "42";"52" \ar "44";"54"
\ar "51";"52" \ar "52";"53" \ar "53";"54" \ar "54";"55"
\end{xy}.
\end{align}
Here, each number $i$ in the row starting at $j$ 
denotes a one-dimensional vector space $K p(i,j,l)$ with a basis $p(i,j,l)$,
and each arrow stands for the identity map $K \to K$
with respect to these bases.

If we use the ``Young diagram-like'' notation as (\ref{eq_Young_A})
for the case $n=9$ and $l=1$, then $\Pi e_l$ is denoted by
\begin{align}\label{eq_Young_D_1}
\begin{ytableau} 
1  \\
2  &  -1 \\
3  &  2  &  1  \\
4  &  3  &  2  &  -1 \\
5  &  4  &  3  &  2  &  1  \\
6  &  5  &  4  &  3  &  2  &  -1 \\
7  &  6  &  5  &  4  &  3  &  2  &  1  \\
8  &  7  &  6  &  5  &  4  &  3  &  2  &  -1 \\  
\end{ytableau}.
\end{align}

The indecomposable $\tau^{-1}$-rigid module $J(w)$ 
for $w \in \jirr W$ of type $l = \pm 1$ is given as follows.

\begin{Prop}\label{Prop_rigid_D_1}\cite[Theorem 6.5]{IRRT}
Let $w \in \jirr W$ be a join-irreducible element of type $l = \pm 1$.
Then the module $J(w) \in \itirigid \Pi$ is expressed as follows.
\begin{itemize}
\item Consider the diagram \textup{(\ref{eq_pattern_D_1})}.
\item For each $j \in \{l\} \cup [2,n-1]$, in the row starting at $j$,
keep the entries $i$ satisfying $i \ge w(|j|+1)$ and delete the others.
\end{itemize}
\end{Prop}

Next, we consider the case $l \ge 2$.
Let $i \in \pm Q_0 = \pm [1,n-1]$ and 
$j \in Q_0=\{-1\} \cup [1,n-1]$ with $i \le j \ge l$.
Set $t:=(-1)^{j-l+1}$.
We define two paths $p_1(i,j,l)$ and $p_{-1}(i,j,l)$ in $Q$ by
\begin{align*}
p_\epsilon(i,j,l):=\begin{cases}
(\alpha_i \alpha_{i-1} \cdots \alpha_{j-1}) \cdot
(\beta_j \beta_{j-1} \cdots \beta_{l+1}) & (i \ge 2) \\
\alpha_1^+ p_\epsilon(2,j,l) & (i=1) \\
\alpha_1^- p_\epsilon(2,j,l) & (i=-1) \\
\beta_2 p_\epsilon(\epsilon t,j,l) & (i=-2) \\
(\beta_{-i} \beta_{-i-1} \cdots \beta_3) p_\epsilon(-2,j,l) 
& (i \le -3)
\end{cases}.
\end{align*}
This is a shortest path 
\begin{itemize}
\item starting from $i$, going through $j$, and ending at $l$ if $i \ge -1$; and
\item starting from $|i|$, going through $\epsilon t$ and then $j$, 
and ending at $l$ if $i \le -2$.
\end{itemize}
As an element in $\Pi$, 
the path $p_\epsilon(i,j,l)$ is not zero in $\Pi$ if and only if $i \ge j-(n-1)-l$,
so set 
\begin{align*}
\Gamma[l]:=\{ (i,j) \in \pm Q_0 \times Q_0 \mid j-(n-1)-l \le i \le j \ge l \}.
\end{align*}

We obtain the following assertion from straightforward calculation.

\begin{Lem}\label{Lem_basis_D_2}
Let $\epsilon=\pm 1$.
Then the set $\{p_\epsilon(i,j,l) \mid (i,j) \in \Gamma[l] \}$ 
forms a $K$-basis of $\Pi e_l$.
\end{Lem}

Each basis above allows us to express $\Pi e_l$ as
{\small\begin{align}\label{eq_pattern_D_2}
\begin{xy}
(  0, 15) *+{l}                = "11",
( 15, 15) *+{l\!-\!1}          = "12",
( 27, 15) *+{\cdots}           = "13",
( 39, 15) *+{2}                = "14",
( 54, 17) *+{-\epsilon}        = "15a",
( 54, 13) *+{\epsilon}         = "15b",
( 69, 15) *+{-2}               = "16",
( 81, 15) *+{\cdots}           = "17",
( 93, 15) *+{-m}               = "18",
(108, 15) *+{-m\!-\!1}         = "19",
(120, 15) *+{\cdots}           = "20",
(132, 15) *+{-n\!+\!2}         = "21",
(147, 15) *+{-n\!+\!1}         = "22",
(  0,  5) *+{l\!+\!1}          = "31",
( 15,  5) *+{l}                = "32",
( 27,  5) *+{\cdots}           = "33",
( 39,  5) *+{3}                = "34",
( 54,  5) *+{2}                = "35",
( 69,  7) *+{\epsilon}         = "36a",
( 69,  3) *+{-\epsilon}        = "36b",
( 81,  5) *+{\cdots}           = "37",
( 93,  5) *+{-m\!+\!1}         = "38",
(108,  5) *+{-m}               = "39",
(120,  5) *+{\cdots}           = "40",
(132,  5) *+{-n\!+\!3}         = "41",
(147,  5) *+{-n\!+\!2}         = "42",
(  0, -5) *+{\vdots}           = "51",
( 15, -5) *+{\vdots}           = "52",
( 39, -5) *+{\vdots}           = "54",
( 54, -5) *+{\vdots}           = "55",
( 69, -5) *+{\vdots}           = "56",
( 93, -5) *+{\vdots}           = "58",
(108, -5) *+{\vdots}           = "59",
(132, -5) *+{\vdots}           = "61",
(147, -5) *+{\vdots}           = "62",
(  0,-15) *+{n\!-\!1}          = "71",
( 15,-15) *+{n\!-\!2}          = "72",
( 27,-15) *+{\cdots}           = "73",
( 39,-15) *+{m\!+\!1}          = "74",
( 54,-15) *+{m}                = "75",
( 69,-15) *+{m\!-\!1}          = "76",
( 81,-15) *+{\cdots}           = "77",
( 93,-13) *+{-\epsilon t}      = "78a",
( 93,-17) *+{\epsilon t}       = "78b",
(108,-15) *+{-2}               = "79",
(120,-15) *+{\cdots}           = "80",
(132,-15) *+{-l\!+\!1}         = "81",
(147,-15) *+{-l}               = "82",
\ar "11";"12" \ar "12";"13" \ar "13";"14" 
\ar "14";"15a" \ar "14";"15b" \ar "15a";"16" \ar_{-1} "15b";"16"
\ar "16";"17" \ar "17";"18" \ar "18";"19" \ar "19";"20" \ar "20";"21" \ar "21";"22"
\ar "11";"31" \ar "12";"32" \ar "14";"34" \ar "15b";"35"
\ar "16";"36a" \ar "18";"38" \ar "19";"39" \ar "21";"41" \ar "22";"42" 
\ar "31";"32" \ar "32";"33" \ar "33";"34" \ar "34";"35" 
\ar "35";"36a" \ar "35";"36b" \ar "36a";"37" \ar_{-1} "36b";"37"
\ar "37";"38" \ar "38";"39" \ar "39";"40" \ar "40";"41" \ar "41";"42"
\ar "31";"51" \ar "32";"52" \ar "34";"54" \ar "35";"55" \ar "36b";"56" 
\ar "38";"58" \ar "39";"59" \ar "41";"61" \ar "42";"62" 
\ar "51";"71" \ar "52";"72" \ar "54";"74" \ar "55";"75" \ar "56";"76" \ar "58";"78a" 
\ar "59";"79" \ar "61";"81" \ar "62";"82" 
\ar "71";"72" \ar "72";"73" \ar "73";"74" \ar "74";"75" \ar "75";"76" \ar "76";"77"
\ar "77";"78a" \ar "77";"78b" \ar "78a";"79" \ar_{-1} "78b";"79" 
\ar "79";"80" \ar "80";"81" \ar "81";"82"
\end{xy}, \end{align}}where $m:=n-l$, $t=(-1)^{m-1}$.
Here, each number $i$ in the row starting at $j$ 
denotes a one-dimensional vector space $K p(i,j,l)$ with a basis $p(i,j,l)$.
Each arrow with the label ``$-1$'' stands for the map $K \ni x \mapsto -x \in K$,
and each of the other arrows stands for the identity map $K \to K$,
with respect to these bases.

If we use the ``Young diagram-like'' notation as (\ref{eq_Young_A})
for the case $n=9$, $l=2$, and $\epsilon=1$, then $\Pi e_l$ is denoted by
\begin{align}\label{eq_Young_D_2}
\begin{ytableau} 
2      & \mpone & -2     & -3     & -4     & -5     & -6     & -7     & -8     \\
3      & 2      & \pmone & -2     & -3     & -4     & -5     & -6     & -7     \\
4      & 3      & 2      & \mpone & -2     & -3     & -4     & -5     & -6     \\
5      & 4      & 3      & 2      & \pmone & -2     & -3     & -4     & -5     \\
6      & 5      & 4      & 3      & 2      & \mpone & -2     & -3     & -4     \\
7      & 6      & 5      & 4      & 3      & 2      & \pmone & -2     & -3     \\
8      & 7      & 6      & 5      & 4      & 3      & 2      & \mpone & -2
\end{ytableau}.
\end{align}
We use similar notation for subfactor modules of $\Pi e_l$. 

The indecomposable $\tau^{-1}$-rigid module $J(w)$ 
for $w \in \jirr W$ of type $l \ne \pm 1$ is given as follows.

\begin{Prop}\label{Prop_rigid_D_2}\cite[Theorem 6.12]{IRRT}
Let $w \in \jirr W$ be a join-irreducible element of type $l \ne \pm 1$.
If $w(l+1) \le 1$, 
then set 
\begin{align*}
m:=\max\{ k \in [l+1,n] \mid w(k) \le 1\}, \quad
\epsilon:=\begin{cases}
(-1)^{m-(l+1)}     & (w(m) \le -2) \\
(-1)^{m-(l+1)}w(m) & (w(m) = \pm 1)
\end{cases};
\end{align*}
otherwise, set $\epsilon := 1$.
Then the module $J(w) \in \itirigid \Pi$ is expressed as follows.
\begin{itemize}
\item Consider the diagram \textup{(\ref{eq_pattern_D_2})}.
\item For each $j \in [l,n-1]$, in the row starting at $j$,
keep the entries $i$ satisfying
\begin{align*}
\begin{cases}
i \ge w(j+1) & (w(j+1) \ge 2) \\
i \ge 2 \quad \textup{or} \quad i=w(j+1) & (w(j+1)=\pm 1) \\
i \ge w(j+1)+1 & (w(j+1) \le -2)
\end{cases}
\end{align*}
and delete the others.
\end{itemize}
\end{Prop}

\section{Description of bricks}\label{Sec_brick}

In this section, 
we describe the bricks over the preprojective algebras $\Pi$ of Dynkin type 
$\Delta=\bbA_n,\bbD_n$.
For $w \in \jirr W$, we have obtained that the brick $S(w)$ 
is $\soc_{\End_\Pi(J(w))}J(w)$ in Proposition \ref{Prop_jirrW_brick},
and the module $J(w) \in \itirigid \Pi$ is combinatorially determined in 
Propositions \ref{Prop_rigid_A}, \ref{Prop_rigid_D_1}, and \ref{Prop_rigid_D_2}.

We remark that the bricks in $\mod \Pi$ coincide 
with the layers of $\Pi$ \cite[Theorem 1.2]{IRRT}.
Thus, the dimension vector of each brick in $\mod \Pi$ is a positive root 
by \cite[Theorem 2.7]{AIRT}. 
Here, a module $L$ in $\mod \Pi$ is called a \textit{layer} 
if there exist some $w \in W$ and some vertex $i$ in $\Delta$ 
such that $w < ws_i$ and $L \cong I(w)/I(ws_i)$ \cite[Section 2]{AIRT}.

\subsection{Type $\bbA_n$}

We state the result and give an example first.

\begin{Thm}\label{Thm_brick_A}
Let $w \in \jirr W$ be a join-irreducible element of type $l$.
Set  
\begin{align*}
R:=w([l+1,n+1]), \quad a:=w(l), \quad b=w(l+1), \quad V=[b,a-1].
\end{align*}
Then, the brick $S(w)$ is isomorphic to the $\Pi$-module $S'(w)$ defined as follows.
\begin{itemize}
\item[(a)]
The brick $S'(w)$ has a $K$-basis $(\ang{i})_{i \in V}$,
and if $j = i$, then $e_j \ang{i}=\ang{i}$; otherwise, $e_j \ang{i}=0$.
\item[(b)]
Let $i \in V$.
If $j \ne i-1$, then $\alpha_j \ang{i}=0$.
If $j \ne i+1$, then $\beta_j \ang{i}=0$.
\item[(c)]
If $i \in V \setminus \{\max V\}$, then
\begin{align*}
\alpha_i \ang{i+1} = \begin{cases}
\ang{i} & (i+1 \notin R) \\
0    & (i+1 \in R)    
\end{cases}, \quad 
\beta_{i+1} \ang{i} = \begin{cases}
0      & (i+1 \notin R) \\   
\ang{i+1} & (i+1 \in R)    
\end{cases}.  
\end{align*}
\end{itemize}
\end{Thm}

\begin{Ex}\label{Ex_brick_A}
Let $n:=8$ and $w=(2,5,8,1,3,4,6,7,9)$.
Then, we have $l=3$, $a=8$, $b=1$, and $V=[1,7]$.
The module $S(w)$ has a $K$-basis $\ang{1},\ang{2},\ldots,\ang{7}$ and
its structure as a $\Pi$-module can be written as 
\begin{align*}
\ang{1} \xleftarrow{\alpha_1} \ang{2} \xrightarrow{\beta_3} \ang{3} 
\xrightarrow{\beta_4} \ang{4} \xleftarrow{\alpha_4} \ang{5} 
\xrightarrow{\beta_6} \ang{6} \xrightarrow{\beta_7} \ang{7}.
\end{align*}

In an abbreviated form, the brick $S(w)$ is denoted by
\begin{align}\label{eq_abbr_A}
1 \leftarrow 2 \rightarrow 3 \rightarrow 4 \leftarrow 5 \rightarrow 6 \rightarrow 7.
\end{align}

If we use the notation as (\ref{eq_Young_A}),
then by Proposition \ref{Prop_rigid_A},
the module $J(w)$ and the ``position'' of a submodule $S(w)$ in $J(w)$ are 
described as follows:
\begin{align*}
J(w)=\begin{ytableau} 
3  &  2  &  1  \\
4  &  3  \\
5  &  4  \\
6  \\
7  \\
\end{ytableau}, \quad
S(w)=\begin{ytableau} 
\none  &  2  &  1  \\
\none  &  3  \\
5      &  4  \\
6      \\
7      \\
\end{ytableau}.
\end{align*}
\end{Ex}

If we use such abbreviated expressions of bricks as (\ref{eq_abbr_A}),
then the theorem can be restated as follows.

\begin{Cor}\label{Cor_brick_abbr_A}
Let $w \in \jirr W$ be a join-irreducible element of type $l$,
and use the setting of Theorem \ref{Thm_brick_A}.
We express the brick $S(w)$ in the following abbreviation rules.
\begin{itemize}
\item For each $i \in V$, the $K$-vector subspace $K\ang{i}$ is denoted by the symbol $i$.
\item If the action of some $\gamma \in Q_1$ on $S(w)$ induces a nonzero $K$-linear map
$K\ang{i} \to K\ang{j}$, then we draw an arrow from the symbol $i$ to the symbol $j$.
\end{itemize}
Then, for each $i \in V \setminus \{\max V\}$,
there exists exactly one arrow between $i$ and $i+1$,
and the orientation is $i \to i+1$ if $i+1 \in R$ and 
$i \leftarrow i+1$ if $i+1 \notin R$.
\end{Cor}

It is easy to see that there exists some path algebra $A$ of type $\bbA_n$ 
such that the brick $S(w)$ is an $A$-module,
and that any 2-cycle in $Q$ annihilates all the bricks in $\Pi$.
Let $I$ be the ideal of $\Pi$ generated by all the 2-cycles in $Q$, 
then \cite[Corollary 5.20]{DIRRT} implies that 
$\torf \Pi \cong \torf (\Pi/I)$ as lattices.
Thus, there is an isomorphism from $W$ to $\torf(\Pi/I)$ as lattices
by Propositions \ref{Prop_comm} and \ref{Prop_W_stitilt}.
The relationship between $W$ and $\Pi/I$ is investigated 
from another point of view in \cite[Section 4]{BCZ}.

In order to show our result, we restate Proposition \ref{Prop_rigid_A} as follows.

\begin{Lem}\label{Lem_rigid_A}
Let $w \in \jirr W$ be a join-irreducible element of type $l$.
\begin{itemize}
\item[(1)] 
Assume $(i,j) \in \Gamma[l]$. 
Then $p(i,j,l) \notin I(w)$ holds if and only if $i \ge w(j+1)$.
\item[(2)]
Define $\Gamma(w) \subset \Gamma[l]$ as the subset consisting of
the elements $(i,j) \in \Gamma[l]$ with $p(i,j,l) \notin I(w)$.
Then the set $\{p(i,j,l) \mid (i,j) \in \Gamma(w)\}$ induces a $K$-basis of $J(w)$.
\end{itemize}
\end{Lem}

To express $S(w)$, we define the following set for $k \ge 1$:
\begin{align*}
\Gamma_k(w):=\{ (i,j) \in \Gamma(w) \mid 
\min \{x \ge 1 \mid (i,j+x) \notin \Gamma(w) \} = k \}.
\end{align*}
It is easy to see that $\Gamma(w)$ is the disjoint union of the $\Gamma_k(w)$'s.
Moreover, we extend the definition of the path $p(i,j,l)$ to
$\tilde{\Gamma}[l]:=\{ (i,j) \in Q_0 \times \Z \mid i \le j \ge l \}$
by setting $p(i,j,l):=0$ if $j \ge n+1$,
and define $w(k):=k$ if $k \ge n+2$.

\begin{Lem}\label{Lem_soc_A}
Let $w \in \jirr W$ be a join-irreducible element of type $l$.
Consider the endomorphism $f:=(\cdot p(l,l+1,l)) \colon J(w) \to J(w)$.
\begin{itemize}
\item[(1)]
We have $S(w)=\Ker f$.
\item[(2)]
Let $(i,j) \in \Gamma(w)$. 
Then $p(i,j,l) \in \Ker f$ holds if and only if $(i,j) \in \Gamma_1(w)$.
\item[(3)]
The set $\{p(i,j,l) \mid (i,j) \in \Gamma_1(w) \}$ induces a $K$-basis of $\Ker f$.
\end{itemize}
\end{Lem}

\begin{proof}
(1)
For every nonisomorphic endomorphism $g \colon J(w) \to J(w)$,
it is clear that there exists $h \colon J(w) \to J(w)$ such that $g=hf$.
Thus, $S(w)=\Ker f$ holds.

(2)
As an element in $\Pi$, we have 
$f(p(i,j,l))=p(i,j,l) p(l,l+1,l)=p(i,j+1,l)$.
Then Lemma \ref{Lem_rigid_A} implies the assertion.

(3)
From Lemma \ref{Lem_rigid_A},
recall that the set $\{p(i,j,l) \mid (i,j) \in \Gamma(w)\}$ induces a basis of $J(w)$,
so this set is linearly independent in $J(w)$.

Thus, the set $\{ p(i,j,l) \mid (i,j) \in \Gamma_1(w)\}$ 
is linearly independent in $J(w)$, and is contained in $\Ker f$ by (2).

On the other hand, in the proof of (2), we got $f(p(i,j,l))=p(i,j+1,l)$. 
If $(i,j) \in \Gamma(w) \setminus \Gamma_1(w)$,
then $(i,j+1,l) \in \Gamma(w)$.
The set
$\{ p(i,j+1,l) \mid (i,j) \in \Gamma(w) \setminus \Gamma_1(w)\}$
is linearly independent in $J(w)$.
Thus, the set $\{ p(i,j,l) \mid (i,j) \in \Gamma_1(w)\}$ 
generates $\Ker f$ as a $K$-vector space in $J(w)$.

Therefore, we conclude that the set $\{ p(i,j,l) \mid (i,j) \in \Gamma_1(w)\}$
induces a $K$-basis of $\Ker f$.
\end{proof}

\begin{Lem}\label{Lem_vertex_A}
Let $w \in \jirr W$ be a join-irreducible element of type $l$,
and define $V$ as in Theorem \ref{Thm_brick_A}.
Then, there exists a bijection $\Gamma_1(w) \to V$ given by $(i,j) \mapsto i$.
\end{Lem}

\begin{proof}
In the proof, we fully use the notation in Theorem \ref{Thm_brick_A}.

We first show the well-definedness of the map $\Gamma_1(w) \to V$.

We remark that, 
for $k \in [l+1,n+1]$, the condition $w(k)=k$ holds if and only if $k > a$,
and that this condition is also equivalent to $w(k)>a$.
Lemma \ref{Lem_rigid_A} and $(i,j) \in \Gamma(w)$ give $j \ge i \ge w(j+1)$.
Thus, $w(j+1) \le j$ holds, so we get $j+1 \le a$, or equivalently, $j<a$.
Therefore, we obtain $i \le j < a$.

On the other hand, Lemma \ref{Lem_rigid_A} and $(i,j) \in \Gamma(w)$ 
also imply $j \ge i \ge w(j+1) \ge w(l+1)=b$.

These imply that the map $\Gamma_1(w) \to V$ is well-defined.
It is clearly injective by Lemma \ref{Lem_rigid_A}.

We next prove that the map $\Gamma_1(w) \to V$ is also surjective.
Let $i \in V$.
Then $i<a$ holds, so there exists some $j \in [l,n]$ 
such that $(i,j) \in \Gamma(w)$ by Lemma \ref{Lem_rigid_A}.
Take the maximum $j$, then it is easy to obtain $(i,j) \in \Gamma_1(w)$
from Lemma \ref{Lem_rigid_A}.

Hence, the map $\Gamma_1(w) \to V$ is also surjective, and thus, bijective.
\end{proof}

Now, we show Theorem \ref{Thm_brick_A}.

\begin{proof}
By Lemma \ref{Lem_soc_A}, we can define a map $\rho \colon V \to Q_0$ as follows:
$\rho(i)$ is the unique element $j \in Q_0$ 
such that $(i,j) \in \Gamma_1(w)$.
Set $\ang{i}:=p(i,\rho(i),l)$ for each $i \in V$.
It suffices to show that $(\ang{i})_{i \in V}$ satisfies the properties (a), (b), and (c),
since the three properties are enough to define a $\Pi$-module.

First, $(\ang{i})_{i \in V}$ is a $K$-basis of $S(w)$ 
by Lemma \ref{Lem_vertex_A}, 
and $K\ang{i}$ is clearly a subspace of $e_i S(w)$.
Thus, the property (a) holds, and (b) follows from (a).

We begin the proof of (c).

Let $i \in V \setminus \{ \max V \}$ and set $j:=\rho(i+1)$.
Then,
\begin{align*}
\alpha_i \ang{i+1}=\alpha_i p(i+1,j,l)=p(i,j,l)=\begin{cases}
\ang{i} & (\text{if $i+1 \notin R$, since $(i,j) \in \Gamma_1(w)$}) \\
0    & (\text{if $i+1 \in R$, since $(i,j) \notin \Gamma(w)$})
\end{cases}.
\end{align*}

Next, let $i \in V \setminus \{ \max V \}$ and set $j:=\rho(i)$.
Then,
\begin{align*}
\beta_{i+1} \ang{i}&=\beta_{i+1} p(i,j,l)=p(i+1,j+1,l)\\
&=\begin{cases}
0      & (\text{if $i+1 \in R$, since $(i+1,j+1) \notin \Gamma(w)$}) \\
\ang{i+1} & (\text{if $i+1 \notin R$, since $(i+1,j+1) \in \Gamma_1(w)$})
\end{cases}.
\end{align*}

From these, we have the property (c).
\end{proof}

\subsection{Type $\bbD_n$}\label{subsec_brick_D}

We state the result and give some examples first.
Recall $\alpha_1=\alpha_1^+ + \alpha_1^-$ and $\beta_2=\beta_2^+ + \beta_2^-$.

\begin{Thm}\label{Thm_brick_D}
Let $w \in \jirr W$ be a join-irreducible element of type $l$.
Set  
\begin{align*}
& R:=w([|l|+1,n]), \quad a:=w(l), \quad b=w(|l|+1), \\
& r:=\max \{ k \ge 0 \mid [1,k] \subset \pm R \}, \quad
c:=\begin{cases}
w^{-1}(|w(1)|)& (r \ge 1) \\
1 & (r=0)
\end{cases}, \\
& (V_-,V_+) :=
\begin{cases}
(\emptyset, [b,a-1]) & (b \ge 2) \\
(\emptyset, \{c\} \cup [2,a-1]) & (b = \pm 1) \\
([b+1,-2] \cup \{-c\}, \{c\} \cup [2,a-1]) & (b \le -2)
\end{cases}, \quad
V := V_+ \amalg V_-.
\end{align*}
Then the brick $S(w)$ is isomorphic to the $\Pi$-module $S'(w)$ defined as follows.
\begin{itemize}
\item[(a)]
The brick $S'(w)$ has a $K$-basis $(\ang{i})_{i \in V}$,
and if $j=|i| \ge 2$ or $j = i \in \{\pm 1\}$, then $e_j \ang{i}=\ang{i}$;
otherwise $e_j \ang{i}=0$.
\item[(b)]
Let $i \in V$.
If $j \ne |i|-1$, then $\alpha_j \ang{i}=0$.
If $j \ne |i|+1$, then $\beta_j \ang{i}=0$.
\item[(c)] 
The remaining actions of arrows are given as follows,
where we set $\ang{j}:=0$ if $j \notin V$.
\begin{itemize}
\item[(i)]
For $i \in V_+ \setminus \{\max V_+\}$, 
we have $\alpha_{|i|} \ang{|i|+1}=\xi_i^+ \ang{i} + \xi_i^- \ang{-i}$, where
\begin{align*}
\xi_i^+ := \begin{cases}
1 & (|i|+1 \notin R) \\
0 & (|i|+1 \in R)
\end{cases}, \quad
\xi_i^- := \begin{cases}
1 & (|i|=1, \ r=0, \ 2 \notin R) \\
0 & (\textup{otherwise})
\end{cases}.
\end{align*}
\item[(ii)]
For $i \in V_+ \setminus \{\max V_+\}$, 
we have $\beta_{|i|+1} \ang{i}=\eta_i^+ \ang{|i|+1} + \eta_i^- \ang{-(|i|+1)}$, where
\begin{align*}
\eta_i^+ := \begin{cases}
1 & (|i|+1 \in R) \\
0 & (|i|+1 \notin R)
\end{cases}, \quad
\eta_i^- := \begin{cases}
-1 & (|i|=1, \ r=0, \ -2 \notin R) \\
0 & (\textup{otherwise})
\end{cases}.
\end{align*}
\item[(iii)]
For $i \in V_- \setminus \{\min V_-\}$, 
we have $\alpha_{|i|} \ang{-(|i|+1)}=\xi_i^+ \ang{-i} + \xi_i^- \ang{i}$, where
\begin{align*}
\xi_i^+ := \begin{cases}
1 & (|i| \le r, \ |i|+1 \in R) \\
0 & (\textup{otherwise})
\end{cases}, \quad
\xi_i^- := \begin{cases}
1 & (-(|i|+1) \in R) \\
0 & (-(|i|+1) \notin R)
\end{cases}.
\end{align*}
\item[(iv)]
For $i \in V_-$, 
we have $\beta_{|i|+1} \ang{i}=\eta_i^+ \ang{|i|+1} + \eta_i^- \ang{-(|i|+1)}$, 
where
\begin{align*}
\eta_i^+ := \begin{cases}
1 & (|i| \le r, \ |i|+1 \notin R) \\
0 & (\textup{otherwise})
\end{cases}, \quad
\eta_i^- := \begin{cases}
1  & (|i| \ne r, \ -(|i|+1) \notin R) \\
-1 & (|i|=r) \\
0  & (\textup{otherwise})
\end{cases}.
\end{align*}
\end{itemize}
\end{itemize}
\end{Thm}

The proof of the theorem given in later depends on 
whether the type $l$ of the join-irreducible element $w \in \jirr W$ is $\pm 1$ or not,
because the description of the indecomposable $\tau^{-1}$-rigid module $J(w)$ does so. 
The following examples show the difference of the calculation
of the brick $S(w)$ in these two cases.

\begin{Ex}\label{Ex_brick_D_1}
Let $n:=9$, $w:=(9,-7,-6,-4,-1,2,3,5,8)$.
Then we have $l=1$, $a=9$, $b=-7$, $r=8$, and $c=-1$.
Thus, $(V_-,V_+)=([-6,-2] \cup \{1\},\{-1\} \cup [2,8])$, 
and the desired brick $S(w)$ is written as
\begin{align*}
\begin{xy}
(  0, 8) *+{\ang{ 1}} = "1-",
( 16, 8) *+{\textcolor{red}{\ang{-2}}} = "2-",
( 32, 8) *+{\textcolor{red}{\ang{-3}}} = "3-",
( 48, 8) *+{\textcolor{red}{\ang{-4}}} = "4-",
( 64, 8) *+{\textcolor{red}{\ang{-5}}} = "5-",
( 80, 8) *+{\textcolor{red}{\ang{-6}}} = "6-",
(  0,-8) *+{\ang{-1}} = "1+",
( 16,-8) *+{\ang{ 2}} = "2+",
( 32,-8) *+{\ang{ 3}} = "3+",
( 48,-8) *+{\ang{ 4}} = "4+",
( 64,-8) *+{\ang{ 5}} = "5+",
( 80,-8) *+{\ang{ 6}} = "6+",
( 96,-8) *+{\ang{ 7}} = "7+",
(112,-8) *+{\ang{ 8}} = "8+",
\ar^{ \beta_2^-} "1+";"2+"
\ar^{ \beta_3}   "2+";"3+"
\ar_{\alpha_3}   "4+";"3+"
\ar^{ \beta_5}   "4+";"5+"
\ar_{\alpha_5}   "6+";"5+"
\ar_{\alpha_6}   "7+";"6+"
\ar^{ \beta_8}   "7+";"8+"
\ar^{ \beta_2^+} "1-";"2-" \ar_{\alpha_1^-} "2-";"1+"
\ar^{ \beta_3}   "2-";"3-" \ar_{\alpha_2  } "3-";"2+"
\ar_{\alpha_3}   "4-";"3-" \ar^{ \beta_4  } "3-";"4+"
\ar^{ \beta_5}   "4-";"5-" \ar_{\alpha_4  } "5-";"4+"
\ar_{\alpha_5}   "6-";"5-" \ar^{ \beta_6  } "5-";"6+"
                           \ar^{ \beta_7  } "6-";"7+"
\end{xy}.
\end{align*}

By omitting the labels of the arrows,
the brick $S(w)$ can be written in the following abbreviated way,
which is enough to determine $S(w)$ up to isomorphisms.

\begin{align}\label{eq_abbr_D_1}
\begin{xy}
(  0, 8) *+{ 1} = "1-",
( 16, 8) *+{\rmtwo} = "2-",
( 32, 8) *+{\rmthr} = "3-",
( 48, 8) *+{\rmfou} = "4-",
( 64, 8) *+{\rmfiv} = "5-",
( 80, 8) *+{\rmsix} = "6-",
(  0,-8) *+{-1} = "1+",
( 16,-8) *+{ 2} = "2+",
( 32,-8) *+{ 3} = "3+",
( 48,-8) *+{ 4} = "4+",
( 64,-8) *+{ 5} = "5+",
( 80,-8) *+{ 6} = "6+",
( 96,-8) *+{ 7} = "7+",
(112,-8) *+{ 8} = "8+",
\ar "1+";"2+"
\ar "2+";"3+"
\ar "4+";"3+"
\ar "4+";"5+"
\ar "6+";"5+"
\ar "7+";"6+"
\ar "7+";"8+"
\ar "1-";"2-" \ar "2-";"1+"
\ar "2-";"3-" \ar "3-";"2+"
\ar "4-";"3-" \ar "3-";"4+"
\ar "4-";"5-" \ar "5-";"4+"
\ar "6-";"5-" \ar "5-";"6+"
              \ar "6-";"7+"
\end{xy}
\end{align} 

We use the notation as (\ref{eq_Young_D_1}),
then by Proposition \ref{Prop_rigid_D_1},
the module $J(w)$ and the ``position'' of a submodule $S(w)$ in $J(w)$ are 
described as follows:
\begin{align*}
J(w)=\begin{ytableau} 
1  \\
2  &  -1 \\
3  &  2  &  1  \\
4  &  3  &  2  &  -1 \\
5  &  4  &  3  &  2  \\
6  &  5  &  4  &  3  \\
7  &  6  &  5  \\
8  \\  
\end{ytableau}, \quad
S(w)=\begin{ytableau} 
\none \\
\none \\
\none &  \none &  1     \\
\none &  \none &  \rtwo &  -1    \\
\none &  \rfou &  \rthr &  2     \\
\rsix &  \rfiv &  4     &  3     \\
7     &  6     &  5     \\
8     \\
\end{ytableau}.
\end{align*}

In the figure for $S(w)$, 
every square $\begin{ytableau}\textcolor{red}{i}\end{ytableau}$ with a red letter 
denotes $K\ang{-i}$, which is a subspace of $e_i S(w)$.
There are five such squares 
$\begin{ytableau}\rtwo\end{ytableau}, \begin{ytableau}\rthr\end{ytableau}, 
\begin{ytableau}\rfou\end{ytableau}, \begin{ytableau}\rfiv\end{ytableau}, 
\begin{ytableau}\rsix\end{ytableau}$.
Every other square $\begin{ytableau}i\end{ytableau}$ denotes $K\ang{i}$, 
and it is a subspace of $e_iS(w)$.
\end{Ex}

\begin{Ex}\label{Ex_brick_D_2}
Let $n:=9$, $w:=(-6,9,-7,-4,-1,2,3,5,8)$. 
Then we have $l=2$, $a=9$, $b=-7$, $r=5$, and $c=-1$.
Thus, $(V_-,V_+)=([-6,-2] \cup \{1\},\{-1\} \cup [2,8])$, 
and the desired brick $S(w)$ is written as 
\begin{align*}
\begin{xy}
(  0, 8) *+{\ang{ 1}} = "1-",
( 16, 8) *+{\textcolor{blue}{\ang{-2}}} = "2-",
( 32, 8) *+{\textcolor{blue}{\ang{-3}}} = "3-",
( 48, 8) *+{\textcolor{blue}{\ang{-4}}} = "4-",
( 64, 8) *+{\textcolor{blue}{\ang{-5}}} = "5-",
( 80, 8) *+{\ang{-6}} = "6-",
(  0,-8) *+{\ang{-1}} = "1+",
( 16,-8) *+{\ang{ 2}} = "2+",
( 32,-8) *+{\ang{ 3}} = "3+",
( 48,-8) *+{\ang{ 4}} = "4+",
( 64,-8) *+{\ang{ 5}} = "5+",
( 80,-8) *+{\ang{ 6}} = "6+",
( 96,-8) *+{\ang{ 7}} = "7+",
(112,-8) *+{\ang{ 8}} = "8+",
\ar^{ \beta_2^-} "1+";"2+"
\ar^{ \beta_3  } "2+";"3+"
\ar_{\alpha_3  } "4+";"3+"
\ar^{ \beta_5  } "4+";"5+"
\ar_{\alpha_5  } "6+";"5+"
\ar_{\alpha_6  } "7+";"6+"
\ar^{ \beta_8  } "7+";"8+"
\ar^{ \beta_2^+} "1-";"2-" \ar_{\alpha_1^-} "2-";"1+"
\ar^{ \beta_3  } "2-";"3-" \ar_{\alpha_2  } "3-";"2+"
\ar_{\alpha_3  } "4-";"3-" \ar^{ \beta_4  } "3-";"4+"
\ar^{ \beta_5  } "4-";"5-" \ar_{\alpha_4  } "5-";"4+"
\ar^{-\beta_6  } "5-";"6-" \ar^{ \beta_6  } "5-";"6+"
\end{xy}
\end{align*}

The brick $S(w)$ can be written in the following abbreviated way.

\begin{align}\label{eq_abbr_D_2}
\begin{xy}
(  0, 8) *+{ 1} = "1-",
( 16, 8) *+{\textcolor{blue}{-2}} = "2-",
( 32, 8) *+{\textcolor{blue}{-3}} = "3-",
( 48, 8) *+{\textcolor{blue}{-4}} = "4-",
( 64, 8) *+{\textcolor{blue}{-5}} = "5-",
( 80, 8) *+{-6} = "6-",
(  0,-8) *+{-1} = "1+",
( 16,-8) *+{ 2} = "2+",
( 32,-8) *+{ 3} = "3+",
( 48,-8) *+{ 4} = "4+",
( 64,-8) *+{ 5} = "5+",
( 80,-8) *+{ 6} = "6+",
( 96,-8) *+{ 7} = "7+",
(112,-8) *+{ 8} = "8+",
\ar "1+";"2+" \ar "1-";"2-" \ar "2-";"1+"
\ar "2+";"3+" \ar "2-";"3-" \ar "3-";"2+"
\ar "4+";"3+" \ar "4-";"3-" \ar "3-";"4+"
\ar "4+";"5+" \ar "4-";"5-" \ar "5-";"4+"
\ar "6+";"5+" \ar "5-";"6-" \ar "5-";"6+"
\ar "7+";"6+"
\ar "7+";"8+"
\end{xy}
\end{align} 

Now we use the notation as (\ref{eq_Young_D_2}),
then by Proposition \ref{Prop_rigid_D_2},
the module $J(w)$ and the ``position'' of a submodule $S(w)$ in $J(w)$ are 
described as follows:
\begin{align*}
J(w)=\begin{ytableau} 
2      & \pmone & -2     & -3     & \bmfou & \bmfiv & -6     \\
3      & 2      & \mpone & \bmtwo & \bmthr \\
4      & 3      & \btwo  & \xmone \\
5      & \bfou  & \bthr  & 2      \\
6      & \bfiv  & 4      & 3\\
7      & 6      & 5      \\
8
\end{ytableau}, \quad
S(w)=\begin{ytableau} 
\none  & \none  & \none  & \none  & \pfou  & \pfiv  & -6     \\
\none  & \none  & \xpone & \ptwo  & \pthr  \\
\none  & \none  & \ptwo  & \xmone \\
\none  & \pfou  & \pthr  & 2      \\
\none  & \pfiv  & 4      & 3\\
7      & 6      & 5      \\
8
\end{ytableau}.
\end{align*}

In the figure for $S(w)$, for each $i=2,3,4,5$, 
the two squares $\begin{ytableau}\textcolor{blue}{(i)}\end{ytableau}$  
together denote a certain one-dimensional subspace
of the two-dimensional vector space corresponding to the two squares 
$\begin{ytableau}\textcolor{blue}{i}\end{ytableau}$ 
and $\begin{ytableau}\textcolor{blue}{-i}\end{ytableau}$ in the figure for $J(w)$.
This one-dimensional vector space is actually $K\ang{-i}$, 
which is a subspace of $e_iS(w)$.
Every other square $\begin{ytableau}i\end{ytableau}$ 
in the figure for $S(w)$ denotes $K\ang{i}$,
which is a subspace of $e_{|i|}S(w)$ if $i \le -2$,
and of $e_iS(w)$ if $i \ge -1$.
\end{Ex}

We mainly use such abbreviated expressions of bricks 
as (\ref{eq_abbr_D_1}) and (\ref{eq_abbr_D_2}) in the rest.
Theorem \ref{Thm_brick_D} can be restated as follows by using the abbreviated expressions.

\begin{Cor}\label{Cor_brick_abbr_D}
Let $w \in \jirr W$ be a join-irreducible element of type $l$,
and use the setting of Theorem \ref{Thm_brick_D}.
We express the brick $S(w)$ in the same abbreviation rules as 
Corollary \ref{Cor_brick_abbr_A}.
Then, there are the following arrows, and no other arrows exist.
\begin{itemize}
\item[(i)]
For each $i \in V_+ \setminus \{\max V_+\}$,
there exists an arrow $i \to |i|+1$ if $|i|+1 \in R$; and $i \leftarrow |i|+1$ otherwise.
\item[(ii)]
For each $i \in V_- \setminus \{\min V_-\}$,
there exists an arrow $i \leftarrow -(|i|+1)$ if $-(|i|+1) \in R$; 
and $i \to -(|i|+1)$ otherwise.
\item[(iii)]
If $r \ge 1$, then for each $i \in V_-$ with $|i| \le r$,
there exists an arrow $-i \leftarrow -(|i|+1)$ if $|i|+1 \in R$;
and $i \to |i|+1$ otherwise.
\item[(iv)]
If $r=0$, then
there exists an arrow $-c \leftarrow 2$ if $c \leftarrow 2$ exists in \textup{(i)}, 
and an arrow $c \to -2$ if $-c \rightarrow -2$ exists in \textup{(ii)}.
\end{itemize}
\end{Cor}

\begin{proof}
We remark that, 
for $i \in [1,r]$, the condition $-i \in R$ is equivalent to $i \notin R$.
Then, Theorem \ref{Thm_brick_D} yields the assertion.
\end{proof}

Unlike the case of type $\bbA_n$, for $w \in \jirr W$,
there may not exist a path algebra $A$ of type $\bbD_n$ 
such that the brick $S(w)$ is an $A$-module.
For example, the bricks obtained in Examples \ref{Ex_brick_D_1} and \ref{Ex_brick_D_2} 
cannot be modules over any path algebra of type $\bbD_n$, 
since the 2-cycle $\alpha_2 \beta_3$ annihilates none of them.
Our results imply that, if an element in $\Pi$ is the product of some two 2-cycles,
then it annihilates all the bricks in $\brick \Pi$.

We give more examples.

\begin{Ex}
In these examples, assume $n:=9$.
\begin{itemize}
\item[(1)]
Let $w:=(3,5,8,-7,-4,1,2,6,9)$. 
Then we have $l=3$, $a=8$, $b=-7$, $r=2$, and $c=1$.
Thus, $(V_-,V_+)=([-6,-1],[1,7])$, 
and the desired brick $S(w)$ is written as 
\begin{align*}
\begin{xy}
(  0, 6) *+{-1} = "1-",
( 12, 6) *+{-2} = "2-",
( 24, 6) *+{-3} = "3-",
( 36, 6) *+{-4} = "4-",
( 48, 6) *+{-5} = "5-",
( 60, 6) *+{-6} = "6-",
(  0,-6) *+{ 1} = "1+",
( 12,-6) *+{ 2} = "2+",
( 24,-6) *+{ 3} = "3+",
( 36,-6) *+{ 4} = "4+",
( 48,-6) *+{ 5} = "5+",
( 60,-6) *+{ 6} = "6+",
( 72,-6) *+{ 7} = "7+",
\ar "1+";"2+" \ar "1-";"2-" \ar "2-";"1+"
\ar "3+";"2+" \ar "2-";"3-" \ar "2-";"3+"
\ar "4+";"3+" \ar "4-";"3-"
\ar "5+";"4+" \ar "4-";"5-"
\ar "5+";"6+" \ar "5-";"6-"
\ar "7+";"6+"
\end{xy}.
\end{align*} 
\item[(2)]
Let $w:=(1,3,5,8,-7,-4,2,6,9)$. 
Then we have $l=4$, $a=8$, $b=-7$, $r=0$, and $c=1$.
Thus, $(V_-,V_+)=([-6,-1],[1,7])$, 
and the desired brick $S(w)$ is written as 
\begin{align*}
\begin{xy}
(  0, 6) *+{-1} = "1-",
( 12, 6) *+{-2} = "2-",
( 24, 6) *+{-3} = "3-",
( 36, 6) *+{-4} = "4-",
( 48, 6) *+{-5} = "5-",
( 60, 6) *+{-6} = "6-",
(  0,-6) *+{ 1} = "1+",
( 12,-6) *+{ 2} = "2+",
( 24,-6) *+{ 3} = "3+",
( 36,-6) *+{ 4} = "4+",
( 48,-6) *+{ 5} = "5+",
( 60,-6) *+{ 6} = "6+",
( 72,-6) *+{ 7} = "7+",
\ar "1+";"2+" \ar "1-";"2-" \ar "1+";"2-"
\ar "3+";"2+" \ar "2-";"3-"
\ar "4+";"3+" \ar "4-";"3-"
\ar "5+";"4+" \ar "4-";"5-"
\ar "5+";"6+" \ar "5-";"6-"
\ar "7+";"6+"
\end{xy}.
\end{align*} 
\item[(3)]
Let $w:=(1,2,3,5,8,-7,-4,6,9)$. 
Then we have $l=5$, $a=8$, $b=-7$, $r=0$, and $c=1$.
Thus, $(V_-,V_+)=([-6,-1],[1,7])$, 
and the desired brick $S(w)$ is written as 
\begin{align*}
\begin{xy}
(  0, 6) *+{-1} = "1-",
( 12, 6) *+{-2} = "2-",
( 24, 6) *+{-3} = "3-",
( 36, 6) *+{-4} = "4-",
( 48, 6) *+{-5} = "5-",
( 60, 6) *+{-6} = "6-",
(  0,-6) *+{ 1} = "1+",
( 12,-6) *+{ 2} = "2+",
( 24,-6) *+{ 3} = "3+",
( 36,-6) *+{ 4} = "4+",
( 48,-6) *+{ 5} = "5+",
( 60,-6) *+{ 6} = "6+",
( 72,-6) *+{ 7} = "7+",
\ar "2+";"1+" \ar "1-";"2-" \ar "1+";"2-" \ar "2+";"1-"
\ar "3+";"2+" \ar "2-";"3-"
\ar "4+";"3+" \ar "4-";"3-"
\ar "5+";"4+" \ar "4-";"5-"
\ar "5+";"6+" \ar "5-";"6-"
\ar "7+";"6+"
\end{xy}.
\end{align*} 
\end{itemize}
\end{Ex}

We also have the list of bricks in the case $\Delta=\bbD_5$ in Appendix.

Now, we start the proof of Theorem \ref{Thm_brick_D}.
We divide the argument by whether the type $l$ of $w \in \jirr W$ is $\pm 1$ or not.

We first assume that $l=\pm 1$.
We can restate Proposition \ref{Prop_rigid_D_1} as follows.

\begin{Lem}\label{Lem_rigid_D_1}
Let $w \in \jirr W$ be a join-irreducible element of type $l=\pm 1$.
\begin{itemize}
\item[(1)] 
Assume $(i,j) \in \Gamma[l]$. 
Then $p(i,j,l) \notin I(w)$ holds if and only if $i \ge w(|j|+1)$.
\item[(2)]
Consider the subset $\Gamma(w) \subset \Gamma[l]$ consisting of
the elements $(i,j) \in \Gamma[l]$ with $p(i,j,l) \notin I(w)$.
Then the set $\{p(i,j,l) \mid (i,j) \in \Gamma(w)\}$ induces a $K$-basis of $J(w)$.\end{itemize}
\end{Lem}

In this lemma, we can replace the condition $i \ge w(|j|+1)$ by $|i| \ge w(|j|+1)$ in (1),
since $w(m)=(-1)^{m-1}l$ holds for the number $m:=|w^{-1}(1)|$.

To express $S(w)$, we define the following set for $k \ge 1$:
\begin{align*}
\Gamma_k(w):=\{ (i,j) \in \Gamma(w) \mid 
\min \{x \ge 1 \mid ((-1)^x i,|j|+x) \notin \Gamma(w) \} = k \}.
\end{align*}
It is easy to see that $\Gamma(w)$ is the disjoint union of the $\Gamma_k(w)$'s.
Moreover, we extend the definition of the path $p(i,j,l)$ to
$\tilde{\Gamma}[l]:=\{ (i,j) \in Q_0 \times \Z \mid i \le j \ge l \}$
by setting $p(i,j,l):=0$ if $j \ge n$, and define $w(k):=k$ if $k \ge n+1$.
In Example \ref{Ex_brick_D_1}, the squares with black letters denote $\Gamma_1(w)$,
and the squares with red letters denote $\Gamma_2(w)$.

\begin{Lem}\label{Lem_soc_D_1}
Let $w \in \jirr W$ be a join-irreducible element of type $l=\pm 1$.
Consider the endomorphism 
$f:=(\cdot p(l,3,l)) \colon J(w) \to J(w)$.
\begin{itemize}
\item[(1)]
We have $S(w)=\Ker f$.
\item[(2)]
Let $(i,j) \in \Gamma(w)$. 
Then $p(i,j,l) \in \Ker f$ holds if and only 
if $(i,j) \in \Gamma_1(w) \amalg \Gamma_2(w)$.
\item[(3)]
The set 
$\{p(i,j,l) \mid (i,j) \in \Gamma_1(w) \amalg \Gamma_2(w) \}$ 
induces a $K$-basis of $\Ker f$.
\end{itemize}
\end{Lem}

\begin{proof}
Similar argument to Lemma \ref{Lem_soc_A} works.
We remark that $f(p(i,j,l))=p(i,j,l) p(l,3,l)=p(i,|j|+2,l)$ hold in $\Pi$.
\end{proof}

\begin{Lem}\label{Lem_vertex_D_1}
Let $w \in \jirr W$ be a join-irreducible element of type $l=\pm 1$.
Define $V_+$ and $V_-$ as in Theorem \ref{Thm_brick_A}.
\begin{itemize}
\item[(1)]
There exists a bijection $\Gamma_1(w) \to V_+$ given by $(i,j) \mapsto i$.
\item[(2)]
There exists a bijection $\Gamma_2(w) \to V_-$ given by
$(i,j) \mapsto i$ if $|i|=1$; and $(i,j) \mapsto -i$ otherwise.
\end{itemize}
\end{Lem}

\begin{proof}
We use the notation in Theorem \ref{Thm_brick_D} in the proof.

(1)
We see the well-definedness of the map $\Gamma_1(w) \to V_+$.

We first show that every $(i,j) \in \Gamma_1(w)$ satisfies that $i < a$.
We remark that, 
for $k \in [2,n]$, the condition $w(k)=k$ holds if and only if $k > a$,
and that this condition is also equivalent to $w(k)>a$.
Lemma \ref{Lem_rigid_D_1} and $(i,j) \in \Gamma(w)$ give $j \ge i \ge w(|j|+1)$.
Thus, $w(|j|+1) \le j$ holds, so we get $|j|+1 \le a$, or equivalently, $|j|<a$.
Therefore, we obtain $i \le |j| < a$.

We also prove that, if $(i,j) \in \Gamma_1(w)$ and $|i|=1$, then $i=c$ (*).
In this case, 
Lemma \ref{Lem_rigid_D_1} and $(i,j) \in \Gamma_1(w)$ yield 
$w(|j|+2) \ge 2$ and $i \ge w(|j|+1)$.
Since $|w^{-1}(1)| \ge 2$, we have $w(|j|+1)=\pm1$.
Thus, $|j|+1=|w^{-1}(1)|$ holds;
hence, we have $i=(-1)^{|j|-1} l=(-1)^{|w^{-1}(1)|} l=c$.

Moreover, Lemma \ref{Lem_rigid_D_1} and $(i,j) \in \Gamma(w)$ imply 
$j \ge i \ge w(|j|+1) \ge w(2)=b$.

These imply that the map $\Gamma_1(w) \to V_+$ is well-defined.
By Lemma \ref{Lem_rigid_D_1}, it is clearly injective.

We next prove that the map $\Gamma_1(w) \to V_+$ is also surjective.
Let $i \in V_+$, then $|i|<a$ holds.
Thus, the first remark yields $w(|i|+1) < |i|+1$,
so there exists some $j \in \{l\} \cup [2,n-1]$ such that $(i,j) \in \Gamma(w)$.
Take the maximum $j$ among such $j$'s.

If $i \ge 2$, then $(i,j)$ belongs to $\Gamma_1(w)$ by Lemma \ref{Lem_rigid_D_1}.

If $i=c$, then $(i,j) \in \Gamma(w)$ and $(i,j+2) \notin \Gamma(w)$ hold.
On the other hand, we obtain $(-i,|j|+1) \notin \Gamma_1(w)$ from (*).
From these, $(i,j)$ must be in $\Gamma_1(w)$.

Therefore, $(i,j) \in \Gamma_1(w)$ holds, 
so the map $\Gamma_1(w) \to V_+$ is also surjective, and thus, bijective.

(2)
We can check the following properties.
\begin{itemize}
\item For $k \in [2,n]$, the condition $w(k) \ge k-1$ holds 
if and only if $k > |b|+1$,
and this condition is also equivalent to $w(k)>|b|$.
\item If $i \in V_-$, then $w(|i|+1)<|i|$ and $w(|i|+2)<|i|+2$ hold.
\end{itemize}
Then, the proof is similar to (1).
\end{proof}

Now, we show Theorem \ref{Thm_brick_D} in the case $l=\pm 1$.

\begin{proof}
By Lemma \ref{Lem_vertex_D_1}, we can define a map $\rho \colon V \to Q_0$ as follows.
\begin{itemize}
\item If $i \in V_+$, then $\rho(i)$ is the unique element $j \in Q_0$ 
such that $(i,j) \in \Gamma_1(w)$.
\item If $i \in V_-$ and $i=\pm 1$, then $\rho(i)$ is the unique element $j \in Q_0$ 
such that $(i,j) \in \Gamma_2(w)$.
\item If $i \in V_-$ and $i \ge 2$, then $\rho(i)$ is the unique element $j \in Q_0$ 
such that $(-i,j) \in \Gamma_2(w)$.
\end{itemize}
Set $\ang{i}:=p(i,\rho(i),l)$ for each $i \in V$.
It suffices to show that $(\ang{i})_{i \in V}$ satisfies the properties (a), (b), and (c),
since the three properties are enough to define a $\Pi$-module.

First, $(\ang{i})_{i \in V}$ is a $K$-basis of $S(w)$ 
by Lemma \ref{Lem_soc_D_1}, 
and $K\ang{i}$ is clearly a subspace of $e_i S(w)$ if $i \ge -1$; 
and of $e_{|i|} S(w)$ if $i \le -2$.
Thus, the property (a) has been proved,
and the property (b) follows from (a).

In the following observation, we fully use Lemma \ref{Lem_rigid_D_1}.

We begin the proof of (c)(i).
First, we assume $2 \in V_+$, and set $j:=\rho(2)$.
\begin{itemize}
\item If $2 \notin R$, then $w(j+1)=c$ and $w(j+2) \ge 3$ hold, 
so we have $(c,j) \in \Gamma_1(w)$ and $(-c,j) \notin \Gamma(w)$.
\item If $2 \in R$, then $w(j+1)=2$ holds, 
so we have $(c,j),(-c,j) \notin \Gamma(w)$.
\end{itemize}
These imply
\begin{align*}
\alpha_1 \ang{2}=\alpha_1 p(2,j,l)=p(c,j,l)+p(-c,j,l) 
=\begin{cases}
\ang{c} & (2 \notin R) \\
0 & (2 \in R)
\end{cases}.
\end{align*}
Second, let $i \in V_+ \setminus \{ \max V_+ \}$ and $i \ge 2$, and set $j:=\rho(i+1)$.
Then, 
\begin{align*}
\alpha_i \ang{i+1}=\alpha_i p(i+1,j,l)=p(i,j,l) 
=\begin{cases}
\ang{i} & (\text{if $i+1 \notin R$, since $(i,j) \in \Gamma_1(w)$)} \\
0 & (\text{if $i+1 \in R$, since $(i,j) \notin \Gamma(w)$)}
\end{cases}.
\end{align*}
Since $l = \pm 1$, we have $r \ge 1$.
Thus, we have proved (c)(i).

Next, we begin the proof of (c)(ii).
Let $i \in V_+ \setminus \{ \max V_+ \}$, and set $j:=\rho(i)$.
Then, 
\begin{align*}
\beta_{|i|+1} \ang{i}=\beta_{|i|+1} p(i,j,l)&=p(|i|+1,j+1,l) \\ 
&=\begin{cases}
0 & (\text{if $i+1 \notin R$, since $(|i|+1,j+1) \notin \Gamma(w)$)} \\
\ang{|i|+1} & (\text{if $i+1 \in R$, since $(|i|+1,j+1) \in \Gamma_1(w)$)}
\end{cases}.
\end{align*}
Since $r \ge 1$, this implies (c)(ii).

Before continuing the proof, we remark the following:
every $i \in V_-$ satisfies $|i| < r$, since $l = \pm 1$.
Thus, if $i \in V_-$, then $|i|+1 \notin R$ is equivalent to $-(|i|+1) \in R$.

We proceed to the proof of (c)(iii).
First, assume $-c \in V_- \setminus \{ \min V_- \}$, and set $j:=\rho(-2)$.
\begin{itemize}
\item If $-2 \notin R$, then $w(j+1)=c$ and $w(j+2)=2$ hold, 
so we have $(c,j) \in \Gamma_1(w)$ and $(-c,j) \notin \Gamma(w)$.
\item If $-2 \in R$, then $w(j+1) \le -2$ and $w(j+2)=c$ hold, 
so we have $(-c,j) \in \Gamma_1(w)$ and $(c,j) \notin \Gamma(w)$.
\end{itemize}
Thus,
\begin{align*}
\alpha_1 \ang{-2}=\alpha_1 p(2,j,l)=p(c,j,l)+p(-c,j,l) =\begin{cases}
\ang{c} & (-2 \notin R) \\
\ang{-c} & (-2 \in R)
\end{cases}.
\end{align*}
Second, let $i \in V_- \setminus \{ \min V_- \}$ and $|i| \ge 2$,
and set $j:=\rho(-(|i|+1))$.
Then, 
\begin{align*}
\alpha_{|i|} \ang{-(|i|+1)}&=\alpha_{|i|} p(|i|+1,j,l)=p(|i|,j,l) \\ 
&=\begin{cases}
\ang{|i|}=\ang{-i} & (\text{if $|i|+1 \in R$, since $(|i|,j) \in \Gamma_1(w)$)} \\
\ang{-|i|}=\ang{i} & (\text{if $|i|+1 \notin R$, since $(|i|,j) \in \Gamma_2(w)$)}
\end{cases}.
\end{align*}
These observations yield (c)(iii).

The remaining task is to check (c)(iv).
Assume $i \in V_-$, and set $j:=\rho(i)$.
Then, 
\begin{align*}
\beta_{|i|+1} \ang{i}&=\beta_{|i|+1} p(i,j,l)=p(|i|+1,j+1,l) \\ 
&=\begin{cases}
\ang{-(|i|+1)} & (\text{if $|i|+1 \in R$, since $(|i|+1,j+1) \in \Gamma_2(w)$)} \\
\ang{|i|+1} & (\text{if $|i|+1 \notin R$, since $(|i|+1,j+1) \in \Gamma_1(w)$)}
\end{cases}.
\end{align*}
Thus, we have obtained (c)(iv).

Now, all the desired properties have been proved.
\end{proof}
 
We next assume that the type $l$ is not $\pm 1$.
We can restate Proposition \ref{Prop_rigid_D_2} as follows.

\begin{Lem}\label{Lem_rigid_D_2}
Let $w \in \jirr W$ be a join-irreducible element of type $l \ne \pm 1$,
and set $c$ as in Theorem \ref{Thm_brick_D}.
If $w(l+1) \le 1$, 
then set $m:=\max\{ k \in [l+1,n] \mid w(k) \le 1\}$ 
and $\epsilon:=(-1)^{m-(l+1)}c$; otherwise, set $\epsilon := 1$.
\begin{itemize}
\item[(1)] 
Assume $(i,j) \in \Gamma[l]$. 
Then $p_\epsilon(i,j,l) \notin I(w)$ holds if and only if 
\begin{align*}
\begin{cases}
i \ge w(j+1) & (w(j+1) \ge 2) \\
i \ge 2 \quad \textup{or} \quad i=w(j+1) & (w(j+1)=\pm 1) \\
i \ge w(j+1)+1 & (w(j+1) \le -2)
\end{cases}.
\end{align*}
\item[(2)]
Consider the subset $\Gamma(w) \subset \Gamma[l]$ consisting of
the elements $(i,j) \in \Gamma[l]$ with $p_\epsilon(i,j,l) \notin I(w)$.
Then the set $\{p_\epsilon(i,j,l) \mid (i,j) \in \Gamma(w)\}$ induces 
a $K$-basis of $J(w)$.
\end{itemize}
\end{Lem}

To express $S(w)$, we define the following set for $k \ge 1$:
\begin{align*}
\Gamma_k(w):=\{ (i,j) \in \Gamma(w) \mid 
\min \{x \ge 1 \mid (i,j+x) \notin \Gamma(w) \} = k \}.
\end{align*}
Moreover, we extend the definition of the path $p_\epsilon(i,j,l)$ to
$\tilde{\Gamma}[l]:=\{ (i,j) \in \pm Q_0 \times \Z \mid i \le j \ge l \}$
by setting $p_\epsilon(i,j,l):=0$ if $j \ge n$, 
and define $w(k):=k$ if $k \ge n+1$.

Then, straightforward calculation yields the relation 
\begin{align}\label{eq_path_lin_comb}
p_{-\epsilon}(i,j,l)=-p_\epsilon(i,j,l)+p_\epsilon(|i|,j+|i|-1,l)
\end{align}
for $(i,j) \in \tilde{\Gamma}[l]$ with $i \le -2$.

\begin{Lem}\label{Lem_soc_D_2}
Let $w \in \jirr W$ be a join-irreducible element of type $l \ne \pm 1$.
Set $R,a,b,r,c$ as in Theorem \ref{Thm_brick_D},
and $\epsilon$ as in Lemma \ref{Lem_rigid_D_2}.
\begin{itemize}
\item[(1)]
Consider the endomorphisms
\begin{align*}
f_1:=(\cdot p_\epsilon(l,l+1,l)) \colon J(w) \to J(w) \quad \text{and} \quad 
f_2:=(\cdot p_\epsilon(-l,l,l)) \colon J(w) \to J(w).
\end{align*}
Then $S(w)=\Ker f_1 \cap \Ker f_2$ holds.
\item[(2)]
Let $(i,j) \in \Gamma(w)$. 
Then $p_{-\epsilon}(i,j,l) \in \Ker f_1$ holds if and only if
$(i,j) \in \Gamma_1(w)$.
\item[(3)]
The set $\{p_{-\epsilon}(i,j,l) \mid (i,j) \in \Gamma_1(w) \}$ 
induces a $K$-basis of $\Ker f_1$.
\item[(4)]
Set $\Lambda_1(w)
:=\{ (i,j) \in \Gamma_1(w) \mid a-1 \ge i \}$.
The set 
$\{p_{-\epsilon}(i,j,l) \mid (i,j) \in \Lambda_1(w) \}$ 
induces a $K$-basis of $S(w)$.
\item[(5)]
Assume $b \le -2$ and $r \ge 1$, and
let $(i,j) \in \Gamma_1(w)$ with $-2 \ge i \ge b+1$.
Then $p_{-\epsilon}(|i|,j+|i|-1,l) \notin \Ker f_1$ holds 
if and only if $|i| \le r$.
In this case, $(|i|,j+|i|-1)$ belongs to $\Gamma_2(w)$.
\item[(6)]
The submodule $\Ker f_1 \cap \Ker f_2$ has a basis formed by
\begin{itemize}
\item $p_\epsilon(i,j,l)$ for each 
$(i,j) \in \Lambda_1(w)$ with $i \ge -1$ and $-r-1 \ge i$; and
\item $p_{-\epsilon}(i,j,l)$ for each 
$(i,j) \in \Lambda_1(w)$ with $-2 \ge i \ge -r$.
\end{itemize}
\end{itemize}
\end{Lem}

\begin{proof}
The part (1) is clear.

The proofs of (2) and (3) are similar to Lemma \ref{Lem_soc_A}.
We remark that 
$f_1(p_{-\epsilon}(i,j,l))=p_\epsilon(i,j+1,l)$
holds in $\Pi$.

(4)
Let $(i,j) \in \Gamma_1(w)$. 
We show that $p_{-\epsilon}(i,j,l) \in \Ker f_2$ holds if and only if $i \le w(l)-1$.

We first assume that $i \ge 2$.
In this case,
$f_2(p_{-\epsilon}(i,j,l))=p_{-\epsilon}(i,j,l) p_\epsilon(-l,l,l)
=p_\epsilon(-i,l+j-i,l)$ hold.
Thus, $f_2(p_{-\epsilon}(i,j,l))=0$ in $J(w)$ holds if and only if
$p_\epsilon(-i,l+j-i,l) \in \Pi$ belongs to $I(w)$.
This is equivalent to $w(l+j-i+1)+1 > -i$ by Lemma \ref{Lem_rigid_D_2},
and also to $\#(R \cap [-n,-i-1]) < j-i+1$.

On the other hand, $(i,j) \in \Gamma_1(w)$ gives
$w(j+2)-1 \ge i \ge w(j+1)$, because $i \ge 2$.
This implies that 
$\#(R \cap [-n,i])=j+1-l$.

Therefore, $f_2(p_{-\epsilon}(i,j,l))=0$ in $J(w)$ holds if and only if
$\#(R \cap [-i,i]) > i-l$.
This condition is equivalent to that $\#(w([1,l]) \cap [-i,i]) < l$.
This exactly means that there exists some $k \in [1,l]$
such that $|w(k)|>i$, and it is equivalent to $a>i$. 

Now, the proof for $i \ge 2$ is complete.

Next, we assume that $i=\pm 1$.
We must show $p_{-\epsilon}(i,j,l) \in \Ker f$.
In this case, 
\begin{align*}
f_2(p_{-\epsilon}(i,j,l)) &= p_{-\epsilon}(i,j,l) p_\epsilon(-l,l,l)
= \alpha_i p_\epsilon(2,j,l) p_\epsilon(-l,l,l)
= \alpha_i p_\epsilon(-2,l+j-2,l) \\
&= \begin{cases}
p_\epsilon(i,l+j-1,l) & (i= \epsilon (-1)^j) \\
0                     & (i=-\epsilon (-1)^j)
\end{cases},
\end{align*}
since $p_\epsilon(-2,l+j-2,l)$ factors through $\epsilon (-1)^{j-1}$.

Thus, we may assume $i=\epsilon(-1)^j$.
First, $p_\epsilon(i,l+j-1,l) \in I(w)$ is equivalent to that 
($w(l+j) \ge 2$ or $w(l+j)=-i$) by Lemma \ref{Lem_rigid_D_2}.
On the other hand, $(i,j) \in \Gamma_1(w)$ implies that $w(j+2) \ge 2$ or $w(j+2)=-i$. 
Since $l \ge 2$, we have ($w(l+j) \ge 2$ or $w(l+j)=-i$),
and $p_\epsilon(i,l+j-1,l) \in I(w)$.

Consequently, $i=\pm 1$ implies that $p_{-\epsilon}(i,j,l) \in \Ker f_2$.

Finally, we assume that $i \le -2$.
Then $p_{-\epsilon}(i,j,l) \in \Ker f_2$ holds,
because the path $p_{-\epsilon}(i,j,l)$ has
$p_{-\epsilon}(1,j,l)$ or $p_{-\epsilon}(-1,j,l)$ in its ending.

Now, we have proved that 
$p_{-\epsilon}(i,j,l) \in \Ker f_2$ holds if and only if $i \le a-1$,
and obtained that 
$f_2(p_{-\epsilon}(i,j,l))=p_\epsilon(-i,l+j-i,l) \ne 0$ in $J(w)$ if
$(i,j) \in \Gamma_1(w)$ and $i \ge a$.

Thus, the set $\{ f_2(p_{-\epsilon}(i,j,l)) \mid 
(i,j) \in \Gamma_1(w) \setminus \Lambda_1(w) \}$ 
is linearly independent in $J(w)$,
so $\{ p_{-\epsilon}(i,j,l) \mid (i,j) \in \Lambda_1(w) \}$
generates $\Ker f_1 \cap \Ker f_2$.
This set is clearly linearly independent in $J(w)$.
Therefore, we obtain the assertion from (1).

(5)
Let $(i,j) \in \Gamma_1(w)$ with $-2 \ge i \ge b+1$.

For the first statement, 
it is easy to see that 
$f_1(p_{-\epsilon}(|i|,j+|i|-1,l))=p_\epsilon(|i|,j+|i|,l)$ in $\Pi$,
so $p_{-\epsilon}(|i|,j+|i|-1,l) \notin \Ker f_1$ precisely means
$p_\epsilon(|i|,j+|i|,l) \notin I(w)$ in $\Pi$.
Lemma \ref{Lem_rigid_D_2} yields that this holds 
if and only if $w(j+|i|+1) \le |i|$, because $|i| \ge 2$.
It is equivalent to $\#(R \cap [-n,|i|]) \ge j+|i|+1-l$.

On the other hand, $(i,j) \in \Gamma_1(w)$ gives
$w(j+2) \ge i=-|i| \ge w(j+1)+1$, because $i \le -2$.
This implies that $\#(R \cap [-n,-|i|-1])=j+1-l$.

Therefore, $f_1(p_{-\epsilon}(|i|,j+|i|-1,l)) \ne 0$ in $J(w)$ holds if and only if
$\#(R \cap [-|i|,|i|]) \ge |i|$.
This exactly means $[1,|i|] \subset \pm R$.
By the definition of the number $r$, it is equivalent to $|i| \le r$.
In this case, $\#(R \cap [-|i|,|i|]) = |i|$.

The first statement has been proved.

Next, we show the second statement, so we assume $|i| \le r$.
It suffices to prove $(|i|,j+|i|) \in \Gamma(w)$ and $(|i|,j+|i|+1) \notin \Gamma(w)$.
We already have $\#(R \cap [-|i|,|i|])=|i|$,
and by the argument above, this yields
$\#(R \cap [-n,|i|])=j+|i|+1-l$.
Thus, we have $w(j+|i|+1) \le |i|$ and $w(j+|i|+2)>|i|$.
Since $|i| \ge 2$, Lemma \ref{Lem_rigid_D_2} implies that 
$(|i|,j+|i|) \in \Gamma(w)$ and $(|i|,j+|i|+1) \notin \Gamma(w)$.
Thus, $(|i|,j+|i|-1)$ belongs to $\Gamma_2(w)$.

(6)
In (4),
$p_{-\epsilon}(i,j,l)=p_\epsilon(i,j,l)$ holds for
each $(i,j) \in \Lambda_1(w)$ with $i \ge -1$.

On the other hand, 
let $(i,j) \in \Lambda_1(w)$ with $i<-r$ and $i \le -2$.
Clearly, $p_{-\epsilon}(i,j,l) \in \Ker f_1$. 
By (5), we have $p_{-\epsilon}(|i|,j+|i|-1,l) \in \Ker f_1$.

If $p_{-\epsilon}(|i|,j+|i|-1,l) \ne 0$ in $J(w)$, then 
$\#(R \cap [-|i|,|i|]) \ge |i|-1$ follows from similar arguments 
to the proof of the first statement of (5). 
This implies $|i| < a$, since $l \ge 2$. 
We have $(|i|,j+|i|-1,l) \in \Lambda_1(w)$.
Thus, in the $K$-basis of $\Ker f_1 \cap \Ker f_2$ given in (4),
we can replace $p_{-\epsilon}(i,j,l)$ to $p_\epsilon(i,j,l)$
to obtain another $K$-basis of $S(w)$.

If $p_{-\epsilon}(|i|,j+|i|-1,l)=0$ in $J(w)$,
then $p_{-\epsilon}(i,j,l)=p_\epsilon(i,j,l)$ holds.

We repeat this procedure, and get that 
the elements in the statement form a $K$-basis of $\Ker f_1 \cap \Ker f_2$.
\end{proof}

The next assertion follows from the definition of $\Lambda_1(w)$.

\begin{Lem}\label{Lem_vertex_D_2}
Let $w \in \jirr W$ be a join-irreducible element of type $l \ne \pm 1$.
Then, there exists a bijection $\Lambda_1(w) \to V$ 
given by $(i,j) \mapsto i$.
\end{Lem}

\begin{proof}
The well-definedness can be checked by Lemma \ref{Lem_soc_D_2}.

We clearly have $\max \{ k \in [l+1,n] \mid w(k) < k \}-1 \ge a-1$.
Then, Lemma \ref{Lem_rigid_D_2} and the definition of $V$ yield that, 
for any $i \in V$, there exists some $j$ such that $(i,j) \in \Gamma(w)$.
Thus, the definition of $\Lambda_1(w)$ and $i \le a-1$ imply that
there uniquely exists $j$ such that $(i,j) \in \Lambda_1(w)$.
This means that the map $\Lambda_1(w) \to V$ is bijective.
\end{proof}

Now, we show Theorem \ref{Thm_brick_D} in the case $l \ne \pm 1$.

\begin{proof}
By Lemma \ref{Lem_soc_D_2}, we can define a map $\rho \colon V \to Q_0$ as follows:
$\rho(i)$ is the unique element $j \in Q_0$ 
such that $(i,j) \in \Gamma_1(w)$.
Set $\ang{i}:=p(i,\rho(i),l)$ for each $i \in V$.
It suffices to show that $(\ang{i})_{i \in V}$ satisfies the properties (a), (b), and (c),
since the three properties are enough to define a $\Pi$-module.

First, $(\ang{i})_{i \in V}$ is a $K$-basis of $S(w)$ 
by Lemma \ref{Lem_vertex_D_2}, 
and $K\ang{i}$ is clearly a subspace of $e_i S(w)$ if $i \ge -1$; 
and of $e_{|i|} S(w)$ if $i \le -2$.
Thus, the property (a) has been obtained,
and the property (b) follows from (a).

In the rest, we fully use Lemma \ref{Lem_rigid_D_2}.

We begin the proof of (c)(i).
First, we assume $2 \in V_+$, and set $j:=\rho(2)$.
\begin{itemize}
\item If $2 \notin R$ and $r \ge 1$, then $w(j+1) = c$.
Thus, $(c,j) \in \Lambda_1(w)$ and $(-c,j) \notin \Gamma(w)$ follow.
\item If $2 \notin R$ and $r=0$, then $w(j+1) \le -2$.
Thus, $(c,j),(-c,j) \in \Lambda_1(w)$ follows.
\item If $2 \in R$, then $w(j+1)=2$.
Thus, $(c,j),(-c,j) \notin \Gamma(w)$ follows.
\end{itemize}
Therefore,
\begin{align*}
\alpha_1 \ang{2}=\alpha_1 p_\epsilon(2,j,l)=p_\epsilon(c,j,l)+p_\epsilon(-c,j,l) 
=\begin{cases}
\ang{c}       & (2 \notin R, \ r \ge 1) \\
\ang{c}+\ang{-c} & (2 \notin R, \ r=0) \\
0          & (2 \in R)
\end{cases}.
\end{align*}

Second, we assume $i \in V_+ \setminus \{ \max V_+ \}$ and $i \ge 2$, 
and set $j:=\rho(i+1)$.
Then,
\begin{align*}
\alpha_i \ang{i+1}=\alpha_i p_\epsilon(i+1,j,l)=p_\epsilon(i,j,l) 
=\begin{cases}
\ang{i} & (\text{if $i+1 \notin R$, since $(i,j) \in \Gamma_1(w)$)} \\
0 & (\text{if $i+1 \in R$, since $(i,j) \notin \Gamma(w)$)}
\end{cases}.
\end{align*}
Thus, we have the property (c)(i).

We begin the proof of (c)(ii).
First, let $c \in V_+ \setminus \{ \max V_+ \}$, and set $j:=\rho(c)$.
In this case, $w(j+1) \le 1$, $w(j+2) \ge 2$, and 
$\epsilon=(-1)^{j-l}c$ hold,
so the path $p_\epsilon(-2,j,l)$ factors through $-c$.
We observe the following properties.
\begin{itemize}
\item If $-2 \notin R$ and $r=0$, then $(-2,j) \in \Lambda_1(w)$; 
otherwise $(-2,j) \notin \Gamma(w)$. 
\item If $2 \in R$, then $(2,j+1) \in \Lambda_1(w)$; 
otherwise $(2,j+1) \notin \Gamma(w)$. 
\end{itemize}
Thus, we have
\begin{align*}
\beta_2 \ang{c} &=\beta_2 p_\epsilon(c,j,l)=p_{-\epsilon}(-2,j,l)
=-p_\epsilon(-2,j,l)+p_\epsilon(2,j+1,l)=\eta_c^- \ang{-2} + \eta_c^+ \ang{2}.
\end{align*}
Second, let $i \in V_+ \setminus \{ \max V_+ \}$ and $i \ge 2$, and set $j:=\rho(i)$.
Then,
\begin{align*}
\beta_{|i|+1} \ang{i}&=\beta_{|i|+1} p_\epsilon(i,j,l)=p_\epsilon(|i|+1,j+1,l) \\ 
&=\begin{cases}
0 & (\text{if $i+1 \notin R$, since $(|i|+1,j+1) \notin \Gamma(w)$)} \\
\ang{|i|+1} & (\text{if $i+1 \in R$, since $(|i|+1,j+1) \in \Gamma_1(w)$)}
\end{cases}.
\end{align*}
These observations imply the property (c)(ii).

We next consider the elements in $V_-$.
In order to observe the actions of the arrows to $\ang{-i}$ ($i \in [2,r]$),
we define sets $\Omega(w)$ and $\Lambda_2(w)$ as 
\begin{align*}
\Omega(w):=\{ (i,j) \in \Lambda_1(w) \mid i \in [-r,-2] \}, \quad
\Lambda_2(w):=
\{ (i,j) \in \Gamma_2(w) \mid i \in [2,r] \}.
\end{align*}
The element $\ang{-i}$ is equal to the path $p_{-\epsilon}(i,j,l)$ with $(i,j) \in \Omega(w)$,
but we want to deal with the paths of the form $p_\epsilon(i',j',l)$.
In the formula (\ref{eq_path_lin_comb}),
$p_{-\epsilon}(i,j,l)$ is a linear combination of
$p_\epsilon(i,j,l)$ and $p_\epsilon(|i|,j+|i|-1,l)$.
By Lemma \ref{Lem_soc_D_2} (5), 
$(|i|,j+|i|-1)$ belongs to $\Lambda_2(w)$.
Moreover, $\phi \colon \Omega(w) \ni (i,j) \mapsto (|i|,j+|i|-1) \in \Lambda_2(w)$
is a bijection.

In the figure for $J(w)$ in Example \ref{Ex_brick_D_2}, the squares 
with positive blue numbers are the elements of $\Lambda_2(w)$,
and that the squares with negative blue numbers are the elements of $\Omega(w)$.

Now, we begin the proof of (c)(iii).
We first assume $-c \in V_- \setminus \{ \min V_- \}$, and set $j:=\rho(-2)$.
\begin{itemize}
\item If $-2 \in R$ and $r \ge 2$, then 
$w(j+1) \le -3$, $w(j+2)=-2$, $w(j+3)=c$, $w(j+4) \ge 3$, and 
$\epsilon=(-1)^{j+2-l}c$ hold,
so the path $p_\epsilon(-2,j,l)$ factors through $-c$, 
and $(-c,j+1) \in \Lambda_1(w)$ follows.
Thus,
\begin{align*}
\alpha_1 \ang{-2} &= \alpha_1 p_{-\epsilon}(-2,j,l) 
= -\alpha_1 p_\epsilon(-2,j,l) + \alpha_1 p_\epsilon(2,j+1,l) \\
&= -p_\epsilon(c,j+1,l)+(p_\epsilon(c,j+1,l)+p_\epsilon(-c,j+1,l)) \\
&= p_\epsilon(-c,j+1,l) = \ang{-c}.
\end{align*}
\item If $-2 \in R$ and $r=0$, then 
$w(j+1) \le -3$, $w(j+2)=-2$, $w(j+3) \ge 3$, and 
$\epsilon=(-1)^{j+1-l}c$ hold,
so the path $p_\epsilon(-2,j,l)$ factors through $c$, 
and $(-c,j+1) \in \Lambda_1(w)$ follows.
Thus,
\begin{align*}
\alpha_1 \ang{-2} &= \alpha_1 p_\epsilon(-2,j,l) = p_\epsilon(-c,j+1,l) = \ang{-c}.
\end{align*}
\item If $-2 \notin R$ and $r \ge 2$, 
then $w(j+1) \le -3$, $w(j+2)=c$, $w(j+3)=2$, and 
$\epsilon=(-1)^{j+1-l}c$ hold,
so the path $p_\epsilon(-2,j,l)$ factors through $c$, 
and $(c,j+1) \in \Lambda_1(w)$ follows.
Thus,
\begin{align*}
\alpha_1 \ang{-2} &= \alpha_1 p_{-\epsilon}(-2,j,l) 
= -\alpha_1 p_\epsilon(-2,j,l) + \alpha_1 p_\epsilon(2,j+1,l) \\
&= -p_\epsilon(-c,j+1,l)+(p_\epsilon(c,j+1,l)+p_\epsilon(-c,j+1,l)) \\
&= p_\epsilon(c,j+1,l) = \ang{c}.
\end{align*}
\item If $-2 \notin R$ and $r=1$, 
then $w(j+1) \le -3$, $w(j+2)=c$, $w(j+3) \ge 3$, and 
$\epsilon=(-1)^{j+1-l}c$ hold,
so the path $p_\epsilon(-2,j,l)$ factors through $c$, 
and $(-c,j+1) \notin \Gamma(w)$ follows.
Thus,
\begin{align*}
\alpha_1 \ang{-2} = \alpha_1 p_\epsilon(-2,j,l) = p_\epsilon(-c,j+1,l) = 0.
\end{align*}
\item If $-2 \notin R$ and $r=0$,  
then $w(j+1) \le -3$, $w(j+2) \ge 2$, and 
$\epsilon=(-1)^{j-l}c$ hold,
so the path $p_\epsilon(-2,j,l)$ factors through $-c$, 
and $(c,j+1) \notin \Gamma(w)$ follows.
Thus,
\begin{align*}
\alpha_1 \ang{-2} &= \alpha_1 p_\epsilon(-2,j,l) = p_\epsilon(c,j+1,l)= 0.
\end{align*}
\end{itemize}
Second, we assume $i \in V_- \setminus \{ \min V_- \}$ and $|i| \ge 2$,
and take the unique $j$ such that $(-(|i|+1),j) \in \Lambda_1(w)$,
and observe the action of the arrow $\alpha_i$ for $\ang{-(|i|+1)} \in S(w)$.
\begin{itemize}
\item If $|i|<r$, then $(-(|i|+1),j) \in \Omega(w)$ and 
$\phi(-(|i|+1),j)=(|i|+1,j+|i|) \in \Lambda_2(w)$ hold, and
\begin{align*}
&\alpha_{|i|} \ang{-(|i|+1)} = \alpha_{|i|} p_{-\epsilon}(-(|i|+1),j,l) 
= -\alpha_{|i|} p_\epsilon(-(|i|+1),j,l) + \alpha_{|i|} p_\epsilon(|i|+1,j+|i|,l) \\
&\qquad = -p_\epsilon(-|i|,j+1,l)+p_\epsilon(|i|,j+|i|,l)\\
&\qquad =\begin{cases}
\ang{-|i|} & (\text{if $-(|i|+1) \in R$, since $(-|i|,j+1) \in \Omega(w)$}) \\
p_\epsilon(|i|,j+|i|,l) & (\text{if $-(|i|+1) \notin R$, 
since $(-|i|,j+1) \notin \Gamma(w)$})
\end{cases}\\
&\qquad =\begin{cases}
\ang{i}   & (\text{if $-(|i|+1) \in R$}) \\
\ang{-i} & (\text{if $-(|i|+1) \notin R$, since $|i|+1 \in R$ and 
$(|i|,j+|i|) \in \Lambda_1(w)$})
\end{cases}.
\end{align*}
\item If $|i| \ge r$, then
\begin{align*}
\alpha_i \ang{-(|i|+1)} &= \alpha_i p_\epsilon(-(|i|+1),j,l) =p_\epsilon(-|i|,j+1,l) \\
&=\begin{cases}
\ang{-|i|}=\ang{i} & (\text{if $-(|i|+1) \in R$, since $(-|i|,j+1) \in \Lambda_1(w)$}) \\
0            & (\text{if $-(|i|+1) \notin R$, since $(-|i|,j+1) \notin \Gamma(w)$}) \\
\end{cases}.
\end{align*}
\end{itemize}
These observations and the definition of $r$ tell us that (c)(iii) holds.

Finally, we would like to show the property (c)(iv).
First, we assume $-c \in V_-$, and set $j:=\rho(-c)$.
\begin{itemize}
\item If $r \ge 2$, then $w(j+1) \le -2$, $w(j+2)=c$, $w(j+3) \ge 2$, and 
$\epsilon=(-1)^{j+1-l}c$ hold,
so the path $p_\epsilon(-2,j,l)$ factors through $c$.
Thus,
\begin{align*}
\beta_2 \ang{-c} &=\beta_2 p_\epsilon(-c,j,l)
=-p_\epsilon(-2,j,l)+p_\epsilon(2,j+1,l) \\
&=\begin{cases}
p_\epsilon(2,j+1,l) & (\text{if $-2 \in R$, since $(-2,j) \notin \Gamma(w)$})\\
p_{-\epsilon}(-2,j,l) & (\text{if $-2 \notin R$}) \\
\end{cases}\\
&=\begin{cases}
\ang{2} & (\text{if $-2 \in R$, since $2 \notin R$ and $(2,j+1) \in \Lambda_1(w)$}) \\
\ang{-2} & (\text{if $-2 \notin R$, since $(-2,j) \in \Omega(w)$}) \\
\end{cases}.
\end{align*}
\item If $r=1$, the path $p_\epsilon(-2,j,l)$ factors through $c$ by the same reason.
Since $r=1$, we have $-2,2 \notin R$, so $(-2,j),(2,j+1) \in \Lambda_1(w)$.
Thus,
\begin{align*}
\beta_2 \ang{-c}&=\beta_2 p_\epsilon(-c,j,l) 
=-p_\epsilon(-2,j,l)+p_\epsilon(2,j+1,l)= -\ang{-2}+\ang{2}.
\end{align*}
\item If $r=0$, then $w(j+1) \le -2$, $w(j+2) \ge 2$, and 
$\epsilon=(-1)^{j-l}c$ hold,
so the path $p_\epsilon(-2,j,l)$ factors through $-c$.
Thus,
\begin{align*}
\beta_2 \ang{-c}&=\beta_2 p_\epsilon(-c,j,l)=p_\epsilon(-2,j,l)\\
&=\begin{cases}
0     & (\text{if $-2 \in R$, since $(-2,j) \notin \Gamma(w)$}) \\
\ang{-2} & (\text{if $-2 \notin R$, since $(-2,j) \in \Lambda_1(w)$})
\end{cases}.
\end{align*}
\end{itemize}
Second, we assume $i \in V_-$, $|i| \ge 2$, and set $j:=\rho(2)$.
We observe the action of the element $\beta_2$ for $\ang{-i} \in S(w)$.
\begin{itemize}
\item If $|i|<r$, then 
$(i,j) \in \Omega(w)$ and $\phi(i,j)=(|i|,j+|i|-1) \in \Lambda_2(w)$ hold, so 
\begin{align*}
&\beta_{|i|+1} \ang{i} = \beta_{|i|+1}p_{-\epsilon}(i,j,l)
= -\beta_{|i|+1}p_\epsilon(i,j,l)+\beta_{|i|+1}p_\epsilon(|i|,j+|i|-1,l) \\ 
&\qquad= -p_\epsilon(-(|i|+1),j,l)+p_\epsilon(|i|+1,j+|i|,l) \\
&\qquad= \begin{cases}
p_\epsilon(|i|+1,j+|i|,l) & (\text{if $-(|i|+1) \in R$, 
since $(-(|i|+1),j) \notin \Gamma(w)$}) \\
p_{-\epsilon}(-(|i|+1),j,l) & (\text{if $-(|i|+1) \notin R$}) \\
\end{cases} \\
&\qquad= \begin{cases}
\ang{|i|+1} & (\text{if $-(|i|+1) \in R$, 
since $|i|+1 \notin R$ and $(|i|+1,j+|i|) \in \Lambda_1(w)$}) \\
\ang{-(|i|+1)} & (\text{if $-(|i|+1) \notin R$, since $(-(|i|+1),j) \in \Omega(w)$}) \\
\end{cases}.
\end{align*}
\item If $|i|=r$, then 
$(i,j) \in \Omega(w)$ and $\phi(i,j)=(|i|,j+|i|-1) \in \Lambda_2(w)$ hold.
Since $|i|=r$, we have $-(|i|+1),|i|+1 \notin R$, 
so $(|i|+1,j+|i|), (-(|i|+1),j) \in \Lambda_1(w)$ hold.
Thus,
\begin{align*}
\beta_{|i|+1} \ang{i} &= \beta_{|i|+1} p_{-\epsilon}(i,j,l)
= -\beta_{|i|+1} p_\epsilon(i,j,l) + \beta_{|i|+1} p_\epsilon(|i|,j+|i|-1,l) \\
&= -p_\epsilon(-(|i|+1),j,l) + p_\epsilon(|i|+1,j+|i|,l)
=-\ang{-(|i|+1)}+\ang{|i|+1}.
\end{align*}
\item If $|i|>r$, then
\begin{align*}
\beta_{|i|+1} \ang{i} &= \beta_{|i|+1} p_\epsilon(i,j,l)=p_\epsilon(-(|i|+1),j,l)\\
&=\begin{cases}
0              & (\text{if $-(|i|+1) \in R$, since $(-(|i|+1),j) \notin \Gamma(w)$}) \\
\ang{-(|i|+1)} & (\text{if $-(|i|+1) \notin R$, 
since $(-(|i|+1),j) \in \Lambda_1(w)$})
\end{cases}.
\end{align*}
\end{itemize}
The property (c)(iv) follows from these observations and the definition of $r$.

Now, all the proof is complete.
\end{proof}

\section{Description of semibricks}\label{Sec_semibrick}

\subsection{Canonical join representations in Coxeter groups}\label{Subsec_can_join}

Let $\Delta$ be a Dynkin diagram $\bbA_n$ or $\bbD_n$,
and $\Pi$ and $W$ be the corresponding preprojective algebra and 
the Coxeter group, respectively.
We obtained a canonical bijection $S(\bullet) \colon W \to \sbrick \Pi$
in Proposition \ref{Prop_W_sbrick}.
The aim of this section is to give the explicit description of this map.
In the previous section, this aim has been achieved for the restricted bijection
$S(\bullet) \colon \jirr W \to \brick \Pi$.
To extend this to all elements in $W$, 
it is enough to determine the canonical join representations in $W$ for 
$\Delta=\bbA_n,\bbD_n$.

It would be difficult to prove that a set of join-irreducible elements
gives a canonical join representation of a given element in $W$ 
by directly checking the conditions in Definition \ref{Def_can_join}.
Fortunately, Reading \cite{Reading} has obtained a nice property 
characterizing canonical join representations in finite Coxeter groups.
To explain this, we prepare some notation.

Let $\Delta_0$ be the vertices set of $\Delta$.
Then, $W$ has the canonical generators $\{s_i \mid i \in \Delta_0 \}$.
For each $w \in W$, set $\des(w)$ and $\cov(w)$ 
as the set of \textit{descents} and the set of \textit{cover reflections} of $w$, respectively: that is,
\begin{align*}
\des(w):=\{i \in \Delta_0 \mid ws_i<w \}, \quad
\cov(w):=\{ws_iw^{-1} \mid i \in \des(w) \}.
\end{align*}
There is a natural bijection $\des(w) \to \cov(w)$
defined by $i \mapsto ws_iw^{-1}$.
Using the set $\cov(w)$,
we can write the canonical join representation of $w$ as follows.

\begin{Prop}\label{Prop_cover_ref}\cite[Theorem 10-3.9]{Reading}
Let $w \in W$.
For each $t \in \cov(w)$, 
the set $\{v \in W \mid v \le w, \ t \in \inv(v) \}$ 
has a unique minimal element $w_t$.
Moreover, $\bigvee_{t \in \cov(w)} w_t$ is the canonical join representation of $w$.
\end{Prop}

Hence, we have the following way to find canonical join representations.

\begin{Prop}\label{Prop_join_rep_suff}
Let $w \in W$.
Assume that, for each $d \in \des(w)$, 
there exists a join-irreducible element $w_d \in \jirr W$ satisfying
$w_d \le w$ and $\cov(w_d)=\{ ws_dw^{-1} \}$. 
Then $\bigvee_{d \in \des(w)} w_d$ is the canonical join representation of $w$.
\end{Prop}

\begin{proof}
Let $d \in \des(w)$ and set $t:=w s_d w^{-1} \in \cov(w)$.
By Proposition \ref{Prop_cover_ref}, it suffices to show that $w_d$ is a minimal element of 
$V:=\{v \in W \mid v \le w, \ t \in \inv(v) \}$.
We assume that $v \in V$ satisfies $v < w_d$ and deduce a contradiction.

Since $w_d s_d=t w_d$, we get 
$l(t \cdot w_d s_d)=l(t \cdot t w_d)=l(w_d)>l(w_d s_d)$.
Thus, $t \notin \inv(w_d s_d)$.

On the other hand, the inequality $v \le w_d s_d$ holds, 
since $w_d$ is a join-irreducible element with its unique descent $d$.
Thus, we have $\inv(v) \subset \inv(w_d s_d)$.
By assumption, $t$ belongs to $\inv(v)$,
so $t$ must be in $\inv(w_d s_d)$.

These two results contradict to each other.
Thus, there exists no $v \in V$ such that $v<w_d$.
This exactly means that $w_d$ is a minimal element of $V$.
\end{proof}

Before proceeding to the next subsection, 
we give an example of canonical join representations.
We recall that the \textit{Hasse quiver} of $W$ is defined as follows.
\begin{itemize}
\item The vertices are the elements of $W$.
\item For any $w,w' \in W$, we write an arrow $w \to w'$ if and only if 
$w > w'$ holds and there exists no $v \in W$ such that $w > v > w'$. 
\end{itemize}

\begin{Ex}
Let $\Delta=\bbA_3$.
Then, the Hasse quiver of $W$ is 
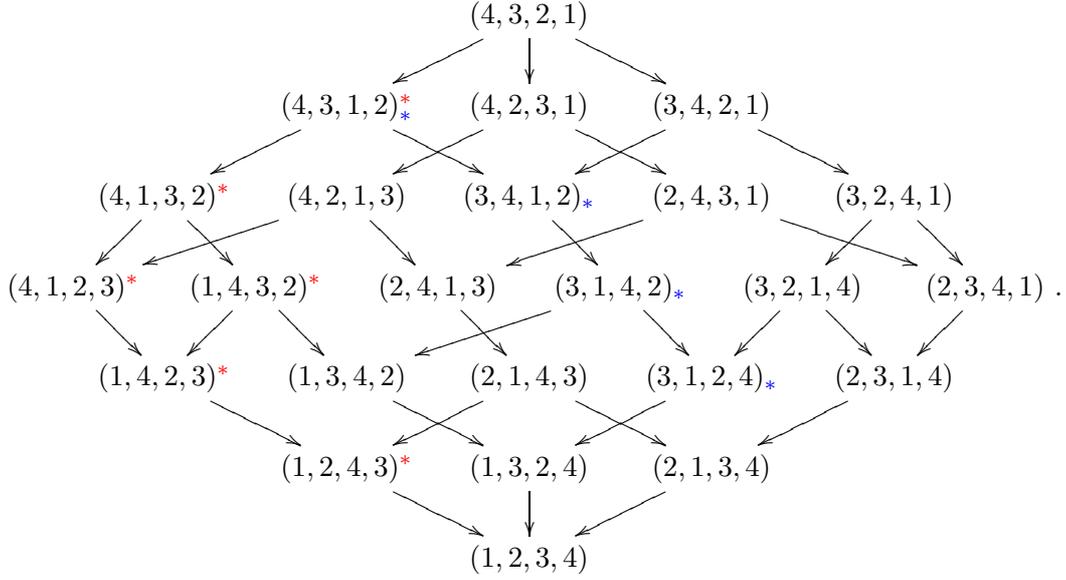
\begin{align*}
\begin{xy}
(  0,-36) *+{(1,2,3,4)}="1234",
( 24,-24) *+{(2,1,3,4)}="2134",
(  0,-24) *+{(1,3,2,4)}="1324",
(-24,-24) *+{(1,2,4,3)^{\textcolor{red}{*}}}="1243",
( 48,-12) *+{(2,3,1,4)}="2314",
( 24,-12) *+{(3,1,2,4)_{\textcolor{blue}{*}}}="3124",
(  0,-12) *+{(2,1,4,3)}="2143",
(-24,-12) *+{(1,3,4,2)}="1342",
(-48,-12) *+{(1,4,2,3)^{\textcolor{red}{*}}}="1423",
( 60,  0) *+{(2,3,4,1)}="2341",
( 36,  0) *+{(3,2,1,4)}="3214",
( 12,  0) *+{(3,1,4,2)_{\textcolor{blue}{*}}}="3142",
(-12,  0) *+{(2,4,1,3)}="2413",
(-36,  0) *+{(1,4,3,2)^{\textcolor{red}{*}}}="1432",
(-60,  0) *+{(4,1,2,3)^{\textcolor{red}{*}}}="4123",
( 48, 12) *+{(3,2,4,1)}="3241",
( 24, 12) *+{(2,4,3,1)}="2431",
(  0, 12) *+{(3,4,1,2)_{\textcolor{blue}{*}}}="3412",
(-24, 12) *+{(4,2,1,3)}="4213",
(-48, 12) *+{(4,1,3,2)^{\textcolor{red}{*}}}="4132",
( 24, 24) *+{(3,4,2,1)}="3421",
(  0, 24) *+{(4,2,3,1)}="4231",
(-24, 24) *+{(4,3,1,2)^{\textcolor{red}{*}}_{\textcolor{blue}{*}}}="4312",
(  0, 36) *+{(4,3,2,1)}="4321",
\ar "2134";"1234"
\ar "1324";"1234"
\ar "1243";"1234"
\ar "2314";"2134"
\ar "2143";"2134"
\ar "3124";"1324"
\ar "1342";"1324"
\ar "2143";"1243"
\ar "1423";"1243"
\ar "2341";"2314"
\ar "3214";"2314"
\ar "3214";"3124"
\ar "3142";"3124"
\ar "2413";"2143"
\ar "3142";"1342"
\ar "1432";"1342"
\ar "1432";"1423"
\ar "4123";"1423"
\ar "3241";"2341"
\ar "2431";"2341"
\ar "3241";"3214"
\ar "3412";"3142"
\ar "2431";"2413"
\ar "4213";"2413"
\ar "4132";"1432"
\ar "4213";"4123"
\ar "4132";"4123"
\ar "3421";"3241"
\ar "4231";"2431"
\ar "3421";"3412"
\ar "4312";"3412"
\ar "4231";"4213"
\ar "4312";"4132"
\ar "4321";"3421"
\ar "4321";"4231"
\ar "4321";"4312"
\end{xy}.
\end{align*}
We determine the canonical join representation of the element $w:=(4,3,1,2)$
from the Hasse quiver.
In this case, we have $\des(w)=\{ 1,2 \}$ and $\cov(w)=\{ (4 \quad 3), (3 \quad 1) \}$.
Thus, we consider the following sets:
\begin{itemize}
\item $\{v \in W \mid v \le w, \ (4 \quad 3) \in \inv(v) \}$,
whose elements are indicated by $^{\textcolor{red}{*}}$; and
\item $\{v \in W \mid v \le w, \ (3 \quad 1) \in \inv(v) \}$,
whose elements are indicated by $_{\textcolor{blue}{*}}$.
\end{itemize}
These sets have $(1,2,4,3)$ and $(3,1,2,4)$ as their unique minimum elements, respectively.
By Proposition \ref{Prop_cover_ref}, the canonical join representation of $w$ is
$(1,2,4,3) \vee (3,1,2,4)$.
We also remark 
that $\cov((1,2,4,3))=\{(4 \quad 3)\}$ and $\cov((3,1,2,4))=\{(3 \quad 1)\}$ hold.
\end{Ex}

\subsection{Type $\bbA_n$}

Let $\Delta=\bbA_n$.
For each element $w$ in $\jirr W$ of type $l$, we set
\begin{align*}
L(w) := w([1,l]), \quad R(w) := w([l+1,n+1]).
\end{align*} 
It is easy to see that the correspondence $w \mapsto R(w)$ is injective.

The following procedure gives the canonical join representation
of a given element of the Coxeter group $W$.
This coincides with \cite[Theorem 10-5.6]{Reading}.

\begin{Prop}\label{Prop_decompose_A}
Let $w \in W$, and set $a_d:=w(d)$, $b_d:=w(d+1)$ for each $d \in \des(w)$.
Then the canonical join representation of $w$ is $\bigvee_{d \in \des(w)} w_d$, 
where $w_d \in \jirr W$ is the unique join-irreducible element 
such that $R(w_d)$ coincides with $R_d$ defined as follows:
\begin{align*}
X_d:=w([d+1,n+1]), \quad
R_d:=([b_d,a_d-1] \cap X_d) \cup [a_d+1,n+1].
\end{align*}
\end{Prop}

\begin{proof}
Let $d \in \des(w)$.
It is easy to see that there uniquely exists $w_d \in \jirr W$ 
with $R(w_d)=R_d$.
Then, $L(w_d)=[1,b_d-1] \cup ([b_d+1,a_d] \setminus X_d)$.
From this, we can straightforwardly check that $\inv(w_d) \subset \inv(w)$,
which is equivalent to $w_d \le w$.
Moreover, the unique cover reflection of $w_d$ is $(a_d \quad b_d)$,
and it is equal to $w s_d w^{-1}$.
Therefore, the assertion follows from Proposition \ref{Prop_join_rep_suff}.
\end{proof}

\begin{Ex}\label{Ex_decompose_A}
Let $n:=8$ and $w:=(4,9,3,6,2,8,5,1,7)$.
Then we have $\des(w)=\{ 2,4,6,7 \}$.
The canonical join representation of $w$ is 
$\bigvee_{d \in \des(w)} w_d$, where $w_d$ is given as follows for each $d \in \des(w)$.
\begin{center}
\begin{tabular}{c|cc|cc}
$d$ & $a_d$ & $b_d$ & $R(w_d)$ & $w_d$ \\
\hline
2 & 9 & 3 & $\{3,5,6,7,8\}$ & $(1,2,4,9,3,5,6,7,8)$ \\
4 & 6 & 2 & $\{2,5,7,8,9\}$ & $(1,3,4,6,2,5,7,8,9)$ \\
6 & 8 & 5 & $\{5,7,9\}$     & $(1,2,3,4,6,8,5,7,9)$ \\
7 & 5 & 1 & $\{1,6,7,8,9\}$ & $(2,3,4,5,1,6,7,8,9)$ \\
\end{tabular}
\end{center}
\end{Ex}

Combining Corollary \ref{Cor_brick_abbr_A} and Proposition \ref{Prop_decompose_A},
we can obtain the semibrick $S(w)$ directly.

\begin{Thm}\label{Thm_sbrick_A}
Let $w \in W$.
Then, the semibrick $S(w)$ is $\bigoplus_{d \in \des(w)}S_d$,
where $S_d$ is the brick whose abbreviated description  
as in Corollary \ref{Cor_brick_abbr_A} is given as follows.
\begin{itemize}
\item Set $R_d$ as in Proposition \ref{Prop_decompose_A},
and $a_d:=w(d)$, $b_d:=w(d+1)$, $V_d := [b_d, a_d-1]$.  
\item The brick $S_d$ has a $K$-basis $(\ang{i}_d)_{i \in V_d}$,
where $\ang{i}_d$ belongs to $e_iS_d$.
\item For each $i \in V_d$, place a symbol $i$ denoting 
the $K$-vector subspace $K\ang{i}_d$.
\item For each $i \in V_d \setminus \{\max V_d\}$,
we write exactly one arrow between $i$ and $i+1$,
and the orientation is $i \to i+1$ if $i+1 \in R_d$ and 
$i \leftarrow i+1$ if $i+1 \notin R_d$.
\end{itemize}
\end{Thm}

\begin{proof}
For each $d \in \des(w)$, let $w_d$ be the join-irreducible element
in the canonical join representation given in Proposition \ref{Prop_decompose_A}.
Then, we can check that the abbreviated description of $S(w_d)$ 
in Corollary \ref{Cor_brick_abbr_A} coincides with the statement.  
\end{proof}

In Theorem \ref{Thm_sbrick_A},
we remark that $R_d$ can be replaced by $R_d \cap V_d=[b_d,a_d-1] \cap X_d$.

\begin{Ex}
Let $n:=8$ and $w:=(4,9,3,6,2,8,5,1,7)$ as in Example \ref{Ex_decompose_A}.
Then the semibrick $S(w)$ is the direct sum of the following bricks:
\begin{align*}
S_2 &= \phantom{1 \rightarrow 2 \rightarrow {}} 
3 \leftarrow 4 \rightarrow 5 \rightarrow 6 \rightarrow 7 \rightarrow 8, \\
S_4 &= \phantom{1 \rightarrow {}}
2 \leftarrow 3 \leftarrow 4 \rightarrow 5
\phantom{{} \leftarrow 6 \leftarrow 7 \leftarrow 8}, \\
S_6 &= \phantom{1 \rightarrow 2 \rightarrow 3 \rightarrow 4 \rightarrow{}}
5 \leftarrow 6 \rightarrow 7
\phantom{{} \leftarrow 8}, \\
S_7 &= 1 \leftarrow 2 \leftarrow 3 \leftarrow 4
\phantom{{} \leftarrow 5 \leftarrow 6 \leftarrow 7 \leftarrow 8}. 
\end{align*}
\end{Ex}

\subsection{Type $\bbD_n$}

Let $\Delta=\bbD_n$.
For each element $w$ in $\jirr W$ of type $l$, we set
\begin{align*}
L(w) := \{ |w(k)| \mid k \in [1,|l|] \},  \quad
R(w) := w([|l|+1,n]).
\end{align*} 
As in the case of type $\bbA_n$, 
it is easy to see that the correspondence $w \mapsto R(w)$ is injective.

The canonical join representations 
of the elements of the Coxeter group $W$ are given by the following procedure. 

\begin{Prop}\label{Prop_decompose_D}
Let $w \in W$, and set $a_d:=w(d)$, $b_d:=w(|d|+1)$,
$X_d:=w([|d|+1,n])$ for each $d \in \des(w)$.
Then the canonical join representation of $w$ is $\bigvee_{d \in \des(w)} w_d$, 
where $w_d \in \jirr W$ is the unique join-irreducible element 
such that $R(w_d)$ coincides with $R_d$ defined as follows.
\begin{itemize}
\item[(A)]
If $a_d+b_d<0$ and $w([1,|d|]) \subset \pm[a_d,n]$, then
\begin{align*}
R_d 
&= \begin{cases}
\{-a_d\} \cup (\pm[1,a_d-1] \cap X_d) \cup ([a_d+1,-b_d-1] \setminus (-X_d)) 
\cup [-b_d+1,n] & (a_d>0) \\
([-a_d,-b_d-1] \setminus (-X_d)) \cup [-b_d+1,n] & (a_d<0)
\end{cases}.
\end{align*}
\item[(B)]
Otherwise, 
\begin{align*}
R_d 
&= \begin{cases}
([b_d,a_d-1] \cap X_d) \cup [a_d+1,n] & (a_d+b_d>0) \\
([b_d,a_d-1] \cap X_d) \cup ([a_d+1,-b_d-1] \setminus (-X_d)) \cup [-b_d+1,n] & (a_d+b_d<0)
\end{cases}
\end{align*}
\end{itemize}
\end{Prop}

\begin{proof}
The proof is similar to the one for type $\bbA_n$.
In this case, the set $L(w_d)$ is given as follows.
\begin{itemize}
\item[(A)]
If $a_d+b_d<0$ and $w([1,|d|]) \subset \pm [a_d,n]$, then
\begin{align*}
L(w_d)=
\begin{cases}
[a_d+1,-b_d] \cap (-X_d) & (a_d>0) \\
[1,-a_d-1] \cup ([-a_d+1,-b_d] \cap (-X_d)) & (a_d<0) \\
\end{cases}.
\end{align*}
\item[(B)]
Otherwise, 
\begin{align*}
L(w_d)=
\begin{cases}
[1,b_d-1] \cup ([b_d+1,a_d] \setminus X_d) & (b_d > 0) \\
([1,-b_d-1] \setminus (\pm X_d)) \cup ([-b_d+1,a_d] \setminus X_d) 
& (b_d < 0, \ a_d+b_d>0) \\
[1,a_d] \setminus (\pm X_d) & (a_d+b_d<0)
\end{cases}.
\end{align*}
\end{itemize}
By using these, we can check $w_d \le w$ and $\cov(w_d)=\{w s_d w^{-1} \}$.
\end{proof}

In the rest, the symbols (A) and (B) mean
the conditions (A) and (B) in Proposition \ref{Prop_decompose_D}, respectively.

\begin{Ex}\label{Ex_decompose_D}
Let $n:=9$ and $w:=(5,3,-7,4,-6,-8,9,-1,2)$.
Then we have $\des(w)=\{ 1,2,4,5,7 \}$.
The canonical join representation of $w$ is 
$\bigvee_{d \in \des(w)} w_d$, where $w_d$ is given as follows for each $d \in \des(w)$.
\begin{center}
\begin{tabular}{c|ccc|cc}
$d$ & $a_d$ & $b_d$ & (A) or (B) & $R(w_d)$ & $w_d$ \\
\hline
1 &  $5$ &  $3$ & (B) & $\{3,4,6,7,8,9\}$     & 
$(\phantom{-}1,\phantom{-}2,\phantom{-}5,\phantom{-}3,\phantom{-}4,\phantom{-}6,
\phantom{-}7,\phantom{-}8,\phantom{-}9)$ \\
2 &  $3$ & $-7$ & (A) & $\{-3,-1,2,4,5,8,9\}$ & 
$(\phantom{-}6,\phantom{-}7,-3,-1,\phantom{-}2,\phantom{-}4,\phantom{-}5,\phantom{-}8,
\phantom{-}9)$ \\
4 &  $4$ & $-6$ & (B) & $\{-6,-1,2,5,7,8,9\}$ & 
$(\phantom{-}3,\phantom{-}4,-6,-1,\phantom{-}2,\phantom{-}5,\phantom{-}7,\phantom{-}8,
\phantom{-}9)$ \\
5 & $-6$ & $-8$ & (A) & $\{6,7,9\}$           & 
$(\phantom{-}1,\phantom{-}2,\phantom{-}3,\phantom{-}4,\phantom{-}5,\phantom{-}8,
\phantom{-}6,\phantom{-}7,\phantom{-}9)$ \\
7 &  $9$ & $-1$ & (B) & $\{-1,2\}$            & 
$(-3,\phantom{-}4,\phantom{-}5,\phantom{-}6,\phantom{-}7,\phantom{-}8,\phantom{-}9,
-1,\phantom{-}2)$ 
\end{tabular}
\end{center}
\end{Ex}

Now, by combining Corollary \ref{Cor_brick_abbr_D} and Proposition \ref{Prop_decompose_D}, 
we can obtain the semibrick $S(w)$ from $w \in W$ directly.
We need to define a few notations: for integers $a>b$ and $c=\pm 1$, we set
\begin{align*}
(V_-(a,b,c),V_+(a,b,c)):=
\begin{cases}
(\emptyset, [b,a-1]) & (b \ge 2) \\
(\emptyset, \{c\} \cup [2,a-1]) & (b = \pm 1) \\
([b+1,-2] \cup \{-c\}, \{c\} \cup [2,a-1]) & (b \le -2)
\end{cases}.
\end{align*}

\begin{Thm}\label{Thm_sbrick_D}
Let $w \in W$.
Then, the semibrick $S(w)$ is $\bigoplus_{d \in \des(w)}S_d$,
where $S_d$ is the brick whose abbreviated description  
as in Corollary \ref{Cor_brick_abbr_A} is given as follows.
\begin{itemize}
\item Set $R_d$ as in Proposition \ref{Prop_decompose_D}, and $a_d:=w(d)$, $b_d:=w(d+1)$, 
\begin{align*}
& r_d:=\max \{ k \ge 0 \mid [1,k] \subset \pm R_d \}, \quad
c_d:=\begin{cases}
w^{-1}(|w(1)|)& (r_d \ge 1) \\
1 & (r_d=0)
\end{cases}, \\
& ((V_-)_d,(V_+)_d) :=
\begin{cases}
(V_-(-b,-a,c),V_+(-b,-a,c)) & (\textup{(A)}) \\
(V_-(a,b,c),V_+(a,b,c)) & (\textup{(B)}) \\
\end{cases}, \quad
V_d := (V_+)_d \amalg (V_-)_d.
\end{align*}
\item The brick $S_d$ has a $K$-basis $(\ang{i}_d)_{i \in V_d}$,
where $\ang{i}_d$ belongs to $e_iS_d$ if $i \ge -1$,
and $e_{|i|}S_d$ if $i \le -2$.
\item For each $i \in V_d$, place a symbol $i$ denoting 
the $K$-vector subspace $K\ang{i}_d$.
\item We write the following arrows.
\begin{itemize}
\item[(i)]
For each $i \in (V_d)_+ \setminus \{\max (V_d)_+\}$,
draw an arrow $i \to |i|+1$ if $|i|+1 \in R_d$; and $i \leftarrow |i|+1$ otherwise.
\item[(ii)]
For each $i \in (V_d)_- \setminus \{\min (V_d)_-\}$,
draw an arrow $i \leftarrow -(|i|+1)$ if $-(|i|+1) \in R_d$; 
and $i \to -(|i|+1)$ otherwise.
\item[(iii)]
If $r_d \ge 1$, then for each $i \in (V_d)_-$ with $|i| \le r_d$,
draw an arrow $-i \leftarrow -(|i|+1)$ if $|i|+1 \in R_d$;
and $i \to |i|+1$ otherwise.
\item[(iv)]
If $r_d=0$, then
draw an arrow $-c \leftarrow 2$ if $2 \notin R_d$, 
and an arrow $c \to -2$ if $-2 \notin R_d$.
\end{itemize}
\end{itemize}
\end{Thm}

\begin{proof}
Apply Corollary \ref{Cor_brick_abbr_D} 
to the element $w_d \in \jirr W$ defined in Proposition \ref{Prop_decompose_D}
for each $d \in \des(w)$.
\end{proof}

\begin{Ex}
Let $n:=9$ and $w:=(5,3,-7,4,-6,8,9,-1,2)$ as in Example \ref{Ex_decompose_D}.
Then the semibrick $S(w)$ is the direct sum of the following bricks:
\begin{align*}
S_1&= \begin{xy}
(  0, 0) *+{\phantom{-1}},
( 24, 0) *+{ 3} = "3+",
( 36, 0) *+{ 4} = "4+",
( 84, 0) *+{\phantom{8}},
\ar "3+";"4+"
\end{xy}, \\
S_2&= \begin{xy}
(  0, 5) *+{ 1} = "1-",
( 12, 5) *+{-2} = "2-",
(  0,-5) *+{-1} = "1+",
( 12,-5) *+{ 2} = "2+",
( 24,-5) *+{ 3} = "3+",
( 36,-5) *+{ 4} = "4+",
( 48,-5) *+{ 5} = "5+",
( 60,-5) *+{ 6} = "6+",
( 84, 0) *+{\phantom{8}},
\ar "1+";"2+"
\ar "3+";"2+"
\ar "3+";"4+"
\ar "4+";"5+"
\ar "6+";"5+"
\ar "1-";"2-" \ar "2-";"1+"
              \ar "2-";"3+"
\end{xy}, \\
S_4&=\begin{xy}
(  0, 6) *+{ 1} = "1-",
( 12, 6) *+{-2} = "2-",
( 24, 6) *+{-3} = "3-",
( 36, 6) *+{-4} = "4-",
( 48, 6) *+{-5} = "5-",
(  0,-6) *+{-1} = "1+",
( 12,-6) *+{ 2} = "2+",
( 24,-6) *+{ 3} = "3+",
( 84, 0) *+{\phantom{8}},
\ar "1+";"2+"
\ar "3+";"2+"
\ar "1-";"2-" \ar "2-";"1+"
\ar "2-";"3-" \ar "2-";"3+"
\ar "3-";"4-"
\ar "4-";"5-"
\end{xy}, \\
S_5&=\begin{xy}
(  0, 0) *+{\phantom{-1}},
( 60, 0) *+{ 6} = "6+",
( 72, 0) *+{ 7} = "7+",
( 84, 0) *+{\phantom{8}},
\ar "6+";"7+"
\end{xy}, \\
S_7&=\begin{xy}
(  0, 0) *+{-1} = "1+",
( 12, 0) *+{ 2} = "2+",
( 24, 0) *+{ 3} = "3+",
( 36, 0) *+{ 4} = "4+",
( 48, 0) *+{ 5} = "5+",
( 60, 0) *+{ 6} = "6+",
( 72, 0) *+{ 7} = "7+",
( 84, 0) *+{ 8} = "8+",
\ar "1+";"2+"
\ar "3+";"2+"
\ar "4+";"3+"
\ar "5+";"4+"
\ar "6+";"5+"
\ar "7+";"6+"
\ar "8+";"7+"
\end{xy}.
\end{align*}
\end{Ex}

\appendix
\section{Example: The bricks over the preprojective algebra of type $\bbD_5$}

In this section, 
we give the list of bricks over the preprojective algebra of type $\bbD_5$.

For the preparation,
we first define two notions denoted by $\sigma(w)$ and $\chi(w)$
associated to each join-irreducible element $w \in \jirr W$ 
in the Coxeter group $W=W(\bbD_n)$ of type $\bbD_n$
(it is not needed to assume $n=5$ here),
and then list all the join-irreducible elements and the corresponding bricks 
by using these notions in the case $n=5$.

First, we define $\sigma(w)$.
Recall that we have defined the integers $a,b,r$ 
in Subsection \ref{subsec_brick_D} for $w$.
By using these integers,
we define $\sigma(w)$ of $w$ as the triple $(a,b,r') \in \Z^3$, where  
\begin{align*}
r':=\begin{cases}
0 & (b \ge -1) \\
\min\{r,|b|-1\} & (b \le -2)
\end{cases},
\end{align*}
and call $\sigma(w)$ the \textit{shape} of $w$.
For any $\sigma \in \Z^3$,
we write $(\jirr W)_\sigma \subset \jirr W$ for the subset of elements in $\jirr W$
whose shapes are $\sigma$.
It is easy to see that $(\jirr W)_\sigma \ne \emptyset$
if and only if $\sigma$ is a triple $(a,b,r')$ satisfying one of the following conditions
(a), (b), (c):
\begin{itemize}
\item[(a)]
$2 \le a \le n$, $-1 \le b < a$, $b \ne 0$, $r'=0$; or
\item[(b)]
$2 \le a \le n$, $-a < b \le -2$, $0 \le r' \le |b|-1$; or
\item[(c)]
$2 \le a \le n$, $-n \le b < -a$, $0 \le r' \le |a|-2$.
\end{itemize}

Next, we define the other notion $\chi(w)$ by using $R$ defined in 
Subsection \ref{subsec_brick_D} for $w$.
We set $\chi(w)$ as the sequence 
$(x(1),x(2),\ldots,x(n)) \in \{0,1,2\}^n$ whose terms are given by
\begin{align*}
x(i):=\begin{cases}
0 & (-j,j \notin R) \\
1 & (-j \in R) \\
2 & (j \in R)
\end{cases}.
\end{align*}
We have a map $\chi \colon \jirr W \to \{0,1,2\}^n$, which is clearly injective.
For $\sigma=(a,b,r')$ satisfying the condition above, we can straightforwardly check
$\chi((\jirr W)_\sigma)=\prod_{i=1}^n X_i$, where $X_i$ is defined as follows
in each of the three cases (a), (b), and (c):
\begin{itemize}
\item[(a)]
\begin{tabular}{|c||c|c|c|c|c|}
\hline
$i$ & $i<|b|$ & $i=|b|$ & $|b|<i<a$ & $i=a$ & $i>a$ \\
\hline
$X_i$ & $\{0\}$ & $\{1\}$ if $b=-1$; $\{2\}$ otherwise & $\{0,2\}$ & $\{0\}$ & $\{2\}$ \\
\hline
\end{tabular},
\item[(b)]
\begin{tabular}{|c||c|c|c|c|c|c|c|}
\hline
$i$ & $i \le r'$ & $i=r'+1 \ne |b|$ & $r'+1<i<|b|$ & $i=|b|$ & $|b|<i<a$ & $i=a$ & $i>a$ \\
\hline
$X_i$ & $\{1,2\}$ & $\{0\}$ & $\{0,1,2\}$ & $\{1\}$ & $\{0,2\}$ & $\{0\}$ & $\{2\}$ \\
\hline
\end{tabular},
\item[(c)]
\begin{tabular}{|c||c|c|c|c|c|c|c|}
\hline
$i$ & $i \le r'$ & $i=r'+1$ & $r'+1<i<a$ & $i=a$ & $a<i<|b|$ & $i=|b|$ & $i>a$ \\
\hline
$X_i$ & $\{1,2\}$ & $\{0\}$ & $\{0,1,2\}$ & $\{0\}$ & $\{1,2\}$ & $\{1\}$ & $\{2\}$ \\
\hline
\end{tabular}.
\end{itemize}
%
%
Therefore, by setting $x:=\max\{a,|b|\}$ and $y:=\min\{a,|b|\}$, we have
\begin{align*}
\#(\jirr W)_\sigma=
\begin{cases}
2^{x-y-1} & (b \ge -1) \\
2^{r'} \cdot 3^{\max\{y-r'-2,0\}} \cdot 2^{x-y-1} & (b \le -2) 
\end{cases}.
\end{align*}

From now on, we consider $\bbD_5$, so let $n=5$.
For $\sigma$ satisfying the condition above,
the following lists show all the elements $w$ in $(\jirr W)_\sigma$ 
and the corresponding bricks $S(w)$ over the preprojective algebra $\Pi$ of type $\bbD_5$.
The elements in $(\jirr W)_\sigma$ are arranged so that
$w$ comes before $w'$ if and only if $\chi(w) < \chi(w')$ in the 
lexicographical order in $\{0,1,2\}^n$,
and each $w$ is shortly denoted by a string $e_1e_2 \cdots e_n$,
where $e_i:=w(i)$ if $w(i)>0$; $e_i:=\underline{w(i)}$ if $w(i)<0$.
For example, $\underline{1}2\underline{5}34$ means $(-1,2,-5,3,4)$.
The join-irreducible elements and the bricks are explicitly described as follows by 
Corollary \ref{Cor_brick_abbr_D}:

\begin{itemize}

\item $\sigma=(2,-5,0)$ (4 elements): {\footnotesize \begin{align*}
S(\underline{1}2\underline{5}\underline{4}\underline{3}) &= \begin{xy} (0,-3) *+{1} = "1+", (0,3) *+{-1} = "1-", (10,3) *+{-2} = "2-", (20,3) *+{-3} = "3-", (30,3) *+{-4} = "4-", \ar "1-";"2-" \ar "3-";"2-" \ar "4-";"3-" \ar "1+";"2-" \end{xy},
& S(12\underline{5}\underline{3}4) &= \begin{xy} (0,-3) *+{1} = "1+", (0,3) *+{-1} = "1-", (10,3) *+{-2} = "2-", (20,3) *+{-3} = "3-", (30,3) *+{-4} = "4-", \ar "1-";"2-" \ar "3-";"2-" \ar "3-";"4-" \ar "1+";"2-" \end{xy},
\\ S(12\underline{5}\underline{4}3) &= \begin{xy} (0,-3) *+{1} = "1+", (0,3) *+{-1} = "1-", (10,3) *+{-2} = "2-", (20,3) *+{-3} = "3-", (30,3) *+{-4} = "4-", \ar "1-";"2-" \ar "2-";"3-" \ar "4-";"3-" \ar "1+";"2-" \end{xy},
& S(\underline{1}2\underline{5}34) &= \begin{xy} (0,-3) *+{1} = "1+", (0,3) *+{-1} = "1-", (10,3) *+{-2} = "2-", (20,3) *+{-3} = "3-", (30,3) *+{-4} = "4-", \ar "1-";"2-" \ar "2-";"3-" \ar "3-";"4-" \ar "1+";"2-" \end{xy};
\end{align*} }
\item $\sigma=(2,-4,0)$ (2 elements): {\footnotesize \begin{align*}
S(12\underline{4}\underline{3}5) &= \begin{xy} (0,-3) *+{1} = "1+", (0,3) *+{-1} = "1-", (10,3) *+{-2} = "2-", (20,3) *+{-3} = "3-", \ar "1-";"2-" \ar "3-";"2-" \ar "1+";"2-" \end{xy},
& S(\underline{1}2\underline{4}35) &= \begin{xy} (0,-3) *+{1} = "1+", (0,3) *+{-1} = "1-", (10,3) *+{-2} = "2-", (20,3) *+{-3} = "3-", \ar "1-";"2-" \ar "2-";"3-" \ar "1+";"2-" \end{xy};
\end{align*} }
\item $\sigma=(2,-3,0)$ (1 element): {\footnotesize \begin{align*}
S(\underline{1}2\underline{3}45) &= \begin{xy} (0,-3) *+{1} = "1+", (0,3) *+{-1} = "1-", (10,3) *+{-2} = "2-", \ar "1-";"2-" \ar "1+";"2-" \end{xy};
\end{align*} }
\item $\sigma=(2,-1,0)$ (1 element): {\footnotesize \begin{align*}
S(\underline{2}\underline{1}345) &= \begin{xy} (0,0) *+{-1} = "1+", \end{xy};
\end{align*} }
\item $\sigma=(2,1,0)$ (1 element): {\footnotesize \begin{align*}
S(21345) &= \begin{xy} (0,0) *+{1} = "1+", \end{xy};
\end{align*} }
\item $\sigma=(3,-5,0)$ (6 elements): {\footnotesize \begin{align*}
S(123\underline{5}\underline{4}) &= \begin{xy} (0,-3) *+{1} = "1+", (10,-3) *+{2} = "2+", (0,3) *+{-1} = "1-", (10,3) *+{-2} = "2-", (20,3) *+{-3} = "3-", (30,3) *+{-4} = "4-", \ar "2+";"1+" \ar "1-";"2-" \ar "2-";"3-" \ar "4-";"3-" \ar "2+";"1-" \ar "1+";"2-" \end{xy},
& S(\underline{1}23\underline{5}4) &= \begin{xy} (0,-3) *+{1} = "1+", (10,-3) *+{2} = "2+", (0,3) *+{-1} = "1-", (10,3) *+{-2} = "2-", (20,3) *+{-3} = "3-", (30,3) *+{-4} = "4-", \ar "2+";"1+" \ar "1-";"2-" \ar "2-";"3-" \ar "3-";"4-" \ar "2+";"1-" \ar "1+";"2-" \end{xy},
\\ S(\underline{1}3\underline{5}\underline{4}\underline{2}) &= \begin{xy} (0,-3) *+{1} = "1+", (10,-3) *+{2} = "2+", (0,3) *+{-1} = "1-", (10,3) *+{-2} = "2-", (20,3) *+{-3} = "3-", (30,3) *+{-4} = "4-", \ar "2+";"1+" \ar "2-";"1-" \ar "2-";"3-" \ar "4-";"3-" \ar "2+";"1-" \end{xy},
& S(13\underline{5}\underline{2}4) &= \begin{xy} (0,-3) *+{1} = "1+", (10,-3) *+{2} = "2+", (0,3) *+{-1} = "1-", (10,3) *+{-2} = "2-", (20,3) *+{-3} = "3-", (30,3) *+{-4} = "4-", \ar "2+";"1+" \ar "2-";"1-" \ar "2-";"3-" \ar "3-";"4-" \ar "2+";"1-" \end{xy},
\\ S(13\underline{5}\underline{4}2) &= \begin{xy} (0,-3) *+{1} = "1+", (10,-3) *+{2} = "2+", (0,3) *+{-1} = "1-", (10,3) *+{-2} = "2-", (20,3) *+{-3} = "3-", (30,3) *+{-4} = "4-", \ar "1+";"2+" \ar "1-";"2-" \ar "2-";"3-" \ar "4-";"3-" \ar "1+";"2-" \end{xy},
& S(\underline{1}3\underline{5}24) &= \begin{xy} (0,-3) *+{1} = "1+", (10,-3) *+{2} = "2+", (0,3) *+{-1} = "1-", (10,3) *+{-2} = "2-", (20,3) *+{-3} = "3-", (30,3) *+{-4} = "4-", \ar "1+";"2+" \ar "1-";"2-" \ar "2-";"3-" \ar "3-";"4-" \ar "1+";"2-" \end{xy};
\end{align*} }
\item $\sigma=(3,-5,1)$ (4 elements): {\footnotesize \begin{align*}
S(\underline{2}3\underline{5}\underline{4}\underline{1}) &= \begin{xy} (0,-3) *+{-1} = "1+", (10,-3) *+{2} = "2+", (0,3) *+{1} = "1-", (10,3) *+{-2} = "2-", (20,3) *+{-3} = "3-", (30,3) *+{-4} = "4-", \ar "2+";"1+" \ar "1-";"2-" \ar "2-";"3-" \ar "4-";"3-" \ar "1-";"2+" \end{xy},
& S(23\underline{5}\underline{1}4) &= \begin{xy} (0,-3) *+{-1} = "1+", (10,-3) *+{2} = "2+", (0,3) *+{1} = "1-", (10,3) *+{-2} = "2-", (20,3) *+{-3} = "3-", (30,3) *+{-4} = "4-", \ar "2+";"1+" \ar "1-";"2-" \ar "2-";"3-" \ar "3-";"4-" \ar "1-";"2+" \end{xy},
\\ S(23\underline{5}\underline{4}1) &= \begin{xy} (0,-3) *+{1} = "1+", (10,-3) *+{2} = "2+", (0,3) *+{-1} = "1-", (10,3) *+{-2} = "2-", (20,3) *+{-3} = "3-", (30,3) *+{-4} = "4-", \ar "2+";"1+" \ar "1-";"2-" \ar "2-";"3-" \ar "4-";"3-" \ar "1-";"2+" \end{xy},
& S(\underline{2}3\underline{5}14) &= \begin{xy} (0,-3) *+{1} = "1+", (10,-3) *+{2} = "2+", (0,3) *+{-1} = "1-", (10,3) *+{-2} = "2-", (20,3) *+{-3} = "3-", (30,3) *+{-4} = "4-", \ar "2+";"1+" \ar "1-";"2-" \ar "2-";"3-" \ar "3-";"4-" \ar "1-";"2+" \end{xy};
\end{align*} }
\item $\sigma=(3,-4,0)$ (3 elements): {\footnotesize \begin{align*}
S(\underline{1}23\underline{4}5) &= \begin{xy} (0,-3) *+{1} = "1+", (10,-3) *+{2} = "2+", (0,3) *+{-1} = "1-", (10,3) *+{-2} = "2-", (20,3) *+{-3} = "3-", \ar "2+";"1+" \ar "1-";"2-" \ar "2-";"3-" \ar "2+";"1-" \ar "1+";"2-" \end{xy},
& S(13\underline{4}\underline{2}5) &= \begin{xy} (0,-3) *+{1} = "1+", (10,-3) *+{2} = "2+", (0,3) *+{-1} = "1-", (10,3) *+{-2} = "2-", (20,3) *+{-3} = "3-", \ar "2+";"1+" \ar "2-";"1-" \ar "2-";"3-" \ar "2+";"1-" \end{xy},
\\ S(\underline{1}3\underline{4}25) &= \begin{xy} (0,-3) *+{1} = "1+", (10,-3) *+{2} = "2+", (0,3) *+{-1} = "1-", (10,3) *+{-2} = "2-", (20,3) *+{-3} = "3-", \ar "1+";"2+" \ar "1-";"2-" \ar "2-";"3-" \ar "1+";"2-" \end{xy};
\end{align*} }
\item $\sigma=(3,-4,1)$ (2 elements): {\footnotesize \begin{align*}
S(23\underline{4}\underline{1}5) &= \begin{xy} (0,-3) *+{-1} = "1+", (10,-3) *+{2} = "2+", (0,3) *+{1} = "1-", (10,3) *+{-2} = "2-", (20,3) *+{-3} = "3-", \ar "2+";"1+" \ar "1-";"2-" \ar "2-";"3-" \ar "1-";"2+" \end{xy},
& S(\underline{2}3\underline{4}15) &= \begin{xy} (0,-3) *+{1} = "1+", (10,-3) *+{2} = "2+", (0,3) *+{-1} = "1-", (10,3) *+{-2} = "2-", (20,3) *+{-3} = "3-", \ar "2+";"1+" \ar "1-";"2-" \ar "2-";"3-" \ar "1-";"2+" \end{xy};
\end{align*} }
\item $\sigma=(3,-2,0)$ (1 element): {\footnotesize \begin{align*}
S(\underline{1}3\underline{2}45) &= \begin{xy} (0,-3) *+{1} = "1+", (10,-3) *+{2} = "2+", (0,3) *+{-1} = "1-", \ar "2+";"1+" \ar "2+";"1-" \end{xy};
\end{align*} }
\item $\sigma=(3,-2,1)$ (2 elements): {\footnotesize \begin{align*}
S(3\underline{2}\underline{1}45) &= \begin{xy} (0,-3) *+{-1} = "1+", (10,-3) *+{2} = "2+", (0,3) *+{1} = "1-", \ar "2+";"1+" \ar "1-";"2+" \end{xy},
& S(\underline{3}\underline{2}145) &= \begin{xy} (0,-3) *+{1} = "1+", (10,-3) *+{2} = "2+", (0,3) *+{-1} = "1-", \ar "2+";"1+" \ar "1-";"2+" \end{xy};
\end{align*} }
\item $\sigma=(3,-1,0)$ (2 elements): {\footnotesize \begin{align*}
S(\underline{2}3\underline{1}45) &= \begin{xy} (0,0) *+{-1} = "1+", (10,0) *+{2} = "2+", \ar "2+";"1+" \end{xy},
& S(\underline{3}\underline{1}245) &= \begin{xy} (0,0) *+{-1} = "1+", (10,0) *+{2} = "2+", \ar "1+";"2+" \end{xy};
\end{align*} }
\item $\sigma=(3,1,0)$ (2 elements): {\footnotesize \begin{align*}
S(23145) &= \begin{xy} (0,0) *+{1} = "1+", (10,0) *+{2} = "2+", \ar "2+";"1+" \end{xy},
& S(31245) &= \begin{xy} (0,0) *+{1} = "1+", (10,0) *+{2} = "2+", \ar "1+";"2+" \end{xy};
\end{align*} }
\item $\sigma=(3,2,0)$ (1 element): {\footnotesize \begin{align*}
S(13245) &= \begin{xy} (10,0) *+{2} = "2+", \end{xy};
\end{align*} }
\item $\sigma=(4,-5,0)$ (9 elements): {\footnotesize \begin{align*}
S(\underline{1}234\underline{5}) &= \begin{xy} (0,-3) *+{1} = "1+", (10,-3) *+{2} = "2+", (20,-3) *+{3} = "3+", (0,3) *+{-1} = "1-", (10,3) *+{-2} = "2-", (20,3) *+{-3} = "3-", (30,3) *+{-4} = "4-", \ar "2+";"1+" \ar "3+";"2+" \ar "1-";"2-" \ar "2-";"3-" \ar "3-";"4-" \ar "2+";"1-" \ar "1+";"2-" \end{xy},
& S(124\underline{5}\underline{3}) &= \begin{xy} (0,-3) *+{1} = "1+", (10,-3) *+{2} = "2+", (20,-3) *+{3} = "3+", (0,3) *+{-1} = "1-", (10,3) *+{-2} = "2-", (20,3) *+{-3} = "3-", (30,3) *+{-4} = "4-", \ar "2+";"1+" \ar "3+";"2+" \ar "1-";"2-" \ar "3-";"2-" \ar "3-";"4-" \ar "2+";"1-" \ar "1+";"2-" \end{xy},
\\ S(\underline{1}24\underline{5}3) &= \begin{xy} (0,-3) *+{1} = "1+", (10,-3) *+{2} = "2+", (20,-3) *+{3} = "3+", (0,3) *+{-1} = "1-", (10,3) *+{-2} = "2-", (20,3) *+{-3} = "3-", (30,3) *+{-4} = "4-", \ar "2+";"1+" \ar "2+";"3+" \ar "1-";"2-" \ar "2-";"3-" \ar "3-";"4-" \ar "2+";"1-" \ar "1+";"2-" \end{xy},
& S(134\underline{5}\underline{2}) &= \begin{xy} (0,-3) *+{1} = "1+", (10,-3) *+{2} = "2+", (20,-3) *+{3} = "3+", (0,3) *+{-1} = "1-", (10,3) *+{-2} = "2-", (20,3) *+{-3} = "3-", (30,3) *+{-4} = "4-", \ar "2+";"1+" \ar "3+";"2+" \ar "2-";"1-" \ar "2-";"3-" \ar "3-";"4-" \ar "2+";"1-" \end{xy},
\\ S(\underline{1}4\underline{5}\underline{3}\underline{2}) &= \begin{xy} (0,-3) *+{1} = "1+", (10,-3) *+{2} = "2+", (20,-3) *+{3} = "3+", (0,3) *+{-1} = "1-", (10,3) *+{-2} = "2-", (20,3) *+{-3} = "3-", (30,3) *+{-4} = "4-", \ar "2+";"1+" \ar "3+";"2+" \ar "2-";"1-" \ar "3-";"2-" \ar "3-";"4-" \ar "2+";"1-" \end{xy},
& S(14\underline{5}\underline{2}3) &= \begin{xy} (0,-3) *+{1} = "1+", (10,-3) *+{2} = "2+", (20,-3) *+{3} = "3+", (0,3) *+{-1} = "1-", (10,3) *+{-2} = "2-", (20,3) *+{-3} = "3-", (30,3) *+{-4} = "4-", \ar "2+";"1+" \ar "2+";"3+" \ar "2-";"1-" \ar "2-";"3-" \ar "3-";"4-" \ar "2+";"1-" \end{xy},
\\ S(\underline{1}34\underline{5}2) &= \begin{xy} (0,-3) *+{1} = "1+", (10,-3) *+{2} = "2+", (20,-3) *+{3} = "3+", (0,3) *+{-1} = "1-", (10,3) *+{-2} = "2-", (20,3) *+{-3} = "3-", (30,3) *+{-4} = "4-", \ar "1+";"2+" \ar "3+";"2+" \ar "1-";"2-" \ar "2-";"3-" \ar "3-";"4-" \ar "1+";"2-" \end{xy},
& S(14\underline{5}\underline{3}2) &= \begin{xy} (0,-3) *+{1} = "1+", (10,-3) *+{2} = "2+", (20,-3) *+{3} = "3+", (0,3) *+{-1} = "1-", (10,3) *+{-2} = "2-", (20,3) *+{-3} = "3-", (30,3) *+{-4} = "4-", \ar "1+";"2+" \ar "3+";"2+" \ar "1-";"2-" \ar "3-";"2-" \ar "3-";"4-" \ar "1+";"2-" \end{xy},
\\ S(\underline{1}4\underline{5}23) &= \begin{xy} (0,-3) *+{1} = "1+", (10,-3) *+{2} = "2+", (20,-3) *+{3} = "3+", (0,3) *+{-1} = "1-", (10,3) *+{-2} = "2-", (20,3) *+{-3} = "3-", (30,3) *+{-4} = "4-", \ar "1+";"2+" \ar "2+";"3+" \ar "1-";"2-" \ar "2-";"3-" \ar "3-";"4-" \ar "1+";"2-" \end{xy};
\end{align*} }
\item $\sigma=(4,-5,1)$ (6 elements): {\footnotesize \begin{align*}
S(234\underline{5}\underline{1}) &= \begin{xy} (0,-3) *+{-1} = "1+", (10,-3) *+{2} = "2+", (20,-3) *+{3} = "3+", (0,3) *+{1} = "1-", (10,3) *+{-2} = "2-", (20,3) *+{-3} = "3-", (30,3) *+{-4} = "4-", \ar "2+";"1+" \ar "3+";"2+" \ar "1-";"2-" \ar "2-";"3-" \ar "3-";"4-" \ar "1-";"2+" \end{xy},
& S(\underline{2}4\underline{5}\underline{3}\underline{1}) &= \begin{xy} (0,-3) *+{-1} = "1+", (10,-3) *+{2} = "2+", (20,-3) *+{3} = "3+", (0,3) *+{1} = "1-", (10,3) *+{-2} = "2-", (20,3) *+{-3} = "3-", (30,3) *+{-4} = "4-", \ar "2+";"1+" \ar "3+";"2+" \ar "1-";"2-" \ar "3-";"2-" \ar "3-";"4-" \ar "1-";"2+" \end{xy},
\\ S(24\underline{5}\underline{1}3) &= \begin{xy} (0,-3) *+{-1} = "1+", (10,-3) *+{2} = "2+", (20,-3) *+{3} = "3+", (0,3) *+{1} = "1-", (10,3) *+{-2} = "2-", (20,3) *+{-3} = "3-", (30,3) *+{-4} = "4-", \ar "2+";"1+" \ar "2+";"3+" \ar "1-";"2-" \ar "2-";"3-" \ar "3-";"4-" \ar "1-";"2+" \end{xy},
& S(\underline{2}34\underline{5}1) &= \begin{xy} (0,-3) *+{1} = "1+", (10,-3) *+{2} = "2+", (20,-3) *+{3} = "3+", (0,3) *+{-1} = "1-", (10,3) *+{-2} = "2-", (20,3) *+{-3} = "3-", (30,3) *+{-4} = "4-", \ar "2+";"1+" \ar "3+";"2+" \ar "1-";"2-" \ar "2-";"3-" \ar "3-";"4-" \ar "1-";"2+" \end{xy},
\\ S(24\underline{5}\underline{3}1) &= \begin{xy} (0,-3) *+{1} = "1+", (10,-3) *+{2} = "2+", (20,-3) *+{3} = "3+", (0,3) *+{-1} = "1-", (10,3) *+{-2} = "2-", (20,3) *+{-3} = "3-", (30,3) *+{-4} = "4-", \ar "2+";"1+" \ar "3+";"2+" \ar "1-";"2-" \ar "3-";"2-" \ar "3-";"4-" \ar "1-";"2+" \end{xy},
& S(\underline{2}4\underline{5}13) &= \begin{xy} (0,-3) *+{1} = "1+", (10,-3) *+{2} = "2+", (20,-3) *+{3} = "3+", (0,3) *+{-1} = "1-", (10,3) *+{-2} = "2-", (20,3) *+{-3} = "3-", (30,3) *+{-4} = "4-", \ar "2+";"1+" \ar "2+";"3+" \ar "1-";"2-" \ar "2-";"3-" \ar "3-";"4-" \ar "1-";"2+" \end{xy};
\end{align*} }
\item $\sigma=(4,-5,2)$ (4 elements): {\footnotesize \begin{align*}
S(\underline{3}4\underline{5}\underline{2}\underline{1}) &= \begin{xy} (0,-3) *+{-1} = "1+", (10,-3) *+{2} = "2+", (20,-3) *+{3} = "3+", (0,3) *+{1} = "1-", (10,3) *+{-2} = "2-", (20,3) *+{-3} = "3-", (30,3) *+{-4} = "4-", \ar "2+";"1+" \ar "3+";"2+" \ar "2-";"1-" \ar "2-";"3-" \ar "3-";"4-" \ar "1-";"2+" \ar "2-";"3+" \end{xy},
& S(34\underline{5}\underline{1}2) &= \begin{xy} (0,-3) *+{-1} = "1+", (10,-3) *+{2} = "2+", (20,-3) *+{3} = "3+", (0,3) *+{1} = "1-", (10,3) *+{-2} = "2-", (20,3) *+{-3} = "3-", (30,3) *+{-4} = "4-", \ar "1+";"2+" \ar "3+";"2+" \ar "1-";"2-" \ar "2-";"3-" \ar "3-";"4-" \ar "2-";"1+" \ar "2-";"3+" \end{xy},
\\ S(34\underline{5}\underline{2}1) &= \begin{xy} (0,-3) *+{1} = "1+", (10,-3) *+{2} = "2+", (20,-3) *+{3} = "3+", (0,3) *+{-1} = "1-", (10,3) *+{-2} = "2-", (20,3) *+{-3} = "3-", (30,3) *+{-4} = "4-", \ar "2+";"1+" \ar "3+";"2+" \ar "2-";"1-" \ar "2-";"3-" \ar "3-";"4-" \ar "1-";"2+" \ar "2-";"3+" \end{xy},
& S(\underline{3}4\underline{5}12) &= \begin{xy} (0,-3) *+{1} = "1+", (10,-3) *+{2} = "2+", (20,-3) *+{3} = "3+", (0,3) *+{-1} = "1-", (10,3) *+{-2} = "2-", (20,3) *+{-3} = "3-", (30,3) *+{-4} = "4-", \ar "1+";"2+" \ar "3+";"2+" \ar "1-";"2-" \ar "2-";"3-" \ar "3-";"4-" \ar "2-";"1+" \ar "2-";"3+" \end{xy};
\end{align*} }
\item $\sigma=(4,-3,0)$ (3 elements): {\footnotesize \begin{align*}
S(\underline{1}24\underline{3}5) &= \begin{xy} (0,-3) *+{1} = "1+", (10,-3) *+{2} = "2+", (20,-3) *+{3} = "3+", (0,3) *+{-1} = "1-", (10,3) *+{-2} = "2-", \ar "2+";"1+" \ar "3+";"2+" \ar "1-";"2-" \ar "2+";"1-" \ar "1+";"2-" \end{xy},
& S(14\underline{3}\underline{2}5) &= \begin{xy} (0,-3) *+{1} = "1+", (10,-3) *+{2} = "2+", (20,-3) *+{3} = "3+", (0,3) *+{-1} = "1-", (10,3) *+{-2} = "2-", \ar "2+";"1+" \ar "3+";"2+" \ar "2-";"1-" \ar "2+";"1-" \end{xy},
\\ S(\underline{1}4\underline{3}25) &= \begin{xy} (0,-3) *+{1} = "1+", (10,-3) *+{2} = "2+", (20,-3) *+{3} = "3+", (0,3) *+{-1} = "1-", (10,3) *+{-2} = "2-", \ar "1+";"2+" \ar "3+";"2+" \ar "1-";"2-" \ar "1+";"2-" \end{xy};
\end{align*} }
\item $\sigma=(4,-3,1)$ (2 elements): {\footnotesize \begin{align*}
S(24\underline{3}\underline{1}5) &= \begin{xy} (0,-3) *+{-1} = "1+", (10,-3) *+{2} = "2+", (20,-3) *+{3} = "3+", (0,3) *+{1} = "1-", (10,3) *+{-2} = "2-", \ar "2+";"1+" \ar "3+";"2+" \ar "1-";"2-" \ar "1-";"2+" \end{xy},
& S(\underline{2}4\underline{3}15) &= \begin{xy} (0,-3) *+{1} = "1+", (10,-3) *+{2} = "2+", (20,-3) *+{3} = "3+", (0,3) *+{-1} = "1-", (10,3) *+{-2} = "2-", \ar "2+";"1+" \ar "3+";"2+" \ar "1-";"2-" \ar "1-";"2+" \end{xy};
\end{align*} }
\item $\sigma=(4,-3,2)$ (4 elements): {\footnotesize \begin{align*}
S(\underline{4}\underline{3}\underline{2}\underline{1}5) &= \begin{xy} (0,-3) *+{-1} = "1+", (10,-3) *+{2} = "2+", (20,-3) *+{3} = "3+", (0,3) *+{1} = "1-", (10,3) *+{-2} = "2-", \ar "2+";"1+" \ar "3+";"2+" \ar "2-";"1-" \ar "1-";"2+" \ar "2-";"3+" \end{xy},
& S(4\underline{3}\underline{1}25) &= \begin{xy} (0,-3) *+{-1} = "1+", (10,-3) *+{2} = "2+", (20,-3) *+{3} = "3+", (0,3) *+{1} = "1-", (10,3) *+{-2} = "2-", \ar "1+";"2+" \ar "3+";"2+" \ar "1-";"2-" \ar "2-";"1+" \ar "2-";"3+" \end{xy},
\\ S(4\underline{3}\underline{2}15) &= \begin{xy} (0,-3) *+{1} = "1+", (10,-3) *+{2} = "2+", (20,-3) *+{3} = "3+", (0,3) *+{-1} = "1-", (10,3) *+{-2} = "2-", \ar "2+";"1+" \ar "3+";"2+" \ar "2-";"1-" \ar "1-";"2+" \ar "2-";"3+" \end{xy},
& S(\underline{4}\underline{3}125) &= \begin{xy} (0,-3) *+{1} = "1+", (10,-3) *+{2} = "2+", (20,-3) *+{3} = "3+", (0,3) *+{-1} = "1-", (10,3) *+{-2} = "2-", \ar "1+";"2+" \ar "3+";"2+" \ar "1-";"2-" \ar "2-";"1+" \ar "2-";"3+" \end{xy};
\end{align*} }
\item $\sigma=(4,-2,0)$ (2 elements): {\footnotesize \begin{align*}
S(\underline{1}34\underline{2}5) &= \begin{xy} (0,-3) *+{1} = "1+", (10,-3) *+{2} = "2+", (20,-3) *+{3} = "3+", (0,3) *+{-1} = "1-", \ar "2+";"1+" \ar "3+";"2+" \ar "2+";"1-" \end{xy},
& S(\underline{1}4\underline{2}35) &= \begin{xy} (0,-3) *+{1} = "1+", (10,-3) *+{2} = "2+", (20,-3) *+{3} = "3+", (0,3) *+{-1} = "1-", \ar "2+";"1+" \ar "2+";"3+" \ar "2+";"1-" \end{xy};
\end{align*} }
\item $\sigma=(4,-2,1)$ (4 elements): {\footnotesize \begin{align*}
S(34\underline{2}\underline{1}5) &= \begin{xy} (0,-3) *+{-1} = "1+", (10,-3) *+{2} = "2+", (20,-3) *+{3} = "3+", (0,3) *+{1} = "1-", \ar "2+";"1+" \ar "3+";"2+" \ar "1-";"2+" \end{xy},
& S(4\underline{2}\underline{1}35) &= \begin{xy} (0,-3) *+{-1} = "1+", (10,-3) *+{2} = "2+", (20,-3) *+{3} = "3+", (0,3) *+{1} = "1-", \ar "2+";"1+" \ar "2+";"3+" \ar "1-";"2+" \end{xy},
\\ S(\underline{3}4\underline{2}15) &= \begin{xy} (0,-3) *+{1} = "1+", (10,-3) *+{2} = "2+", (20,-3) *+{3} = "3+", (0,3) *+{-1} = "1-", \ar "2+";"1+" \ar "3+";"2+" \ar "1-";"2+" \end{xy},
& S(\underline{4}\underline{2}135) &= \begin{xy} (0,-3) *+{1} = "1+", (10,-3) *+{2} = "2+", (20,-3) *+{3} = "3+", (0,3) *+{-1} = "1-", \ar "2+";"1+" \ar "2+";"3+" \ar "1-";"2+" \end{xy};
\end{align*} }
\item $\sigma=(4,-1,0)$ (4 elements): {\footnotesize \begin{align*}
S(\underline{2}34\underline{1}5) &= \begin{xy} (0,0) *+{-1} = "1+", (10,0) *+{2} = "2+", (20,0) *+{3} = "3+", \ar "2+";"1+" \ar "3+";"2+" \end{xy},
& S(\underline{2}4\underline{1}35) &= \begin{xy} (0,0) *+{-1} = "1+", (10,0) *+{2} = "2+", (20,0) *+{3} = "3+", \ar "2+";"1+" \ar "2+";"3+" \end{xy},
\\ S(\underline{3}4\underline{1}25) &= \begin{xy} (0,0) *+{-1} = "1+", (10,0) *+{2} = "2+", (20,0) *+{3} = "3+", \ar "1+";"2+" \ar "3+";"2+" \end{xy},
& S(\underline{4}\underline{1}235) &= \begin{xy} (0,0) *+{-1} = "1+", (10,0) *+{2} = "2+", (20,0) *+{3} = "3+", \ar "1+";"2+" \ar "2+";"3+" \end{xy};
\end{align*} }
\item $\sigma=(4,1,0)$ (4 elements): {\footnotesize \begin{align*}
S(23415) &= \begin{xy} (0,0) *+{1} = "1+", (10,0) *+{2} = "2+", (20,0) *+{3} = "3+", \ar "2+";"1+" \ar "3+";"2+" \end{xy},
& S(24135) &= \begin{xy} (0,0) *+{1} = "1+", (10,0) *+{2} = "2+", (20,0) *+{3} = "3+", \ar "2+";"1+" \ar "2+";"3+" \end{xy},
\\ S(34125) &= \begin{xy} (0,0) *+{1} = "1+", (10,0) *+{2} = "2+", (20,0) *+{3} = "3+", \ar "1+";"2+" \ar "3+";"2+" \end{xy},
& S(41235) &= \begin{xy} (0,0) *+{1} = "1+", (10,0) *+{2} = "2+", (20,0) *+{3} = "3+", \ar "1+";"2+" \ar "2+";"3+" \end{xy};
\end{align*} }
\item $\sigma=(4,2,0)$ (2 elements): {\footnotesize \begin{align*}
S(13425) &= \begin{xy} (10,0) *+{2} = "2+", (20,0) *+{3} = "3+", \ar "3+";"2+" \end{xy},
& S(14235) &= \begin{xy} (10,0) *+{2} = "2+", (20,0) *+{3} = "3+", \ar "2+";"3+" \end{xy};
\end{align*} }
\item $\sigma=(4,3,0)$ (1 element): {\footnotesize \begin{align*}
S(12435) &= \begin{xy} (20,0) *+{3} = "3+", \end{xy};
\end{align*} }
\item $\sigma=(5,-4,0)$ (9 elements): {\footnotesize \begin{align*}
S(\underline{1}235\underline{4}) &= \begin{xy} (0,-3) *+{1} = "1+", (10,-3) *+{2} = "2+", (20,-3) *+{3} = "3+", (30,-3) *+{4} = "4+", (0,3) *+{-1} = "1-", (10,3) *+{-2} = "2-", (20,3) *+{-3} = "3-", \ar "2+";"1+" \ar "3+";"2+" \ar "4+";"3+" \ar "1-";"2-" \ar "2-";"3-" \ar "2+";"1-" \ar "1+";"2-" \end{xy},
& S(125\underline{4}\underline{3}) &= \begin{xy} (0,-3) *+{1} = "1+", (10,-3) *+{2} = "2+", (20,-3) *+{3} = "3+", (30,-3) *+{4} = "4+", (0,3) *+{-1} = "1-", (10,3) *+{-2} = "2-", (20,3) *+{-3} = "3-", \ar "2+";"1+" \ar "3+";"2+" \ar "4+";"3+" \ar "1-";"2-" \ar "3-";"2-" \ar "2+";"1-" \ar "1+";"2-" \end{xy},
\\ S(\underline{1}25\underline{4}3) &= \begin{xy} (0,-3) *+{1} = "1+", (10,-3) *+{2} = "2+", (20,-3) *+{3} = "3+", (30,-3) *+{4} = "4+", (0,3) *+{-1} = "1-", (10,3) *+{-2} = "2-", (20,3) *+{-3} = "3-", \ar "2+";"1+" \ar "2+";"3+" \ar "4+";"3+" \ar "1-";"2-" \ar "2-";"3-" \ar "2+";"1-" \ar "1+";"2-" \end{xy},
& S(135\underline{4}\underline{2}) &= \begin{xy} (0,-3) *+{1} = "1+", (10,-3) *+{2} = "2+", (20,-3) *+{3} = "3+", (30,-3) *+{4} = "4+", (0,3) *+{-1} = "1-", (10,3) *+{-2} = "2-", (20,3) *+{-3} = "3-", \ar "2+";"1+" \ar "3+";"2+" \ar "4+";"3+" \ar "2-";"1-" \ar "2-";"3-" \ar "2+";"1-" \end{xy},
\\ S(\underline{1}5\underline{4}\underline{3}\underline{2}) &= \begin{xy} (0,-3) *+{1} = "1+", (10,-3) *+{2} = "2+", (20,-3) *+{3} = "3+", (30,-3) *+{4} = "4+", (0,3) *+{-1} = "1-", (10,3) *+{-2} = "2-", (20,3) *+{-3} = "3-", \ar "2+";"1+" \ar "3+";"2+" \ar "4+";"3+" \ar "2-";"1-" \ar "3-";"2-" \ar "2+";"1-" \end{xy},
& S(15\underline{4}\underline{2}3) &= \begin{xy} (0,-3) *+{1} = "1+", (10,-3) *+{2} = "2+", (20,-3) *+{3} = "3+", (30,-3) *+{4} = "4+", (0,3) *+{-1} = "1-", (10,3) *+{-2} = "2-", (20,3) *+{-3} = "3-", \ar "2+";"1+" \ar "2+";"3+" \ar "4+";"3+" \ar "2-";"1-" \ar "2-";"3-" \ar "2+";"1-" \end{xy},
\\ S(\underline{1}35\underline{4}2) &= \begin{xy} (0,-3) *+{1} = "1+", (10,-3) *+{2} = "2+", (20,-3) *+{3} = "3+", (30,-3) *+{4} = "4+", (0,3) *+{-1} = "1-", (10,3) *+{-2} = "2-", (20,3) *+{-3} = "3-", \ar "1+";"2+" \ar "3+";"2+" \ar "4+";"3+" \ar "1-";"2-" \ar "2-";"3-" \ar "1+";"2-" \end{xy},
& S(15\underline{4}\underline{3}2) &= \begin{xy} (0,-3) *+{1} = "1+", (10,-3) *+{2} = "2+", (20,-3) *+{3} = "3+", (30,-3) *+{4} = "4+", (0,3) *+{-1} = "1-", (10,3) *+{-2} = "2-", (20,3) *+{-3} = "3-", \ar "1+";"2+" \ar "3+";"2+" \ar "4+";"3+" \ar "1-";"2-" \ar "3-";"2-" \ar "1+";"2-" \end{xy},
\\ S(\underline{1}5\underline{4}23) &= \begin{xy} (0,-3) *+{1} = "1+", (10,-3) *+{2} = "2+", (20,-3) *+{3} = "3+", (30,-3) *+{4} = "4+", (0,3) *+{-1} = "1-", (10,3) *+{-2} = "2-", (20,3) *+{-3} = "3-", \ar "1+";"2+" \ar "2+";"3+" \ar "4+";"3+" \ar "1-";"2-" \ar "2-";"3-" \ar "1+";"2-" \end{xy};
\end{align*} }
\item $\sigma=(5,-4,1)$ (6 elements): {\footnotesize \begin{align*}
S(235\underline{4}\underline{1}) &= \begin{xy} (0,-3) *+{-1} = "1+", (10,-3) *+{2} = "2+", (20,-3) *+{3} = "3+", (30,-3) *+{4} = "4+", (0,3) *+{1} = "1-", (10,3) *+{-2} = "2-", (20,3) *+{-3} = "3-", \ar "2+";"1+" \ar "3+";"2+" \ar "4+";"3+" \ar "1-";"2-" \ar "2-";"3-" \ar "1-";"2+" \end{xy},
& S(\underline{2}5\underline{4}\underline{3}\underline{1}) &= \begin{xy} (0,-3) *+{-1} = "1+", (10,-3) *+{2} = "2+", (20,-3) *+{3} = "3+", (30,-3) *+{4} = "4+", (0,3) *+{1} = "1-", (10,3) *+{-2} = "2-", (20,3) *+{-3} = "3-", \ar "2+";"1+" \ar "3+";"2+" \ar "4+";"3+" \ar "1-";"2-" \ar "3-";"2-" \ar "1-";"2+" \end{xy},
\\ S(25\underline{4}\underline{1}3) &= \begin{xy} (0,-3) *+{-1} = "1+", (10,-3) *+{2} = "2+", (20,-3) *+{3} = "3+", (30,-3) *+{4} = "4+", (0,3) *+{1} = "1-", (10,3) *+{-2} = "2-", (20,3) *+{-3} = "3-", \ar "2+";"1+" \ar "2+";"3+" \ar "4+";"3+" \ar "1-";"2-" \ar "2-";"3-" \ar "1-";"2+" \end{xy},
& S(\underline{2}35\underline{4}1) &= \begin{xy} (0,-3) *+{1} = "1+", (10,-3) *+{2} = "2+", (20,-3) *+{3} = "3+", (30,-3) *+{4} = "4+", (0,3) *+{-1} = "1-", (10,3) *+{-2} = "2-", (20,3) *+{-3} = "3-", \ar "2+";"1+" \ar "3+";"2+" \ar "4+";"3+" \ar "1-";"2-" \ar "2-";"3-" \ar "1-";"2+" \end{xy},
\\ S(25\underline{4}\underline{3}1) &= \begin{xy} (0,-3) *+{1} = "1+", (10,-3) *+{2} = "2+", (20,-3) *+{3} = "3+", (30,-3) *+{4} = "4+", (0,3) *+{-1} = "1-", (10,3) *+{-2} = "2-", (20,3) *+{-3} = "3-", \ar "2+";"1+" \ar "3+";"2+" \ar "4+";"3+" \ar "1-";"2-" \ar "3-";"2-" \ar "1-";"2+" \end{xy},
& S(\underline{2}5\underline{4}13) &= \begin{xy} (0,-3) *+{1} = "1+", (10,-3) *+{2} = "2+", (20,-3) *+{3} = "3+", (30,-3) *+{4} = "4+", (0,3) *+{-1} = "1-", (10,3) *+{-2} = "2-", (20,3) *+{-3} = "3-", \ar "2+";"1+" \ar "2+";"3+" \ar "4+";"3+" \ar "1-";"2-" \ar "2-";"3-" \ar "1-";"2+" \end{xy};
\end{align*} }
\item $\sigma=(5,-4,2)$ (4 elements): {\footnotesize \begin{align*}
S(\underline{3}5\underline{4}\underline{2}\underline{1}) &= \begin{xy} (0,-3) *+{-1} = "1+", (10,-3) *+{2} = "2+", (20,-3) *+{3} = "3+", (30,-3) *+{4} = "4+", (0,3) *+{1} = "1-", (10,3) *+{-2} = "2-", (20,3) *+{-3} = "3-", \ar "2+";"1+" \ar "3+";"2+" \ar "4+";"3+" \ar "2-";"1-" \ar "2-";"3-" \ar "1-";"2+" \ar "2-";"3+" \end{xy},
& S(35\underline{4}\underline{1}2) &= \begin{xy} (0,-3) *+{-1} = "1+", (10,-3) *+{2} = "2+", (20,-3) *+{3} = "3+", (30,-3) *+{4} = "4+", (0,3) *+{1} = "1-", (10,3) *+{-2} = "2-", (20,3) *+{-3} = "3-", \ar "1+";"2+" \ar "3+";"2+" \ar "4+";"3+" \ar "1-";"2-" \ar "2-";"3-" \ar "2-";"1+" \ar "2-";"3+" \end{xy},
\\ S(35\underline{4}\underline{2}1) &= \begin{xy} (0,-3) *+{1} = "1+", (10,-3) *+{2} = "2+", (20,-3) *+{3} = "3+", (30,-3) *+{4} = "4+", (0,3) *+{-1} = "1-", (10,3) *+{-2} = "2-", (20,3) *+{-3} = "3-", \ar "2+";"1+" \ar "3+";"2+" \ar "4+";"3+" \ar "2-";"1-" \ar "2-";"3-" \ar "1-";"2+" \ar "2-";"3+" \end{xy},
& S(\underline{3}5\underline{4}12) &= \begin{xy} (0,-3) *+{1} = "1+", (10,-3) *+{2} = "2+", (20,-3) *+{3} = "3+", (30,-3) *+{4} = "4+", (0,3) *+{-1} = "1-", (10,3) *+{-2} = "2-", (20,3) *+{-3} = "3-", \ar "1+";"2+" \ar "3+";"2+" \ar "4+";"3+" \ar "1-";"2-" \ar "2-";"3-" \ar "2-";"1+" \ar "2-";"3+" \end{xy};
\end{align*} }
\item $\sigma=(5,-4,3)$ (8 elements): {\footnotesize \begin{align*}
S(5\underline{4}\underline{3}\underline{2}\underline{1}) &= \begin{xy} (0,-3) *+{-1} = "1+", (10,-3) *+{2} = "2+", (20,-3) *+{3} = "3+", (30,-3) *+{4} = "4+", (0,3) *+{1} = "1-", (10,3) *+{-2} = "2-", (20,3) *+{-3} = "3-", \ar "2+";"1+" \ar "3+";"2+" \ar "4+";"3+" \ar "2-";"1-" \ar "3-";"2-" \ar "1-";"2+" \ar "2-";"3+" \ar "3-";"4+" \end{xy},
& S(\underline{5}\underline{4}\underline{2}\underline{1}3) &= \begin{xy} (0,-3) *+{-1} = "1+", (10,-3) *+{2} = "2+", (20,-3) *+{3} = "3+", (30,-3) *+{4} = "4+", (0,3) *+{1} = "1-", (10,3) *+{-2} = "2-", (20,3) *+{-3} = "3-", \ar "2+";"1+" \ar "2+";"3+" \ar "4+";"3+" \ar "2-";"1-" \ar "2-";"3-" \ar "1-";"2+" \ar "3-";"2+" \ar "3-";"4+" \end{xy},
\\ S(\underline{5}\underline{4}\underline{3}\underline{1}2) &= \begin{xy} (0,-3) *+{-1} = "1+", (10,-3) *+{2} = "2+", (20,-3) *+{3} = "3+", (30,-3) *+{4} = "4+", (0,3) *+{1} = "1-", (10,3) *+{-2} = "2-", (20,3) *+{-3} = "3-", \ar "1+";"2+" \ar "3+";"2+" \ar "4+";"3+" \ar "1-";"2-" \ar "3-";"2-" \ar "2-";"1+" \ar "2-";"3+" \ar "3-";"4+" \end{xy},
& S(5\underline{4}\underline{1}23) &= \begin{xy} (0,-3) *+{-1} = "1+", (10,-3) *+{2} = "2+", (20,-3) *+{3} = "3+", (30,-3) *+{4} = "4+", (0,3) *+{1} = "1-", (10,3) *+{-2} = "2-", (20,3) *+{-3} = "3-", \ar "1+";"2+" \ar "2+";"3+" \ar "4+";"3+" \ar "1-";"2-" \ar "2-";"3-" \ar "2-";"1+" \ar "3-";"2+" \ar "3-";"4+" \end{xy},
\\ S(\underline{5}\underline{4}\underline{3}\underline{2}1) &= \begin{xy} (0,-3) *+{1} = "1+", (10,-3) *+{2} = "2+", (20,-3) *+{3} = "3+", (30,-3) *+{4} = "4+", (0,3) *+{-1} = "1-", (10,3) *+{-2} = "2-", (20,3) *+{-3} = "3-", \ar "2+";"1+" \ar "3+";"2+" \ar "4+";"3+" \ar "2-";"1-" \ar "3-";"2-" \ar "1-";"2+" \ar "2-";"3+" \ar "3-";"4+" \end{xy},
& S(5\underline{4}\underline{2}13) &= \begin{xy} (0,-3) *+{1} = "1+", (10,-3) *+{2} = "2+", (20,-3) *+{3} = "3+", (30,-3) *+{4} = "4+", (0,3) *+{-1} = "1-", (10,3) *+{-2} = "2-", (20,3) *+{-3} = "3-", \ar "2+";"1+" \ar "2+";"3+" \ar "4+";"3+" \ar "2-";"1-" \ar "2-";"3-" \ar "1-";"2+" \ar "3-";"2+" \ar "3-";"4+" \end{xy},
\\ S(5\underline{4}\underline{3}12) &= \begin{xy} (0,-3) *+{1} = "1+", (10,-3) *+{2} = "2+", (20,-3) *+{3} = "3+", (30,-3) *+{4} = "4+", (0,3) *+{-1} = "1-", (10,3) *+{-2} = "2-", (20,3) *+{-3} = "3-", \ar "1+";"2+" \ar "3+";"2+" \ar "4+";"3+" \ar "1-";"2-" \ar "3-";"2-" \ar "2-";"1+" \ar "2-";"3+" \ar "3-";"4+" \end{xy},
& S(\underline{5}\underline{4}123) &= \begin{xy} (0,-3) *+{1} = "1+", (10,-3) *+{2} = "2+", (20,-3) *+{3} = "3+", (30,-3) *+{4} = "4+", (0,3) *+{-1} = "1-", (10,3) *+{-2} = "2-", (20,3) *+{-3} = "3-", \ar "1+";"2+" \ar "2+";"3+" \ar "4+";"3+" \ar "1-";"2-" \ar "2-";"3-" \ar "2-";"1+" \ar "3-";"2+" \ar "3-";"4+" \end{xy};
\end{align*} }
\item $\sigma=(5,-3,0)$ (6 elements): {\footnotesize \begin{align*}
S(\underline{1}245\underline{3}) &= \begin{xy} (0,-3) *+{1} = "1+", (10,-3) *+{2} = "2+", (20,-3) *+{3} = "3+", (30,-3) *+{4} = "4+", (0,3) *+{-1} = "1-", (10,3) *+{-2} = "2-", \ar "2+";"1+" \ar "3+";"2+" \ar "4+";"3+" \ar "1-";"2-" \ar "2+";"1-" \ar "1+";"2-" \end{xy},
& S(\underline{1}25\underline{3}4) &= \begin{xy} (0,-3) *+{1} = "1+", (10,-3) *+{2} = "2+", (20,-3) *+{3} = "3+", (30,-3) *+{4} = "4+", (0,3) *+{-1} = "1-", (10,3) *+{-2} = "2-", \ar "2+";"1+" \ar "3+";"2+" \ar "3+";"4+" \ar "1-";"2-" \ar "2+";"1-" \ar "1+";"2-" \end{xy},
\\ S(145\underline{3}\underline{2}) &= \begin{xy} (0,-3) *+{1} = "1+", (10,-3) *+{2} = "2+", (20,-3) *+{3} = "3+", (30,-3) *+{4} = "4+", (0,3) *+{-1} = "1-", (10,3) *+{-2} = "2-", \ar "2+";"1+" \ar "3+";"2+" \ar "4+";"3+" \ar "2-";"1-" \ar "2+";"1-" \end{xy},
& S(15\underline{3}\underline{2}4) &= \begin{xy} (0,-3) *+{1} = "1+", (10,-3) *+{2} = "2+", (20,-3) *+{3} = "3+", (30,-3) *+{4} = "4+", (0,3) *+{-1} = "1-", (10,3) *+{-2} = "2-", \ar "2+";"1+" \ar "3+";"2+" \ar "3+";"4+" \ar "2-";"1-" \ar "2+";"1-" \end{xy},
\\ S(\underline{1}45\underline{3}2) &= \begin{xy} (0,-3) *+{1} = "1+", (10,-3) *+{2} = "2+", (20,-3) *+{3} = "3+", (30,-3) *+{4} = "4+", (0,3) *+{-1} = "1-", (10,3) *+{-2} = "2-", \ar "1+";"2+" \ar "3+";"2+" \ar "4+";"3+" \ar "1-";"2-" \ar "1+";"2-" \end{xy},
& S(\underline{1}5\underline{3}24) &= \begin{xy} (0,-3) *+{1} = "1+", (10,-3) *+{2} = "2+", (20,-3) *+{3} = "3+", (30,-3) *+{4} = "4+", (0,3) *+{-1} = "1-", (10,3) *+{-2} = "2-", \ar "1+";"2+" \ar "3+";"2+" \ar "3+";"4+" \ar "1-";"2-" \ar "1+";"2-" \end{xy};
\end{align*} }
\item $\sigma=(5,-3,1)$ (4 elements): {\footnotesize \begin{align*}
S(245\underline{3}\underline{1}) &= \begin{xy} (0,-3) *+{-1} = "1+", (10,-3) *+{2} = "2+", (20,-3) *+{3} = "3+", (30,-3) *+{4} = "4+", (0,3) *+{1} = "1-", (10,3) *+{-2} = "2-", \ar "2+";"1+" \ar "3+";"2+" \ar "4+";"3+" \ar "1-";"2-" \ar "1-";"2+" \end{xy},
& S(25\underline{3}\underline{1}4) &= \begin{xy} (0,-3) *+{-1} = "1+", (10,-3) *+{2} = "2+", (20,-3) *+{3} = "3+", (30,-3) *+{4} = "4+", (0,3) *+{1} = "1-", (10,3) *+{-2} = "2-", \ar "2+";"1+" \ar "3+";"2+" \ar "3+";"4+" \ar "1-";"2-" \ar "1-";"2+" \end{xy},
\\ S(\underline{2}45\underline{3}1) &= \begin{xy} (0,-3) *+{1} = "1+", (10,-3) *+{2} = "2+", (20,-3) *+{3} = "3+", (30,-3) *+{4} = "4+", (0,3) *+{-1} = "1-", (10,3) *+{-2} = "2-", \ar "2+";"1+" \ar "3+";"2+" \ar "4+";"3+" \ar "1-";"2-" \ar "1-";"2+" \end{xy},
& S(\underline{2}5\underline{3}14) &= \begin{xy} (0,-3) *+{1} = "1+", (10,-3) *+{2} = "2+", (20,-3) *+{3} = "3+", (30,-3) *+{4} = "4+", (0,3) *+{-1} = "1-", (10,3) *+{-2} = "2-", \ar "2+";"1+" \ar "3+";"2+" \ar "3+";"4+" \ar "1-";"2-" \ar "1-";"2+" \end{xy};
\end{align*} }
\item $\sigma=(5,-3,2)$ (8 elements): {\footnotesize \begin{align*}
S(\underline{4}5\underline{3}\underline{2}\underline{1}) &= \begin{xy} (0,-3) *+{-1} = "1+", (10,-3) *+{2} = "2+", (20,-3) *+{3} = "3+", (30,-3) *+{4} = "4+", (0,3) *+{1} = "1-", (10,3) *+{-2} = "2-", \ar "2+";"1+" \ar "3+";"2+" \ar "4+";"3+" \ar "2-";"1-" \ar "1-";"2+" \ar "2-";"3+" \end{xy},
& S(\underline{5}\underline{3}\underline{2}\underline{1}4) &= \begin{xy} (0,-3) *+{-1} = "1+", (10,-3) *+{2} = "2+", (20,-3) *+{3} = "3+", (30,-3) *+{4} = "4+", (0,3) *+{1} = "1-", (10,3) *+{-2} = "2-", \ar "2+";"1+" \ar "3+";"2+" \ar "3+";"4+" \ar "2-";"1-" \ar "1-";"2+" \ar "2-";"3+" \end{xy},
\\ S(45\underline{3}\underline{1}2) &= \begin{xy} (0,-3) *+{-1} = "1+", (10,-3) *+{2} = "2+", (20,-3) *+{3} = "3+", (30,-3) *+{4} = "4+", (0,3) *+{1} = "1-", (10,3) *+{-2} = "2-", \ar "1+";"2+" \ar "3+";"2+" \ar "4+";"3+" \ar "1-";"2-" \ar "2-";"1+" \ar "2-";"3+" \end{xy},
& S(5\underline{3}\underline{1}24) &= \begin{xy} (0,-3) *+{-1} = "1+", (10,-3) *+{2} = "2+", (20,-3) *+{3} = "3+", (30,-3) *+{4} = "4+", (0,3) *+{1} = "1-", (10,3) *+{-2} = "2-", \ar "1+";"2+" \ar "3+";"2+" \ar "3+";"4+" \ar "1-";"2-" \ar "2-";"1+" \ar "2-";"3+" \end{xy},
\\ S(45\underline{3}\underline{2}1) &= \begin{xy} (0,-3) *+{1} = "1+", (10,-3) *+{2} = "2+", (20,-3) *+{3} = "3+", (30,-3) *+{4} = "4+", (0,3) *+{-1} = "1-", (10,3) *+{-2} = "2-", \ar "2+";"1+" \ar "3+";"2+" \ar "4+";"3+" \ar "2-";"1-" \ar "1-";"2+" \ar "2-";"3+" \end{xy},
& S(5\underline{3}\underline{2}14) &= \begin{xy} (0,-3) *+{1} = "1+", (10,-3) *+{2} = "2+", (20,-3) *+{3} = "3+", (30,-3) *+{4} = "4+", (0,3) *+{-1} = "1-", (10,3) *+{-2} = "2-", \ar "2+";"1+" \ar "3+";"2+" \ar "3+";"4+" \ar "2-";"1-" \ar "1-";"2+" \ar "2-";"3+" \end{xy},
\\ S(\underline{4}5\underline{3}12) &= \begin{xy} (0,-3) *+{1} = "1+", (10,-3) *+{2} = "2+", (20,-3) *+{3} = "3+", (30,-3) *+{4} = "4+", (0,3) *+{-1} = "1-", (10,3) *+{-2} = "2-", \ar "1+";"2+" \ar "3+";"2+" \ar "4+";"3+" \ar "1-";"2-" \ar "2-";"1+" \ar "2-";"3+" \end{xy},
& S(\underline{5}\underline{3}124) &= \begin{xy} (0,-3) *+{1} = "1+", (10,-3) *+{2} = "2+", (20,-3) *+{3} = "3+", (30,-3) *+{4} = "4+", (0,3) *+{-1} = "1-", (10,3) *+{-2} = "2-", \ar "1+";"2+" \ar "3+";"2+" \ar "3+";"4+" \ar "1-";"2-" \ar "2-";"1+" \ar "2-";"3+" \end{xy};
\end{align*} }
\item $\sigma=(5,-2,0)$ (4 elements): {\footnotesize \begin{align*}
S(\underline{1}345\underline{2}) &= \begin{xy} (0,-3) *+{1} = "1+", (10,-3) *+{2} = "2+", (20,-3) *+{3} = "3+", (30,-3) *+{4} = "4+", (0,3) *+{-1} = "1-", \ar "2+";"1+" \ar "3+";"2+" \ar "4+";"3+" \ar "2+";"1-" \end{xy},
& S(\underline{1}35\underline{2}4) &= \begin{xy} (0,-3) *+{1} = "1+", (10,-3) *+{2} = "2+", (20,-3) *+{3} = "3+", (30,-3) *+{4} = "4+", (0,3) *+{-1} = "1-", \ar "2+";"1+" \ar "3+";"2+" \ar "3+";"4+" \ar "2+";"1-" \end{xy},
\\ S(\underline{1}45\underline{2}3) &= \begin{xy} (0,-3) *+{1} = "1+", (10,-3) *+{2} = "2+", (20,-3) *+{3} = "3+", (30,-3) *+{4} = "4+", (0,3) *+{-1} = "1-", \ar "2+";"1+" \ar "2+";"3+" \ar "4+";"3+" \ar "2+";"1-" \end{xy},
& S(\underline{1}5\underline{2}34) &= \begin{xy} (0,-3) *+{1} = "1+", (10,-3) *+{2} = "2+", (20,-3) *+{3} = "3+", (30,-3) *+{4} = "4+", (0,3) *+{-1} = "1-", \ar "2+";"1+" \ar "2+";"3+" \ar "3+";"4+" \ar "2+";"1-" \end{xy};
\end{align*} }
\item $\sigma=(5,-2,1)$ (8 elements): {\footnotesize \begin{align*}
S(345\underline{2}\underline{1}) &= \begin{xy} (0,-3) *+{-1} = "1+", (10,-3) *+{2} = "2+", (20,-3) *+{3} = "3+", (30,-3) *+{4} = "4+", (0,3) *+{1} = "1-", \ar "2+";"1+" \ar "3+";"2+" \ar "4+";"3+" \ar "1-";"2+" \end{xy},
& S(35\underline{2}\underline{1}4) &= \begin{xy} (0,-3) *+{-1} = "1+", (10,-3) *+{2} = "2+", (20,-3) *+{3} = "3+", (30,-3) *+{4} = "4+", (0,3) *+{1} = "1-", \ar "2+";"1+" \ar "3+";"2+" \ar "3+";"4+" \ar "1-";"2+" \end{xy},
\\ S(45\underline{2}\underline{1}3) &= \begin{xy} (0,-3) *+{-1} = "1+", (10,-3) *+{2} = "2+", (20,-3) *+{3} = "3+", (30,-3) *+{4} = "4+", (0,3) *+{1} = "1-", \ar "2+";"1+" \ar "2+";"3+" \ar "4+";"3+" \ar "1-";"2+" \end{xy},
& S(5\underline{2}\underline{1}34) &= \begin{xy} (0,-3) *+{-1} = "1+", (10,-3) *+{2} = "2+", (20,-3) *+{3} = "3+", (30,-3) *+{4} = "4+", (0,3) *+{1} = "1-", \ar "2+";"1+" \ar "2+";"3+" \ar "3+";"4+" \ar "1-";"2+" \end{xy},
\\ S(\underline{3}45\underline{2}1) &= \begin{xy} (0,-3) *+{1} = "1+", (10,-3) *+{2} = "2+", (20,-3) *+{3} = "3+", (30,-3) *+{4} = "4+", (0,3) *+{-1} = "1-", \ar "2+";"1+" \ar "3+";"2+" \ar "4+";"3+" \ar "1-";"2+" \end{xy},
& S(\underline{3}5\underline{2}14) &= \begin{xy} (0,-3) *+{1} = "1+", (10,-3) *+{2} = "2+", (20,-3) *+{3} = "3+", (30,-3) *+{4} = "4+", (0,3) *+{-1} = "1-", \ar "2+";"1+" \ar "3+";"2+" \ar "3+";"4+" \ar "1-";"2+" \end{xy},
\\ S(\underline{4}5\underline{2}13) &= \begin{xy} (0,-3) *+{1} = "1+", (10,-3) *+{2} = "2+", (20,-3) *+{3} = "3+", (30,-3) *+{4} = "4+", (0,3) *+{-1} = "1-", \ar "2+";"1+" \ar "2+";"3+" \ar "4+";"3+" \ar "1-";"2+" \end{xy},
& S(\underline{5}\underline{2}134) &= \begin{xy} (0,-3) *+{1} = "1+", (10,-3) *+{2} = "2+", (20,-3) *+{3} = "3+", (30,-3) *+{4} = "4+", (0,3) *+{-1} = "1-", \ar "2+";"1+" \ar "2+";"3+" \ar "3+";"4+" \ar "1-";"2+" \end{xy};
\end{align*} }
\item $\sigma=(5,-1,0)$ (8 elements): {\footnotesize \begin{align*}
S(\underline{2}345\underline{1}) &= \begin{xy} (0,0) *+{-1} = "1+", (10,0) *+{2} = "2+", (20,0) *+{3} = "3+", (30,0) *+{4} = "4+", \ar "2+";"1+" \ar "3+";"2+" \ar "4+";"3+" \end{xy},
& S(\underline{2}35\underline{1}4) &= \begin{xy} (0,0) *+{-1} = "1+", (10,0) *+{2} = "2+", (20,0) *+{3} = "3+", (30,0) *+{4} = "4+", \ar "2+";"1+" \ar "3+";"2+" \ar "3+";"4+" \end{xy},
\\ S(\underline{2}45\underline{1}3) &= \begin{xy} (0,0) *+{-1} = "1+", (10,0) *+{2} = "2+", (20,0) *+{3} = "3+", (30,0) *+{4} = "4+", \ar "2+";"1+" \ar "2+";"3+" \ar "4+";"3+" \end{xy},
& S(\underline{2}5\underline{1}34) &= \begin{xy} (0,0) *+{-1} = "1+", (10,0) *+{2} = "2+", (20,0) *+{3} = "3+", (30,0) *+{4} = "4+", \ar "2+";"1+" \ar "2+";"3+" \ar "3+";"4+" \end{xy},
\\ S(\underline{3}45\underline{1}2) &= \begin{xy} (0,0) *+{-1} = "1+", (10,0) *+{2} = "2+", (20,0) *+{3} = "3+", (30,0) *+{4} = "4+", \ar "1+";"2+" \ar "3+";"2+" \ar "4+";"3+" \end{xy},
& S(\underline{3}5\underline{1}24) &= \begin{xy} (0,0) *+{-1} = "1+", (10,0) *+{2} = "2+", (20,0) *+{3} = "3+", (30,0) *+{4} = "4+", \ar "1+";"2+" \ar "3+";"2+" \ar "3+";"4+" \end{xy},
\\ S(\underline{4}5\underline{1}23) &= \begin{xy} (0,0) *+{-1} = "1+", (10,0) *+{2} = "2+", (20,0) *+{3} = "3+", (30,0) *+{4} = "4+", \ar "1+";"2+" \ar "2+";"3+" \ar "4+";"3+" \end{xy},
& S(\underline{5}\underline{1}234) &= \begin{xy} (0,0) *+{-1} = "1+", (10,0) *+{2} = "2+", (20,0) *+{3} = "3+", (30,0) *+{4} = "4+", \ar "1+";"2+" \ar "2+";"3+" \ar "3+";"4+" \end{xy};
\end{align*} }
\item $\sigma=(5,1,0)$ (8 elements): {\footnotesize \begin{align*}
S(23451) &= \begin{xy} (0,0) *+{1} = "1+", (10,0) *+{2} = "2+", (20,0) *+{3} = "3+", (30,0) *+{4} = "4+", \ar "2+";"1+" \ar "3+";"2+" \ar "4+";"3+" \end{xy},
& S(23514) &= \begin{xy} (0,0) *+{1} = "1+", (10,0) *+{2} = "2+", (20,0) *+{3} = "3+", (30,0) *+{4} = "4+", \ar "2+";"1+" \ar "3+";"2+" \ar "3+";"4+" \end{xy},
\\ S(24513) &= \begin{xy} (0,0) *+{1} = "1+", (10,0) *+{2} = "2+", (20,0) *+{3} = "3+", (30,0) *+{4} = "4+", \ar "2+";"1+" \ar "2+";"3+" \ar "4+";"3+" \end{xy},
& S(25134) &= \begin{xy} (0,0) *+{1} = "1+", (10,0) *+{2} = "2+", (20,0) *+{3} = "3+", (30,0) *+{4} = "4+", \ar "2+";"1+" \ar "2+";"3+" \ar "3+";"4+" \end{xy},
\\ S(34512) &= \begin{xy} (0,0) *+{1} = "1+", (10,0) *+{2} = "2+", (20,0) *+{3} = "3+", (30,0) *+{4} = "4+", \ar "1+";"2+" \ar "3+";"2+" \ar "4+";"3+" \end{xy},
& S(35124) &= \begin{xy} (0,0) *+{1} = "1+", (10,0) *+{2} = "2+", (20,0) *+{3} = "3+", (30,0) *+{4} = "4+", \ar "1+";"2+" \ar "3+";"2+" \ar "3+";"4+" \end{xy},
\\ S(45123) &= \begin{xy} (0,0) *+{1} = "1+", (10,0) *+{2} = "2+", (20,0) *+{3} = "3+", (30,0) *+{4} = "4+", \ar "1+";"2+" \ar "2+";"3+" \ar "4+";"3+" \end{xy},
& S(51234) &= \begin{xy} (0,0) *+{1} = "1+", (10,0) *+{2} = "2+", (20,0) *+{3} = "3+", (30,0) *+{4} = "4+", \ar "1+";"2+" \ar "2+";"3+" \ar "3+";"4+" \end{xy};
\end{align*} }
\item $\sigma=(5,2,0)$ (4 elements): {\footnotesize \begin{align*}
S(13452) &= \begin{xy} (10,0) *+{2} = "2+", (20,0) *+{3} = "3+", (30,0) *+{4} = "4+", \ar "3+";"2+" \ar "4+";"3+" \end{xy},
& S(13524) &= \begin{xy} (10,0) *+{2} = "2+", (20,0) *+{3} = "3+", (30,0) *+{4} = "4+", \ar "3+";"2+" \ar "3+";"4+" \end{xy},
\\ S(14523) &= \begin{xy} (10,0) *+{2} = "2+", (20,0) *+{3} = "3+", (30,0) *+{4} = "4+", \ar "2+";"3+" \ar "4+";"3+" \end{xy},
& S(15234) &= \begin{xy} (10,0) *+{2} = "2+", (20,0) *+{3} = "3+", (30,0) *+{4} = "4+", \ar "2+";"3+" \ar "3+";"4+" \end{xy};
\end{align*} }
\item $\sigma=(5,3,0)$ (2 elements): {\footnotesize \begin{align*}
S(12453) &= \begin{xy} (20,0) *+{3} = "3+", (30,0) *+{4} = "4+", \ar "4+";"3+" \end{xy},
& S(12534) &= \begin{xy} (20,0) *+{3} = "3+", (30,0) *+{4} = "4+", \ar "3+";"4+" \end{xy};
\end{align*} }
\item $\sigma=(5,4,0)$ (1 element): {\footnotesize \begin{align*}
S(12354) &= \begin{xy} (30,0) *+{4} = "4+", \end{xy}.
\end{align*} }

\end{itemize}

\section*{Funding}

The author is a Research Fellow of Japan Society for the Promotion of Science (JSPS).
This work was supported by Japan Society for the Promotion of Science KAKENHI JP16J02249.

\section*{Acknowledgement}

The author thanks to his supervisor Osamu Iyama for kind instructions.
He also thanks to Laurent Demonet, Ryoichi Kase, and 
Yoshihisa Saito for giving me valuable information.

\end{document}